\documentclass[12pt,a4paper,twoside,reqno]{amsart}
\usepackage[utf8]{inputenc}
\usepackage[margin=0.9in]{geometry} 
\usepackage{amsthm}
\usepackage{amsmath}
\usepackage{amsfonts}
\usepackage{amssymb}
\usepackage{chngcntr}
\usepackage{url}
\usepackage{hyperref}
\usepackage[makeroom]{cancel}
\usepackage{enumerate}
\usepackage{mathrsfs}
\newtheorem{Theorem}{Theorem}
\newtheorem{Conjecture}[Theorem]{Conjecture}
\newtheorem{Proposition}[Theorem]{Proposition}
\newtheorem{Corollary}[Theorem]{Corollary}
\newtheorem{Lemma}[Theorem]{Lemma}

\newtheorem{Example}[Theorem]{Example}

\newtheorem{Definition}[Theorem]{Definition}
\newtheorem{Remark}[Theorem]{Remark}

\newcommand{\bb}{\boldsymbol{\beta}}
\newcommand{\ba}{\boldsymbol{\alpha}}
\newcommand{\bc}{\boldsymbol{\chi}}
\newcommand{\bo}{\boldsymbol{\omega}}
\newcommand{\degen}{\operatorname{degen}}

\newcommand{\im}{\operatorname{Im}}
\newcommand{\dist}{\operatorname{dist}}

\newcommand{\vol}{\operatorname{vol}}
\newcommand{\spn}{\operatorname{span}}

\newcommand{\diam}{\operatorname{diam}}
\newcommand{\Rad}{\operatorname{Rad}}
\counterwithin{Theorem}{section}
\numberwithin{equation}{section}

\title{Linear inequalities in primes}

\begin{document}
\author[Aled Walker]{Aled Walker}
\address{Trinity College, Cambridge, CB2 1TQ, United Kingdom}
\email{aw530@cam.ac.uk}
\begin{abstract}
In this paper we prove an asymptotic formula for the number of solutions in prime numbers to systems of simultaneous linear inequalities with algebraic coefficients. For $m$ simultaneous inequalities we require at least $m+2$ variables, improving upon existing methods, which generically require at least $2m+1$ variables. Our result also generalises the theorem of Green-Tao-Ziegler on linear equations in primes. Many of the methods presented apply for arbitrary coefficients, not just for algebraic coefficients, and we formulate a conjecture concerning the pseudorandomness of sieve weights which, if resolved, would remove the algebraicity assumption entirely.  
\end{abstract}
\maketitle
\tableofcontents
\section{Introduction}
\label{section introduction}
Fourier analysis is a vital tool in the study of diophantine problems. In recent years, however, new tools have been developed which can prove asymptotic formulae for the number of solutions to certain systems even when the Fourier-analytic approach is not known to succeed. In particular, in \cite{GT10} Green and Tao established an asymptotic formula for the number of prime solutions to generic systems of $m$ simultaneous linear equations in at least $m+2$ variables. Their result was conditional on various conjectures, but these conjectures were later proved by the same authors and Ziegler, in the series of papers \cite{GT12}, \cite{GTa12} and \cite{GTZ12}. 

\begin{Theorem}[Theorem 1.8, \cite{GT10}, Green-Tao-Ziegler]
\label{Theorem Green Tao}
Let $N$, $m$, and $d$ be natural numbers, with $d\geqslant m+2$, and let $C$ be a positive constant. Let $L = (\lambda_{ij})_{i\leqslant m,j\leqslant d}$ be an $m$-by-$d$ matrix with integer coefficients, with rank $m$, and assume the non-degeneracy condition that the only element of the row-space of $L$ over $\mathbb{Q}$ with two or fewer non-zero entries is the zero vector. Let $\mathbf{b} \in \mathbb{Z}^m$, and suppose that $\Vert \mathbf{b}\Vert_\infty \leqslant CN$ and that $\vert \lambda_{ij}\vert \leqslant C$ for all $i$ and $j$. Let $K \subset [-N,N]^d$ be a convex set. Then
\begin{equation}
\label{with non archimdedean factors}
\vert \{ \mathbf{p} \in K : L\mathbf{p} = \mathbf{b}\}\vert = \Big(\alpha_\infty\prod\limits_{p}\alpha_p\Big) (\log N)^{-d}  + o_{C,d,m}(N^{d-m}(\log N)^{-d}),
\end{equation}
\noindent where the \emph{local densities} $\alpha_p$ are given, for each prime $p$, by 
\begin{equation}
\alpha_p : = \lim\limits_{M \rightarrow \infty} \frac{1}{(2M)^d} \sum\limits_{\substack{\mathbf{n} \in [-M,M]^d \\ L\mathbf{n} = \mathbf{b} \\ (n_i,p) = 1 \, \text{for all } i}} \Big(1 +\frac{1}{p-1}\Big)^d\nonumber
\end{equation}
\noindent and the \emph{global factor} $\alpha_\infty$ is given by 
\begin{equation}
\alpha_\infty : = \vert \{ \mathbf{n} \in \mathbb{Z}^d : \mathbf{n} \in K, L\mathbf{n} = \mathbf{b}, n_i \geqslant 0 \, \text{for all } i\} \vert.\nonumber
\end{equation}
\end{Theorem}
\noindent Here and throughout, $p$ denotes a prime, $\mathbf{p}$ denotes a vector all of whose coordinates are prime, and $\mathbf{n}$ denotes a vector all of whose coordinates are integers $n_i$. The expression $(n_i,p)$ denotes the greatest common divisor of $n_i$ and $p$.

To give a concrete example to which this result may be applied, by considering \[ L = \left(\begin{matrix} 1 & -2 & 1 & 0 \\ 0 & 1 & -2 & 1 \end{matrix} \right) , \qquad \mathbf{b} = \mathbf{0}\] one may deduce an asymptotic formula for the number of four-term arithmetic progressions of primes that are less than $N$.

For $m \geqslant 2$, Theorem \ref{Theorem Green Tao} is stronger than any similar statement that may be proved using the Fourier transform alone. Indeed, notwithstanding Balog's example \cite[Corollary 3]{Ba92} of a certain non-generic class of $m$ equations in $ m + \lceil \sqrt{2m}\rceil $ prime variables, generically the Fourier transform approach needs at least $2m + 1$ prime variables in order to succeed. The proof of Theorem \ref{Theorem Green Tao} rests on many creative innovations, in particular the authors' use of Gowers norms and their inverse theory, which is a subject that is now referred to as `higher order Fourier analysis'. The object of the present paper is to use certain aspects of this machinery to establish, in a related setting, an analogous reduction in the number of variables that are required to prove an asymptotic formula.\\

We will be concerned with diophantine inequalities, a topic that we first considered in \cite{Wa17}. Before giving our first main result (Theorem \ref{Main theorem simpler version}) let us briefly review some previous results concerning diophantine inequalities in the primes. Consider the following classical theorem of Baker.\footnote{In fact Baker proved a slightly different result, writing in the cited paper that the result we quote here followed easily from the then existing methods. Vaughan proved a similar result in \cite{Va74}.}
\begin{Theorem}[\cite{Ba67}, Baker]
\label{Theorem Baker}
Let $\varepsilon>0$, and let $\lambda_1,\lambda_2,\lambda_3\in\mathbb{R}\setminus \{0\}$ be three non-zero reals that are not all of the same sign. Furthermore, suppose that for all $\alpha\in\mathbb{R}\setminus\{0\}$ the relation $(\alpha\lambda_1,\alpha\lambda_2,\alpha\lambda_3)\notin \mathbb{Z}^3$ holds. Then there exist infinitely many triples of primes $(p_1,p_2,p_3)$ satisfying \begin{equation}
\label{the type of inequality Baker thought about}
\vert \lambda_1p_1 + \lambda_2p_2 + \lambda_3 p_3\vert \leqslant \varepsilon.
\end{equation}
\end{Theorem}

\begin{Remark}
\emph{The condition concerning the signs of $\lambda_1,\lambda_2,\lambda_3$ is clearly a necessary one, as otherwise there exist only finitely many solutions to (\ref{the type of inequality Baker thought about}) in the positive integers (and so certainly there exist only finitely many solutions in the primes). Regarding the other condition, the conclusion of Theorem \ref{Theorem Baker} may hold even if there exists some $\alpha \in \mathbb{R}\setminus \{0\}$ for which \[(\alpha\lambda_1,\alpha\lambda_2,\alpha\lambda_3) = (q_1,q_2,q_3) \in \mathbb{Z}^3.\] But then one is required to solve \[ \vert q_1p_1 + q_2p_2 + q_3p_3\vert \leqslant \varepsilon \alpha,\] which, if $\varepsilon$ is small enough, is equivalent to solving \[ q_1p_1 + q_2p_2 + q_3p_3 = 0.\] Theorem \ref{Theorem Green Tao} can then affirm that there are infinitely many solutions, provided that $q_1$, $q_2$, and $q_3$ satisfy certain local properties. This issue, of when an inequality can encode a certain equation with rational coefficients, will be an important theme of the paper.}
\end{Remark}

The classical approach to proving results such as Theorem \ref{Theorem Baker} involves Fourier analysis over $\mathbb{R}$, after having replaced the characteristic function of the interval $[-\varepsilon,\varepsilon]$ with a smoother cut-off function. This approach is known as the Davenport-Heilbronn method, it having originated in a paper \cite{DaHe46} of those two authors.  For a variety of technical reasons this method was, until relatively recently, unable to give an asymptotic formula for the number of solutions to (\ref{the type of inequality Baker thought about}) that satisfied $1\leqslant p_1,p_2,p_3\leqslant N$, or even give a lower bound of the expected order of magnitude (at least for arbitrary $N$). However, certain advances of Freeman \cite{Fr00, Fr02} enabled Parsell to achieve the second of these two goals. 
\begin{Theorem}[Theorem 1, \cite{Pa02}, Parsell] 
\label{Thm Parsel}
Let $\varepsilon>0$, and let $\lambda_1,\lambda_2,\lambda_3\in\mathbb{R}\setminus \{0\}$ be three non-zero reals that are not all of the same sign. Furthermore, suppose that for all $\alpha\in\mathbb{R}\setminus\{0\}$ the relation $(\alpha\lambda_1,\alpha\lambda_2,\alpha\lambda_3)\notin \mathbb{Z}^3$ holds. Then the number of prime triples $(p_1,p_2,p_3)$ satisfying $1\leqslant p_1,p_2,p_3\leqslant N$ and
\begin{equation}
\label{the type of inequality we're talking about}
\vert \lambda_1p_1 + \lambda_2p_2 + \lambda_3 p_3\vert \leqslant \varepsilon
\end{equation}
\noindent is $\Omega_{\lambda_1,\lambda_2,\lambda_3} (\varepsilon N^2 (\log N)^{-3}).$
\end{Theorem}
\noindent Since \cite{Pa02} was published, it has been understood that a very minor modification to Parsell's analytic method can be used to obtain an asymptotic expression for the number of solutions to (\ref{the type of inequality we're talking about}), namely \[C_{\lambda_1,\lambda_2,\lambda_3} \varepsilon N^2 \log^{-3} N + o_{\lambda_1,\lambda_2,\lambda_3}(N^2 \log^{-3} N),\] for some positive constant $C_{\lambda_1,\lambda_2,\lambda_3}$. Furthermore, in the case of $m$ simultaneous (rationally independent) inequalities of the form (\ref{the type of inequality we're talking about}), Parsell's method can calculate an asymptotic formula for the number of solutions in primes provided the number of variables is at least $2m + 1$. In Appendix \ref{section an analytic argument} we take the opportunity to record the details of both the statement and the proof of this result.\\

In the main theorems of this paper (Theorem \ref{Main theorem simpler version} and Theorem \ref{Main theorem}) we specialise to the case of algebraic coefficients and reduce the number of variables that are required from $2m+1$ to $m+2$. Our first result does not concern the most general type of diophantine inequality, but nonetheless it enjoys several applications. To state it, we recall the notion of the \emph{dual degeneracy variety}, which we defined in Definition 2.3 of \cite{Wa17} in order to manipulate the non-degeneracy conditions more succinctly.


\begin{Definition}[Dual degeneracy variety, \cite{Wa17}]
\label{Definition dual degeneracy variety}
Let $m,d$ be natural numbers satisfying $d\geqslant m+2$. Let $V^*_{\degen}(m,d)$ denote the set of all $m$-by-$d$ matrices with real coefficients that contain a non-zero row-vector in their row-space over $\mathbb{R}$ that has two or fewer non-zero co-ordinates. We call $V^*_{\degen}(m,d)$ the \emph{dual degeneracy variety}.
\end{Definition}

\noindent For example, the matrix \[ \left ( \begin{matrix} 1 & -2 & 1 & 0 \\ 2 & -4 & 0 & \sqrt{3} \end{matrix} \right) \] is in $V_{\degen}^*(2,4)$, since the vector $(0,0,-2,\sqrt{3})$ lies in its row space. As is explained at length in \cite{Wa17}, if one wishes to count solutions to an inequality given by $L$ using a method involving Gowers norms then one can only possibly succeed if $L \notin V_{\degen}^*(m,d)$. Returning to Theorem \ref{Theorem Green Tao}, we observe that the non-degeneracy condition in the statement of that theorem is exactly the condition that $L\notin V_{\degen}^*(m,d)$. If $d=m+2$, non-degeneracy in this sense is easy to detect. Indeed, $L \notin V_{\degen}^*(m,d)$ if and only if the determinants of all the $m$-by-$m$ submatrices of $L$ are non-vanishing. 

\begin{Remark}
\label{Remark dual CS}
\emph{The above notion is `dual' to the notion of finite Cauchy-Schwarz complexity (see Definition \ref{Definition finite complexity}), in the sense that $L$ is in the dual degeneracy variety if and only if $\ker L$ may be parametrised by a system of linear forms with finite Cauchy-Schwarz complexity.  In \cite{Wa17} we also introduced a \emph{degeneracy variety} in order to manipulate quantitative versions of this fact, but this will not be necessary here. For more on these issues, we invite the reader to consult Sections 6 and 7 of \cite{Wa17}.}
\end{Remark} 

We are now ready to state our first main result. In the statement below, $[N]$ refers to the set $\mathbb{N} \cap [1,N]$ and the function $1_{[-\varepsilon,\varepsilon]^m}$ refers to the indicator function of the set $[-\varepsilon,\varepsilon]^m$. 

\begin{Theorem}[Main theorem, purely irrational version]
\label{Main theorem simpler version}
Let $N,m,d$ be natural numbers, with $d\geqslant m+2$, and let $C,\varepsilon$ be positive constants. Let $L$ be an $m$-by-$d$ real matrix with algebraic coefficients and rank $m$. Suppose that $L\notin V^*_{\degen}(m,d)$. Suppose further that for all $\ba\in\mathbb{R}^m\setminus \{ \mathbf{0}\}$ one has $L^T\ba \notin \mathbb{Z}^d$, i.e. suppose that $L$ is purely irrational in the sense of Definition 2.4 of \cite{Wa17}. Let $\mathbf{v}\in\mathbb{R}^m$ be any vector satisfying $\Vert\mathbf{v}\Vert_\infty\leqslant CN$. Then
\begin{equation}
\label{simpler asymptotic formula}
\sum\limits_{\mathbf{p} \in [N]^d} 1_{[-\varepsilon,\varepsilon]^m}(L\mathbf{p}+\mathbf{v}) =  \frac{1}{\log^d N}\int\limits_{\mathbf{x} \in [0,N]^d} 1_{[-\varepsilon,\varepsilon]^m}(L\mathbf{x}+\mathbf{v}) \, d\mathbf{x}  + o_{C,L,\varepsilon}(N^{d-m}(\log N)^{-d})
\end{equation}
as $N\rightarrow \infty$.
\end{Theorem}

\begin{Remark}
\emph{One notes that in the asymptotic formula (\ref{simpler asymptotic formula}) there is not a contribution from any non-archimedean local factors. In Theorem \ref{Main theorem} below, we will remove the supposition that there does not exist any non-zero vector $\ba\in\mathbb{R}^m\setminus \{\mathbf{0}\}$ for which $L^T\ba \in \mathbb{Z}^d$. Once these potential rational relations are permitted, one does indeed observe a contribution from local factors.}
\end{Remark}
\begin{Remark}
\label{Remark approximation of integral}
\emph{When $\mathbf{v} = \mathbf{0}$, it is straightforward to show (see Lemma \ref{Lemma singular integral}) that the main term in (\ref{simpler asymptotic formula}) is equal to \[C_L \varepsilon^m N^{d-m} (\log N)^{-d} + o_{L,\varepsilon}(N^{d-m}(\log N)^{-d}),\] where $C_L$ is a constant depending only on $L$. The positivity of $C_L$ may be determined in practice.}
\end{Remark}
\begin{Remark}
\label{Remark fixed L}
\emph{The reader may note that Theorem \ref{Main theorem simpler version} insists upon a fixed matrix $L$, rather than a matrix $L$ with bounded coefficients (as appeared in Theorem \ref{Theorem Green Tao}). In our previous work \cite[Theorem 2.10]{Wa17}, performed in the context of linear inequalities weighted by bounded functions we proved a result that enabled $L$ to vary, as long as the coefficients of $L$ were bounded and $L$ was bounded away from $V_{\degen}^*(m,d)$. In the present paper there are many auxiliary linear equalities $L^\prime$, which will also need to enjoy such a quantitative non-degeneracy. We found keeping track of these features throughout the whole argument to be extremely complicated, but in principle it should be possible to do so.}
\end{Remark}

\begin{Remark}
\emph{Theorem \ref{Main theorem simpler version} strengthens Theorem \ref{Parsell general asymptotic} of Parsell, in the sense that the number of variables has been reduced (from $2m+1$ to $m+2$). But unfortunately this has been achieved at the cost of imposing an algebraicity assumption on the coefficients of $L$. The situation is regrettable as, under this assumption, the classical Davenport-Heilbronn method alone is adequate to count the number of prime solutions to $m$ simultaneous linear inequalities in $2m+1$ variables, without needing the developments of Parsell. We should stress that most of our method does not rely on the algebraicity assumption. Indeed, the conclusions of Theorems \ref{Main theorem simpler version} and \ref{Main theorem} do in fact hold for some explicit set of matrices $L$ that has full Lebesgue measure (see Remark \ref{Remark almost all}). Unfortunately, owing to the intricacy of the linear-algebraic manipulations in Section \ref{section Cauchy Schwarz argument}, we have not been able to formulate a clean or enlightening characterisation of this full-measure set. We have decided to clarify the exposition of the paper by working with algebraic coefficients throughout.}
\end{Remark}

Let us give a concrete example of a linear inequality to which Theorem \ref{Main theorem simpler version} applies but the Davenport-Heilbronn method does not. 
\begin{Example}
\label{Example surd primes}
Let $\varepsilon>0$. Then the number of prime quadruples $(p_1,p_2,p_3,p_4)\in [N]^4$ satisfying 
\begin{align}
\label{solutions in concrete example}
\vert p_1  +  p_3\sqrt{2} - p_4\sqrt{3} \vert &\leqslant \varepsilon\nonumber\\
\vert  p_2+  p_3\sqrt{5} - p_4\sqrt{7} \vert &\leqslant \varepsilon
\end{align}

\noindent is equal to $C\varepsilon^2 N^2 (\log N)^{-4} + o_{\varepsilon}( N^2 (\log N)^{-4} )$, for some positive constant $C$. 
\end{Example}
\begin{proof}
Taking \[ L = \left(\begin{matrix} 1 & 0 & \sqrt{2} & -\sqrt{3} \\ 0 & 1 & \sqrt{5} & -\sqrt{7} \end{matrix}\right),\] $L$ certainly satisfies the hypotheses of Theorem \ref{Main theorem simpler version}, since all the $2$-by-$2$ submatrices have non-zero determinant and surds of primes are rationally independent. Taking $\mathbf{v} = \mathbf{0}$, one may therefore apply Theorem \ref{Main theorem simpler version}. 

This yields an asymptotic expression for the number of solutions to (\ref{solutions in concrete example}) with the main term in the form of an integral. Since $\mathbf{v} = \mathbf{0}$, by Remark \ref{Remark approximation of integral} we may express the main term as $C_L \varepsilon^2 N^2 (\log N)^{-4}$ for some constant $C_L$. Explicitly, from Lemma \ref{Lemma singular integral} and expression (\ref{CL}) therein, \[ \frac{C_L}{4} =  \int\limits_{\substack{0\leqslant x_1,x_2\leqslant 1 \\ x_1 \mathbf{v^{(1)}} + x_2 \mathbf{v^{(2)}} \in [0,1]^2}} 1 \, dx_1 \, dx_2,\] where \[ \mathbf{v^{(1)}} = \left( \begin{matrix} 
-\sqrt{2} \\ -\sqrt{5}\end{matrix} \right), \quad \mathbf{v^{(2)}} = \left( \begin{matrix} \sqrt{3} \\ \sqrt{7} \end{matrix} \right).\] By a computation, we satisfy ourselves that $C_L \approx 1.394...$ is positive.   
\end{proof}
Theorem \ref{Main theorem simpler version} may also be used to count prime solutions to other systems.
\begin{Corollary}
\label{Example irrational szemeredi}
Let $(\theta_1,\dots,\theta_d)^T = \boldsymbol{\theta}\in \mathbb{R}^d$ be a real vector with algebraic coefficients. Suppose that there does not exist any $\mathbf{k}\in\mathbb{Z}^d\setminus \{\mathbf{0}\}$ that satisfies $\mathbf{k}\cdot \boldsymbol{\theta} \in \mathbb{Z}$. Let $\mathcal{P}$ denote the set of primes. Then \begin{equation}
\label{expression for corollary}
\sum\limits_{p_1,p_2\leqslant N} \prod\limits_{j=1}^d 1_{\mathcal{P}\cap [N]}(\lfloor p_1+ p_2\theta_j  \rfloor) = C_{\boldsymbol{\theta}}\frac{N^2}{\log ^d N} + o_{\boldsymbol{\theta}}\Big(\frac{N^2} {\log ^d N}\Big),
\end{equation} for some positive constant $C_{\boldsymbol{\theta}}$. 
\end{Corollary}
\noindent Here $\lfloor x \rfloor$ denotes the floor function of $x$, i.e. the greatest integer that is at most $x$. 
\begin{proof}
We can expand the left-hand side of (\ref{expression for corollary}) as \begin{equation*}
\sum\limits_{p_1,p_2\leqslant N}\sum\limits_{p_3,\dots,p_{d+2} \leqslant N} \prod\limits_{j=3}^{d+2} 1_{[0,1)}(p_1 + p_2\theta_{j-2}  - p_j). 
\end{equation*} Observe that the equation $p_1 + p_2\theta_{j-2} - p_j = 1$ has no solutions, since $\theta_{j-2}$ is irrational by assumption. So the above is equal to \[\sum\limits_{p_1,p_2\leqslant N}\sum\limits_{p_3,\dots,p_{d+2} \leqslant N} \prod\limits_{j=3}^{d+2} 1_{[0,1]}(p_1 + p_2\theta_{j-2}  - p_j),\] and this in turn is equal to 
\begin{equation}
\label{the expression that were going to apply the simple theorem to}
\sum\limits_{p_1,p_2\leqslant N}\sum\limits_{p_3,\dots,p_{d+2} \leqslant N} 1_{[0,1]^d}(p_1\mathbf{1} + p_2 \boldsymbol{\theta} - \mathbf{p_3^{d+2}}),
\end{equation} where $\mathbf{1}\in \mathbb{R}^d$ is the vector with every coordinate equal to $1$, and $\mathbf{p_3^{d+2}} : = (p_3,\dots,p_{d+2})^T$.

Let $L$ be the $d$-by-$(d+2)$ matrix \[L = \left( \begin{matrix}
\mathbf{1} & \boldsymbol{\theta} & -I\end{matrix} \right).\] Then (\ref{the expression that were going to apply the simple theorem to}) is equal to \[ \sum\limits_{\mathbf{p} \in [N]^{d+2}} 1_{[-\frac{1}{2},\frac{1}{2}]^d}(L\mathbf{p} + \mathbf{v}),\] where $\mathbf{v}: = (-1/2,\dots,-1/2)^T$.

One sees that $L$ satisfies the hypotheses of Theorem \ref{Main theorem simpler version}. Indeed, note first that if there exists some $\ba \in \mathbb{R}^d  \setminus \{\mathbf{0}\}$ for which $L^T \ba \in \mathbb{Z}^{d+2}$ then by considering the final $d$ coordinates of $L^T \ba$ it follows such an $\ba$ must have integer coordinates. But by considering the second coordinate of $L^T \ba$ it follows that $\ba \cdot \boldsymbol{\theta} \in \mathbb{Z}$, which is a contradiction to our assumptions on $\boldsymbol{\theta}$. Secondly, if $L$ were in $V_{\degen}^*(d,d+2)$ then either $\theta_i = 0$ for some index $i$, or $\theta_i = \theta_j$ for two different indices $i$ and $j$. Both of these possibilities are precluded by the assumptions on $\boldsymbol{\theta}$. 

Therefore we may apply Theorem \ref{Main theorem simpler version}, and by Remark \ref{Remark approximation of integral} we get a main term of the form $C_{\boldsymbol{\theta}} N^2 (\log N)^{-d}$. Explicitly, using Lemma \ref{Lemma singular integral} and expression (\ref{CL}) as above, we have \[ C_{\boldsymbol{\theta}}= \int\limits_{\substack{0 \leqslant x_1,x_2\leqslant 1 \\ 0 \leqslant x_1 + \theta_i x_2 \leqslant 1 \text{ for all } i}} 1 \, dx_1 \, dx_2.\] For any vector $\boldsymbol{\theta}$ this integral is positive, and so the corollary is proved. 
\end{proof}

Let us now present a theorem which does not require $L$ to be purely irrational. This is Theorem \ref{Main theorem} below, and we consider it to be our main result. \\

For ease of notation, we introduce the following definition.
\begin{Definition}
\label{Definition discrete solution count}
Let $N,m,d$ be natural numbers, and let $L:\mathbb{R}^d \longrightarrow \mathbb{R}^m$ be a linear map. Let $F:\mathbb{R}^d\rightarrow \mathbb{R}$ and $G:\mathbb{R}^m\rightarrow \mathbb{R}$ be functions with compact support. Let $\mathbf{v} \in\mathbb{R}^m$. Then, for functions $f_1,\dots,f_{d}:\mathbb{Z}\longrightarrow \mathbb{R}$, we define
\begin{equation}
\label{definiton of the solution count form}
T^{L,\mathbf{v}}_{F,G,N}(f_1,\dots,f_{d}) := \frac{1}{N^{d-m}}\sum\limits_{\mathbf{n}\in \mathbb{Z}^d}\Big(\prod\limits_{j=1}^{d}f_j(n_j)\Big)F(\mathbf{n}/N)G(L\mathbf{n}+\mathbf{v}).
\end{equation}
\end{Definition}

It will be convenient to introduce a logarithmic weighting to the primes. To this end, following \cite{GT10}, we define the function $\Lambda^\prime : \mathbb{Z} \longrightarrow \mathbb{R}$ by \[ \Lambda^\prime(n) : = \begin{cases}
\log n & n \text{ is prime}\\
0 & \text{otherwise}.
\end{cases}\]
\noindent The von Mangoldt function $\Lambda$ will not be needed in this paper. 

Another notion from \cite{GT10} will be useful.

\begin{Definition}[Local von Mangoldt function]
\label{Definition local von M function}
For $q\geqslant 2$, the \emph{local von Mangoldt function} $\Lambda_{\mathbb{Z}/q\mathbb{Z}}:\mathbb{Z}\longrightarrow \mathbb{R}$ is the $q$-periodic function defined by  $$\Lambda_{\mathbb{Z}/q\mathbb{Z}}(n) = \begin{cases}
\frac{q}{\varphi(q)} & (n,q)=1\\
0 & \text{otherwise.}
\end{cases}$$ We let $\Lambda_{\mathbb{Z}/q\mathbb{Z}}^+: \mathbb{Z} \longrightarrow \mathbb{R}$ denote the restriction of $\Lambda_{\mathbb{Z}/q\mathbb{Z}}$ to the non-negative reals, namely the function $\Lambda_{\mathbb{Z}/q\mathbb{Z}} 1_{[0,\infty)}$.
\end{Definition}
\noindent The local von Mangoldt function, when $q$ is the product of small primes, can be viewed as a model for the function $\Lambda^\prime$. This model\footnote{This is essentially the modified Cram\'{e}r random model.} is intimately connected to a technical device known as the $W$-trick, which we recall in Section \ref{section W trick and Gowers norms}. \\ 

For a function $F:\mathbb{R}^d \longrightarrow \mathbb{R}^m$ we define the Lipschitz constant of $F$ to be \[ \sup\limits_{\mathbf{x},\mathbf{y} \in \mathbb{R}^d} \frac{\Vert F(\mathbf{x}) - F(\mathbf{y})\Vert_\infty}{\Vert \mathbf{x} - \mathbf{y} \Vert_\infty },\] and call $F$ Lipschitz if this value is finite. \\

We may now state the main theorem. 
\begin{Theorem}[Main theorem]
\label{Main theorem}
Let $N,m,d$ be natural numbers with $d\geqslant m+2$, and let $C,\varepsilon,\sigma$ be positive real parameters. Let $L:\mathbb{R}^d\longrightarrow \mathbb{R}^m$ be a surjective linear map with algebraic coefficients, and suppose that $L\notin V_{\degen}^*(m,d)$. Let $\mathbf{v} \in\mathbb{R}^m$ be any vector that satisfies $\Vert \mathbf{v}\Vert_\infty\leqslant CN$.  Let $F:\mathbb{R}^d\longrightarrow [0,1]$ and $G: \mathbb{R}^m\longrightarrow [0,1]$ be compactly supported Lipschitz functions with Lipschitz constants at most $\sigma^{-1}$, and assume that $F$ is supported on $[-1,1]^d$ and $G$ is supported on $[-\varepsilon,\varepsilon]^m$. Let $w = w(N):= \log\log \log N$, assuming that $N$ is large enough for this function to be well defined, and let $W = W(N) := \prod\limits_{p\leqslant w} p$. Then
\begin{equation}
\label{expression from main theorem}
T_{F,G,N}^{L,\mathbf{v}}(\Lambda^\prime,\dots,\Lambda^\prime) = T_{F,G,N}^{L,\mathbf{v}}(\Lambda_{\mathbb{Z}/W\mathbb{Z}}^+,\dots,\Lambda_{\mathbb{Z}/W\mathbb{Z}}^+) + o_{C,L,\varepsilon,\sigma}(1)
\end{equation}
as $N\rightarrow \infty$. 
\end{Theorem} 

\begin{Remark}
\label{Remark after statement of Main theorem}
\emph{If $F$ is supported on $[0,1]^d$, we have \[T_{F,G,N}^{L,\mathbf{v}}(\Lambda_{\mathbb{Z}/W\mathbb{Z}}^+,\dots,\Lambda_{\mathbb{Z}/W\mathbb{Z}}^+) = T_{F,G,N}^{L,\mathbf{v}}(\Lambda_{\mathbb{Z}/W\mathbb{Z}},\dots,\Lambda_{\mathbb{Z}/W\mathbb{Z}}).\] We will prove an asymptotic formula for $T_{F,G,N}^{L,\mathbf{v}}(\Lambda_{\mathbb{Z}/W\mathbb{Z}},\dots,\Lambda_{\mathbb{Z}/W\mathbb{Z}})$ later, in Lemma \ref{Lemma problem for local von Mangoldt} and Remark \ref{Remark asymptotic for local von mangoldt in general}. For example, if $\mathbf{v} = \mathbf{0}$ and \[ L = \left ( \begin{matrix} 1 & -2 & 1 & 0 \\
0 & 1 & - \sqrt{3} & 1 \end{matrix} \right),\] say, and $F$ and $G$ are smooth functions supported on $[0,1]^4$ and $[-1/2,1/2]^2$ respectively, one may use Lemma \ref{Lemma problem for local von Mangoldt} and Remark \ref{Remark asymptotic for local von mangoldt in general} to show that \[ T_{F,G,N}^{L,\mathbf{v}}(\Lambda_{\mathbb{Z}/W\mathbb{Z}},\dots,\Lambda_{\mathbb{Z}/W\mathbb{Z}}) = \mathfrak{S} J + o_{F,G,L}(1),\] where \[ \mathfrak{S} = \frac{1}{2}\prod\limits_{p\geqslant 3} \Big(1-\frac{2}{p}\Big) \Big(\frac{p}{p-1}\Big)^2\] and \[ J = \frac{1}{N^2}\int\limits_{\mathbf{x} \in \mathbb{R}^3}  F(\Xi(\mathbf{x})/N)G(L\Xi(\mathbf{x})) \, d\mathbf{x},\] where \[\Xi(x_1,x_2,x_3) = (x_1,x_1 + x_2,x_1 + 2x_2,x_3).\] The constant $\mathfrak{S}$ is in fact equal to \[ \prod\limits_{p} \frac{1}{p^3}\sum\limits_{x_1,x_2,x_3 \leqslant p} \prod\limits_{j=1}^4 \Lambda_{\mathbb{Z}/p\mathbb{Z}}(\xi_j(x_1,x_2,x_3)),\] where $(\xi_1,\xi_2,\xi_3,\xi_4)$ are the coordinate maps for $\Xi$. }

\emph{It takes some effort to establish precisely what the map $\Xi$ should be for a given $L$. What's more, the asymptotic formula in the general case is not just a product of a local factor and a global factor but rather a finite sum of products of local factors and global factors, and we will need to introduce an abundance of additional notation in order to be able to state these terms properly. Thus, in the interests of readability, we choose not to include this formula as part of the statement of Theorem \ref{Main theorem}.}
\end{Remark}
\begin{Remark}
\emph{If $L$ has rational coefficients\footnote{or more generally if $L$ has rational dimension $m$, see Definition \ref{Definition rational space} below.}, then Theorem \ref{Main theorem} reduces to a statement on linear equations in primes (a reduction which we will make precise in Remark \ref{Remark generalising Green Tao} below). In this sense, our work is a generalisation of Green-Tao-Ziegler.}
\end{Remark}

\begin{Remark}
\emph{We have phrased Theorem \ref{Main theorem} with Lipschitz cut-offs $F$ and $G$. In Section \ref{section removing sharp cut-offs} we will demonstrate how these cut-offs may be removed when $L$ is `purely irrational', and in doing so will demonstrate how Theorem \ref{Main theorem} implies Theorem \ref{Main theorem simpler version}. The same methods may be applied when $L$ is not purely irrational, but they will not always succeed, due to the rational degeneracy introduced in those cases. Unfortunately we have not been able to formulate what we regard to be a satisfactory general condition for saying when (\ref{expression from main theorem}) holds with sharp cut-offs $F$ and $G$. Note in particular how the proof of Lemma \ref{Lemma singular integral} relies heavily on the convex sets $[-\varepsilon,\varepsilon]^m$ and $[0,N]^d$ being axis-parallel boxes. Therefore we do not present a version of the theorem in which summation is over a general convex set $K$, as is done in Theorem \ref{Theorem Green Tao}. However, if the reader wishes to apply a specific instance of Theorem \ref{Main theorem} with sharp cut-offs, the methods of Section \ref{section removing sharp cut-offs} and Appendix \ref{section Easy calculations} will almost certainly suffice for the purpose.}
\end{Remark}

\begin{Remark}
\emph{The reader will observe that, as in Theorem \ref{Main theorem simpler version}, we do not determine the nature of the dependence of the error term in (\ref{expression from main theorem}) on the map $L$. We discussed this feature in Remark \ref{Remark fixed L}.}


\end{Remark}

We conjecture that the conclusion of Theorem \ref{Main theorem} holds for all $L \notin V_{\degen}^*(m,d)$, provided $w$ grows slowly enough in terms of $L$.
 \begin{Conjecture}[Transcendental case]
\label{Conjecture without algebraic}
Let $L$, $\mathbf{v}$, $F$, and $G$ be as in the statement of Theorem \ref{Main theorem}, but do not assume that $L$ necessarily has algebraic coefficients. Then there is some function $w:\mathbb{N} \longrightarrow \mathbb{R}_{\geqslant 0}$, with $w(N) \rightarrow \infty$ as $N\rightarrow \infty$, such that (\ref{expression from main theorem}) holds with $W = \prod\limits_{p\leqslant w}p$. 
\end{Conjecture}

\noindent In Section \ref{section proof of pseduorandomness} we will formulate a statement involving smoothed sieve weights (namely Conjecture \ref{Conjecture pseudorandomness}) which, if resolved, would settle Conjecture \ref{Conjecture without algebraic}. \\

\textbf{Acknowledgments.} During the writing of this paper we benefited greatly from the supervision of Ben Green, and had helpful conversations with Sam Chow, Trevor Wooley, Yufei Zhao, Joni Ter\"{a}v\"{a}inen and Kaisa Matom\"{a}ki. We would like to thank an anonymous referee for an exceptionally detailed reading of the manuscript and for many helpful corrections and comments. The majority of the work was carried out while the author was supported by EPSRC grant no.
EP/M50659X/1, continued while the author was a Program Associate at the Mathematical Sciences Research Institute in Berkeley, and finished while the author was supported by a Junior Research Fellowship at Trinity College Cambridge. \\

\section{The structure of the argument}

In this section we discuss our approach to proving Theorem \ref{Main theorem}, and describe the geography of the paper as a whole. 

Initially, one might hope that Theorem \ref{Main theorem} could be proved by replacing the coefficients of $L$ with some rational approximations, by considering the corresponding linear equation with rational coefficients, and then by appealing directly to Theorem \ref{Theorem Green Tao} on linear equations in primes. However, unless the coefficients of $L$ are extremely well-approximable by rationals (and in particular are transcendental), such an approach does not seem to succeed. Indeed, let $L = (\lambda_{ij})_{i\leqslant m,j\leqslant d}$ and let $\lambda^\prime_{ij}$ be a rational approximation to $\lambda_{ij}$, with $L^\prime$ being the corresponding approximation to $L$. In order for the comparison of $L$ with $L^\prime$ to be meaningful, we will need $\Vert L\mathbf{n} - L^\prime \mathbf{n}\Vert_\infty = O(1)$ for all relevant $\mathbf{n}$, and in the general situation where all coordinates of $\mathbf{n}$ have magnitude $\Omega(N)$ this requires $\vert\lambda^\prime_{ij} - \lambda_{ij}\vert$ to be $O( N^{-1})$. Hence the numerator and denominator of $\lambda^\prime_{ij}$ must grow rapidly with $N$, unless $\lambda_{ij}$ is extremely well-approximable. Yet Theorem \ref{Theorem Green Tao} requires the coefficients of the associated affine linear equations to have height $O(1)$ (excepting the constant term, which may be $O(N)$). In \cite{B17} Bienvenu offers a slight improvement, but even with this refinement it does not seem that we can apply an existing result on linear equations in primes as a black box.

Instead, we will follow a similar approach to that which we used in our work \cite{Wa17}, a paper that considered diophantine inequalities in the setting of bounded functions. Namely, we replace the function $\Lambda^\prime:\mathbb{Z}\rightarrow \mathbb{R}$ by a suitable convolution $ \Lambda^\prime \ast \chi:\mathbb{R}\rightarrow \mathbb{R}$, designed to ensure the validity of the approximation 
\begin{equation}
\label{intro twidles}
\sum\limits_{\mathbf{n}\in \mathbb{Z}^d}\Big(\prod\limits_{j=1}^d \Lambda^\prime(n_j) \Big) F(\mathbf{n}/N) G(L\mathbf{n} + \mathbf{v})\approx \int\limits_{\mathbf{x} \in \mathbb{R}^d}\Big( \prod\limits_{j=1}^d (\Lambda^\prime \ast \chi)(x_j)\Big) F(\mathbf{x}/N) G(L\mathbf{x} + \mathbf{v})\, d\mathbf{x}.
\end{equation} The integral may be manipulated by certain reparametrisations (Lemma \ref{separating out the kernel}), yielding expressions of the form \[ \int\limits_{\mathbf{y} \in \mathbb{R}^{d-m}}\Big(\prod\limits_{j=1}^d g_j(\psi_j(\mathbf{y})) \Big) F(\Psi(\mathbf{y})/N) \, d\mathbf{y},\] where $(\psi_1,\dots,\psi_d) = \Psi:\mathbb{R}^{d-m} \longrightarrow \mathbb{R}^d$ parametrises $\ker L$ and $g_1,\dots,g_d$ are certain functions. By applying the Gowers-Cauchy-Schwarz inequality, in a manner strongly resembling \cite[Appendix C]{GT10}, such expressions may be bounded by the Gowers norm $\Vert\Lambda^\prime - \Lambda_{\mathbb{Z}/W\mathbb{Z}}\Vert_{U^{s+1}[N]}$, for some $s = O(1)$. A qualitative bound on this Gowers norm is known by the work of Green-Tao-Ziegler (see Lemma \ref{Lemma Corollary of tool from Green-Tao}), and so Theorem \ref{Main theorem} follows.\\

The novel aspect of this manipulation, over the work of \cite{GT10} and \cite{Wa17}, is the appearance of various auxiliary linear inequalities, weighted by upper bound sieve weights. These enter in a manner that is somewhat analogous to the way in which the so-called `linear forms condition' arises in \cite{GT10}. Asymptotics for the number of solutions to these auxiliary inequalities underpin the argument, and this leads to a `linear inequalities condition' \[ T_{F,G,N}^{L,\mathbf{v}}( \nu,\dots,\nu) \approx T_{F,G,N}^{L,\mathbf{v}}(\Lambda_{\mathbb{Z}/W\mathbb{Z}},\dots,\Lambda_{\mathbb{Z}/W\mathbb{Z}})\] for a sieve weight $\nu$, which is our corresponding notion of pseudorandomness (made precise in Definition \ref{Definition linear inequalities condition}). We are unable to verify this pseudorandomness condition in full generality, but we succeed in the case when $L$ has algebraic coefficients. Our key technical tool is a bound for the number of solutions to a diophantine inequality restricted to a lattice, which we prove using the Davenport-Heilbronn method. This is the only part of the entire argument that uses the fact that the coefficients are assumed to be algebraic.

There is a final technical manoeuvre that we employ, one which has no direct analogue in \cite{GT10} or \cite{Wa17}. It will transpire that passing to the local von Mangoldt function $\Lambda_{\mathbb{Z}/W\mathbb{Z}}$ introduces certain singular expressions, which arise from the fact that we are dealing with inequalities rather than equations. To circumvent this issue we find it necessary to work at two different `local scales', introducing functions $\Lambda_{\mathbb{Z}/W^*\mathbb{Z}}$ and $\Lambda_{\mathbb{Z}/W\mathbb{Z}}$. By careful manoeuvring one can ensure that the singular expressions are only introduced by the $W^*$ scale, and so, provided $W^*$ grows slowly enough compared to $W$, these singularities may be offset by the decay in the Gowers norm expressions involving $W$. This further complicates the analysis of the expressions, and in fact our final choice of function $W^*$ will be non-effective. \\

The structure of the paper is as follows. The main elements of the proof of Theorem \ref{Main theorem} take place in Part \ref{part gen von neu}, and the reader may wish to begin with this section. It is here that we reduce matters to bounding certain systems by Gowers norms (Section \ref{section Controlling by Gowers norms}), prove the approximation (\ref{intro twidles}) (Section \ref{section transfer}), and apply the Gowers-Cauchy-Schwarz inequality (Section \ref{section Cauchy Schwarz argument}).

However, the arguments of this part rely heavily on lemmas that are proved earlier in the paper, and these lemmas split naturally into four types. There are those results that are standard properties of smooth functions, and these are recorded in Section \ref{section smooth functions}. We also have lemmas whose proofs involve manipulation of a purely linear algebraic nature, in order to reduce inequalities to ones that are `purely irrational' or to put linear equations into `normal form'. We describe these notions in Part \ref{Part linear algebra}. The definition of pseudorandomness for an enveloping sieve weight is contained in Part \ref{part pseudorandomness}, as is our proof that a certain weight satisfies this pseudorandomness condition. Also in this part one may find Conjecture \ref{Conjecture pseudorandomness}, which, if resolved, would remove the algebraicity assumptions. Part \ref{part the structure of inequalities} is reserved for those lemmas that involve the (somewhat tedious) manipulation of integrals into more pleasant forms. One of these lemmas is Lemma \ref{Lemma approximation of Q}, which is the lemma that introduces the second local scale $W^*$ that we mentioned above. 

The first appendix is concerned with elementary estimates relating to the integral that appears in the global factor of Theorem \ref{Main theorem}. As we have already said, Appendix \ref{section an analytic argument} presents a Fourier-analytic argument which is essentially due to Parsell. \\

Finally, let us mention that, to help to streamline the statements of various propositions and lemmas in the paper as a whole, we have found it useful to introduce certain notational conventions that are unique to this paper. We describe these in Section \ref{section conventions}. \\



\part{Preliminaries}
\label{part Preliminaries}

\section{Smooth functions}
\label{section smooth functions}
Smooth functions will play a significant role in the paper, and in this section we collect together those notions and lemmas that will be necessary for our forthcoming manipulations.

Following \cite[Section 2]{HB96}, given a natural number $d$ and a compactly supported smooth function $F:\mathbb{R}^d \longrightarrow \mathbb{R}$, we define $d(F)$ to be the corresponding value of $d$, $\Rad(F)$ to be the smallest $R$ such that $F$ is supported on $[-R,R]^d$, and for every non-negative integer $j$ we define \[ d_j(F): = \max \Big \{ \bigg\vert \frac{\partial^{j_1+\dots + j_d } F}{\partial^{j_1} x_1\dots \partial^{j_d} x_d } \Big|_{\mathbf{x} = \mathbf{a}} \bigg\vert : \mathbf{a} \in \mathbb{R}^d , \, \sum_{i=1}^d j_i = j \Big\}.\]

Then, if $P$ is any set, we shall define $\mathcal{C}(P)$ to be the set of those smooth functions $F$ for which \[ d(F), \operatorname{Rad}(F), d_0(F),d_1(F),d_2(F), \dots\] can be bounded above by quantities that depend only on the elements of $P$. For example, let $g: \mathbb{R} \longrightarrow \mathbb{R}$ be the function given by \[ g(x): = \begin{cases}
\exp(-(1-x^2)^{-1}) & \vert x\vert \leqslant 1 \\
0 & \vert x\vert >1 ,
\end{cases}\] and then for a positive parameter $\delta$ let $g_\delta:\mathbb{R}\longrightarrow \mathbb{R}$ be defined by $g_\delta(x):= g(x/\delta)$. Then $g_\delta \in \mathcal{C}(\delta)$, as is proved rather succinctly in \cite[Lemma 9]{BFI87}, say. 

In order to shorten some of the statements in the main part of the paper, it will be convenient to consider all functions on $\mathbb{R}^0$ to be smooth (with derivatives equal to $0$). \\

Let us record a standard  proposition on smooth majorants and minorants. 
 \begin{Lemma}
\label{Lemma smooth approximations}
Let $\delta$ be a real number in the range $0<\delta < 1$. Then there exist two smooth functions $f^{+\delta}, f^{-\delta}:\mathbb{R}\rightarrow [0,1]$, with $f^{+\delta}, f^{-\delta} \in \mathcal{C}(\delta)$, satisfying \[ 1_{[\delta,1-\delta]}(x)\leqslant f^{-\delta}(x)\leqslant 1_{[0,1]}(x)\leqslant f^{+\delta}(x)\leqslant 1_{[-\delta,1+\delta]}(x)\] for all $x \in \mathbb{R}$. 
\end{Lemma}
\begin{proof}
Let $g$ be as above, and let $C := \int g(x) \, dx$. Then one may define \[ f^{-\delta}(x): = \frac{4}{\delta C}  \int\limits_{y\in\mathbb{R}}1_{[\delta/2,1-\delta/2]}(y) g(4(x-y)/\delta) \, dy\] and \[ f^{+\delta}(x): = \frac{4}{\delta C}  \int\limits_{y\in\mathbb{R}}1_{[-\delta/2,1+\delta/2]}(y) g(4(x-y)/\delta) \, dy.\] The fact that $f^{+\delta},f^{-\delta} \in \mathcal{C}(\delta)$ follows from differentiating under the integral (which is easily justified by the mean value theorem). 
\end{proof}

\begin{Lemma}[Smooth partition of unity]
\label{Lemma smooth partition of unity on interval}
Let $\delta$ be a real number in the range $0< \delta < 1$. Then there exists a natural number $t$, satisfying $t = O(\delta^{-1})$, and functions $f_1,\dots,f_t: \mathbb{R} \longrightarrow [0,1]$ such that
\begin{enumerate}
\item for each $i\leqslant t$, $f_i \in \mathcal{C}(\delta)$;
\item for each $i\leqslant t$, $f_i$ is supported on an interval of length at most $2\delta$; 
\item for all $x \in \mathbb{R}$, $1_{[-1 + \delta, 1 - \delta]}(x)\leqslant \sum\limits_{i=1}^t f_i(x) \leqslant 1_{[-1 - \delta, 1 + \delta]}(x) $;
\item for all $x \in \mathbb{R}$, $x$ is contained in the support of at most $2$ of the functions $f_i$.
\end{enumerate} 
\end{Lemma}
\begin{proof}
Let $t = \lceil 4\delta^{-1} \rceil$, and write \[1_{[-1,-1+ t\delta/2)}=\sum\limits_{i=1}^t 1_{I_i}.\] where \[I_i := [-1 + (i-1)\delta/2, -1 + i\delta/2).\] Then define \[ f_i (x) : = \frac{4}{\delta C} \int\limits_{y\in \mathbb{R}} 1_{I_i}(y) g(4(x-y)/\delta) \, dy.\] The desired properties are immediate.
\end{proof}

\begin{Lemma}[Approximating Lipschitz functions by smooth boxes]
\label{Lemma approximating Lipschitz functions by smooth boxes}
Let $\delta,\sigma,N$ be positive real parameters, with $\delta,\sigma$ in the range $0 < \delta,\sigma < 1/2$. Let $d$ be a natural number, and let $F:\mathbb{R}^d \longrightarrow [0,1]$ be a Lipschitz function supported on $[-N,N]^d$ with Lipschitz constant at most $(\sigma N)^{-1}$. Then there exists a natural number $k$, satisfying $k = O(\delta^{-d})$, and functions $F_1,\dots,F_{k}:\mathbb{R}^d \longrightarrow [0,1]$ such that
\begin{enumerate}
\item $\Vert F - \sum\limits_{i=1}^{k} F_i\Vert_\infty = O(\delta \sigma^{-1})$;
\item for each $i\leqslant k$, $F_i$ is supported on a box with side length $O(\delta)$;
\item there is a natural number $t$, satisfying $t = O(\delta^{-1})$, and functions $f_1,\dots,f_t: \mathbb{R} \longrightarrow [0,1]$, satisfying $f_1,\dots,f_t \in \mathcal{C}(\delta)$, such that \[F_i(\mathbf{x}) = c_{i,F} \prod\limits_{j=1}^d f_{S^{(i)}_{j}}(x_j/2N)\] for each $i\leqslant k$, for some element $S^{(i)} \in [t]^d$ and some constant $c_{i,F} \in [0,1]$. 
\end{enumerate}
\end{Lemma}
\begin{proof}
We have \begin{align}
\label{boxing F}
F(\mathbf{x}) &= F(\mathbf{x})1_{[-1,1]^d}(\mathbf{x}/2N) \nonumber\\
& = F(\mathbf{x}) \Big( \prod\limits_{j=1}^d 1_{[-1,1]}(x_j/2N) \Big) \nonumber\\
& = F(\mathbf{x}) \Big( \prod\limits_{j=1}^d \sum\limits_{i=1}^t f_i(x_j/2N)\Big),
\end{align}
\noindent where the functions $f_1,\dots, f_t$ are those constructed by applying Lemma \ref{Lemma smooth partition of unity on interval} with this value of $\delta$. This manipulation is indeed valid, since $F(\mathbf{x}) = 0$ for any $\mathbf{x}$ for which \[ \prod\limits_{j=1}^d\sum\limits_{i=1}^t f_i(x_j/2N) \neq \prod\limits_{j=1}^d 1_{[-1,1]}(x_j/2N).\]

 Swapping the product and summation, (\ref{boxing F}) equals \[ \sum\limits_{S \in [t]^d} F(\mathbf{x})\Big(\prod\limits_{j=1}^df_{S_j}(x_j/2N)\Big) .\] Let $\mathbf{x}^{(S)}\in \mathbb{R}^d$ be any point at which $\prod\limits_{j=1}^d f_{S_j}(x^{(S)}_j/2N)$ is non-zero. Then the above is equal to
\begin{equation*}
\sum\limits_{S \in [t]^d}(F(\mathbf{x}^{(S)}) + O(\delta \sigma^{-1}))\Big( \prod\limits_{j=1}^d f_{S_j}(x_j/2N)\Big) ,
\end{equation*} by the Lipschitz properties of $F$ and the limited support of the functions $f_1,\dots,f_t$ (which was part (2) of Lemma \ref{Lemma smooth partition of unity on interval}).

Define \[F_S(\mathbf{x}) : = F(\mathbf{x}^{(S)})\Big(\prod\limits_{j=1}^d f_{S_j}(x_j/2N)\Big) .\] These functions satisfy properties (2) and (3) of Lemma \ref{Lemma approximating Lipschitz functions by smooth boxes}. Finally note that, by part (4) of Lemma \ref{Lemma smooth partition of unity on interval}, each $\mathbf{x} \in \mathbb{R}$ is contained in the support of at most $O(1)$ of the functions $F_S$, and hence $\Vert F - \sum\limits_{S \in [t]^d} F_S\Vert_\infty = O(\delta \sigma^{-1})$, as required. 
\end{proof}

The Fourier transform of smooth functions will be an important tool in Section \ref{section Inequalities in arithmetic progressions}. We choose the following convention. If $F:\mathbb{R}^d \longrightarrow \mathbb{R}$ is a compactly supported smooth function, we define the Fourier transform $\widehat{F}:\mathbb{R}^d \longrightarrow \mathbb{C}$ by the formula \[ \widehat{F}(\ba): = \int\limits_{\mathbf{x} \in \mathbb{R}^d} F(\mathbf{x}) e(-\ba \cdot \mathbf{x}) \, d\mathbf{x}.\]

\begin{Lemma}
\label{Lemma by parts}
Let $P$ be a set of parameters and suppose $F \in \mathcal{C}(P)$. Then for every $\ba$ and every non-negative integer $K$ one has \[\vert \widehat{F}(\ba)\vert \ll_{P,K} \Vert 1+ \ba\Vert_\infty^{-K}.\]
\end{Lemma}
\begin{proof}
This follows from integration by parts.
\end{proof}

Finally, we recall the definition of dual lattices and the version of the Poisson summation formula that we will use. 

\begin{Definition}[Dual lattice]
\label{Definition dual lattice}
Let $h$ be a natural number and let $\Gamma \leqslant  \mathbb{R}^h$ be a lattice of rank $h$. Then the dual lattice $\Gamma^*$ is defined by \[ \Gamma^*: = \{ \mathbf{y} \in \mathbb{R}^h: \langle \mathbf{y}, \mathbf{x} \rangle \in \mathbb{Z} \text{ for all } \mathbf{x} \in \Gamma\}.\]
\end{Definition}
\noindent It is easily seen that if $M$ is an $h$-by-$h$ matrix whose columns are a lattice basis for $\Gamma$, then $(M^{-1})^T$ is an $h$-by-$h$ matrix whose columns are a lattice basis for $\Gamma^*$. 
\begin{Lemma}[Poisson summation]
\label{Lemma Poisson summation}
Let $h$ be a natural number and let $\Gamma \leqslant \mathbb{R}^h$ be a lattice of rank $h$. Let $F: \mathbb{R}^h \longrightarrow \mathbb{C}$ be a smooth compactly supported function. Then \[\sum\limits_{\mathbf{x} \in \Gamma} F(\mathbf{x}) = \frac{1}{\vol(\mathbb{R}^h/\Gamma)} \sum\limits_{ \mathbf{y} \in \Gamma^*} \widehat{F}(\mathbf{y}).\]
\end{Lemma} 
\begin{proof}
This is a standard result. The version in which $\Gamma = \mathbb{Z}^h$ appears as \cite[Theorem 3.1.17]{Gr08}, with the extension to general full-rank lattices following from a change of variables. 
\end{proof}

\section{Notation and Conventions}
\label{section conventions}
For the most part the notation used in this paper is very standard, and any usage that could be viewed as somewhat unusual will be introduced as and when it is required. However, there are a few particular points that will apply to the paper as a whole which we believe to be important to address now. \\

 We will use the Bachmann-Landau asymptotic notation $O$, $o$, and $\Omega$, but we do not, as is sometimes the convention, for a function $f$ and a positive function $g$ choose to write $f = O(g)$ if there exists a constant $C$ such that $\vert f(N)\vert \leqslant C g(N) $ for $N$ sufficiently large. Rather we require the inequality to hold for all $N$ in some pre-specified range. If $N$ is a natural number, the range is always assumed to be $\mathbb{N}$ unless otherwise specified. It will be a convenient shorthand to use these symbols in conjunction with minus signs, whenever they appear in exponents. For example, $N^{-\Omega(1)}$ refers to a term $N^{-c}$, where $c$ is some positive quantity bounded away from $0$ as the asymptotic parameter tends to infinity.  

The Vinogradov symbol $\ll$ will be used, where for a function $f$ and a positive function $g$ we write $f\ll g$ if and only if $f = O(g)$. We write $f\asymp g$ if $f\ll g$ and $g\ll f$. If an implied constant or a $o(1)$ term depends on other parameters, we will denote these by subscripts, e.g. $O_{c,C,\varepsilon}(1)$, or $f\asymp_{\varepsilon} g$. However, if the implied constants depend on the underlying dimensions (denoted by $m$, $d$, and occasionally by $h$, $s$, and $t$) we will not record this fact explicitly, as this would render most of the expressions unreadable.

The notation $\Rad(F)$, which was introduced in the previous section for compactly supported smooth functions $F$, will also be used when $F$ is not smooth. \\

In order to keep track of which variables are scalars and which are vectors, we will use boldface $\mathbf{x}$ to denote any $\mathbf{x} \in \mathbb{R}^d$ where $d$ could be at least $2$. In order to describe certain integrals over many variables, the following notational convention will be useful. If $\mathbf{x} \in \mathbb{R}^d$ and if $a$ and $b$ are two subscripts with $1\leqslant a\leqslant b\leqslant d$, we use $\mathbf{x_a^b}$ to denote the vector $(x_a,x_{a+1},\cdots,x_{b})^T \in \mathbb{R}^{b-a+1}$.\\



With a view to trying to shorten some of the statements and proofs to follow, there are certain functions that we will fix throughout the paper, namely $w$, $W$, $\rho$, and $\chi$. From now on, the function $w: \mathbb{N} \longrightarrow \mathbb{R}_{\geqslant 0}$ will always be defined by \[w(N): = \max(1,\log\log\log N).\] Whenever $N$ is a quantity that we have defined, we write $w$ for $w(N)$ and let \[W = W(N) = \prod_{p\leqslant w(N)}p.\] The empty product is considered to be equal to $1$. Whenever other functions $w_1,\dots,w_d, w^*:\mathbb{N}\longrightarrow \mathbb{R}_{\geqslant 0}$ occur, and a natural number $N$ is given, we will define $W_1,\dots,W_d,W^*$ analogously. 

The following definition (a smooth version of \cite[Definition 5.2]{Wa17}) will be a useful way to control certain functions that are required in the argument.  

\begin{Definition}[$\eta$-supported]
\label{Defintion eta supported}
Let $\chi:\mathbb{R}\longrightarrow [0,1]$ be a smooth function, and let $\eta$ be a positive parameter. We say that $\chi$ is \emph{$\eta$-supported} if $\chi$ is supported on $[-\eta,\eta]$ and $\chi(x) \equiv 1$ for all $x\in [-\eta/2,\eta/2]$. 
\end{Definition}
\noindent It follows from Lemma \ref{Lemma smooth approximations} that $1$-supported functions exist. From now on we fix a smooth function \[\rho:\mathbb{R} \longrightarrow [0,1]\] that is $1$-supported. We think of $\rho$ as an element of $\mathcal{C}(\emptyset)$. Whenever a positive parameter $\eta$ is defined we also define \[\chi:\mathbb{R} \longrightarrow [0,1]\] by the relation $\chi(x) : = \rho(x/\eta)$. The function $\chi$ is $\eta$-supported, and satisfies $\chi \in \mathcal{C}(\eta)$. \\

We finish this section with some pieces of notation of a more standard nature. If $X,Y\subset\mathbb{R}^d$ for some $d$, we define \[\operatorname{dist}(X,Y): = \inf\limits_{x\in X, y\in Y} \Vert x - y\Vert_{\infty}.\] If $X$ is the singleton $\{x\}$, we write $\operatorname{dist}(x,Y)$ for $\operatorname{dist}(\{x\},Y)$. We let $\partial(X)$ denote the topological boundary of $X$ (though the symbol $\partial$ will also be used for partial differentiation, as usual). If $A$ and $B$ are two sets with $A\subseteq B$, we let $1_A:B\longrightarrow \{0,1\}$ denote the indicator function of $A$. The relevant set $B$ will usually be obvious from context. If $E$ is some event, e.g. a divisor condition, we will also use $1_E$ for the indicator function of this event.  For $\theta \in \mathbb{R}$ we adopt the standard shorthand $e(\theta)$ to mean $e^{2\pi i \theta}$. The M\"{o}bius function will be denoted by $\mu$, though in Section \ref{section Cauchy Schwarz argument} the symbol $\mu$ will also be used to denote a measure. In Section \ref{section proof of pseduorandomness} we will use $\varphi$ for Euler's $\varphi$-function, and for two natural numbers $a$ and $b$ we use the shorthand $(a,b)$ to denote their greatest common divisor. \\

\part{Linear algebra}
\label{Part linear algebra}
In \cite{Wa17} we developed an armoury of linear-algebraic methods, which enabled us to manipulate linear inequalities into certain desired forms. The same manipulation is necessary here. We have chosen not to consign this material to an appendix, nor simply to cite \cite{Wa17}, since the result of Lemma \ref{Lemma generating a purely irrational map} below will be very important during subsequent sections. We will also need a few results (on the vector $\widetilde{\mathbf{r}}$ below) that were not required in our previous work, and so citing \cite{Wa17} won't quite do. 

Fortunately, as we do not seek to determine exactly how the error term in Theorem \ref{Main theorem} depends on $L$, we can offer a significant simplification over the work that was presented in \cite{Wa17}. This is another reason to include this material. \\

Before starting, we remind the reader of some of the central definitions from the theory of dual vector spaces and dual linear maps, which will be used liberally throughout. Let $V$ be a finite-dimensional vector space over a field $\mathbb{F}$. Then $V^*$ denotes the dual vector space, i.e. the vector space of all linear maps $\omega: V \longrightarrow \mathbb{F}$ under pointwise addition and scalar multiplication. If $L: V \longrightarrow W$ is a linear map between two finite-dimensional vector spaces, the dual map $L^*: W^* \longrightarrow V^*$ is defined by the relation $(L^*(\boldsymbol{\omega}))(\mathbf{v}): = \bo(L(\mathbf{v}))$ for all $\bo\in W^*$ and $\mathbf{v} \in V$. Given a basis $\mathbf{e_1}, \dots, \mathbf{e_n}$ for $V$, the dual basis $\mathbf{e_1^*}, \dots, \mathbf{e_n^*}$ for $V^*$ is defined by extending linearly the relations $$\mathbf{e_i^*}(\mathbf{e_j}) = \begin{cases} 1 & \text{if } i = j \\
0 & \text{otherwise.}\end{cases}$$ Finally, given a set $S \subset V$ the annihilator $S^0 \subset V^*$ is defined by \[S^0 :=\{ \bo \in V^* : \bo(\mathbf{v}) = 0 \text{ for all } \mathbf{v} \in S \}.\] \\

\section{Dimension reduction}
\label{section linear algebra and dimension reduction}
 
We begin with a generalisation of Definition \ref{definiton of the solution count form}. Note that the case $m=0$ is permitted below.

\begin{Definition}
\label{Definition most general discrete solution count form}
Let $N,d,h$ be natural numbers, and let $m$ be a non-negative integer. Let $L:\mathbb{R}^h \longrightarrow \mathbb{R}^m$ be a linear map, and let $(\xi_1,\dots,\xi_d) =\Xi:\mathbb{R}^h \longrightarrow \mathbb{R}^d$ be a linear map with integer coefficients. Let $F:\mathbb{R}^d \longrightarrow \mathbb{R}$ and $G:\mathbb{R}^m \longrightarrow \mathbb{R}$ be functions with compact support. Let $\mathbf{v} \in \mathbb{R}^m$ and $\widetilde{\mathbf{r}} \in \mathbb{Z}^d$. Then for $f_1, \dots, f_d:\mathbb{Z} \longrightarrow \mathbb{R}$ we define 
\begin{equation}
T_{F,G,N}^{L,\mathbf{v} ,\Xi,\widetilde{\mathbf{r}}}(f_1,\dots,f_d): = \frac{1}{N^{h-m}} \sum\limits_{\mathbf{n} \in \mathbb{Z}^h} \Big( \prod\limits_{j=1}^d f_j(\xi_j(\mathbf{n}) + \widetilde{r}_j) \Big) F\Big(\frac{\Xi(\mathbf{n}) + \widetilde{\mathbf{r}}}{N}\Big) G(L\mathbf{n} + \mathbf{v}),
\end{equation}
\noindent where $\widetilde{r}_j$ is the $j^{th}$ coordinate of $\widetilde{\mathbf{r}}$.
\end{Definition} 
\noindent The reader might notice that this definition is subtly different from the similar definition that appeared in \cite{Wa17}, namely Definition 4.3 of that paper, in which the function $\mathbf{n} \mapsto F((\Xi(\mathbf{n})+ \widetilde{\mathbf{r}})/N)$ was treated as an arbitrary function $F_1:\mathbb{R}^h\longrightarrow [0,1]$. When dealing with quantitative aspects of smooth functions (a feature of this paper that is not required in \cite{Wa17}) it is convenient to preserve the internal structure of this particular function, and so we have modified Definition \ref{Definition most general discrete solution count form} accordingly.  \\


Recall the notion of \emph{rational maps} from \cite{Wa17}.

\begin{Definition}[Rational dimension, rational map, purely irrational]
\label{Definition rational space}
Let $m$ and $d$ be natural numbers, with $d\geqslant m$. Let $L:\mathbb{R}^d\longrightarrow \mathbb{R}^m$ be a surjective linear map. Let $u$ denote the largest integer for which there exists a surjective linear map $\Theta:\mathbb{R}^m \longrightarrow \mathbb{R}^u$ for which $\Theta L (\mathbb{Z}^d) \subseteq \mathbb{Z}^u$. We call $u$ the \emph{rational dimension} of $L$, and we call any map $\Theta$ with the above property a \emph{rational map} for $L$. We say that $L$ is \emph{purely irrational} if $u=0$. 
\end{Definition}
\begin{Remark}
\label{Remark algebraic coefficients of rational map}
\emph{If (the matrix of) $L$ has algebraic coefficients, then there exists a rational map for $L$ that also has algebraic coefficients.}
\end{Remark}

Purely irrational linear maps are those that we may analyse most easily using the Davenport-Heilbronn method (see Section \ref{section Inequalities in arithmetic progressions}). However, even when proving Theorem \ref{Main theorem simpler version}, whose statement concerns only purely irrational linear maps, we will be forced to consider auxiliary linear maps that are not purely irrational. It is necessary  therefore to develop a rudimentary theory of these maps. Readers desiring more detail and motivating examples concerning rational maps and rational dimension may consult Sections 2, 4, and 6 of \cite{Wa17}. \\

Our key tool will be Lemma \ref{Lemma generating a purely irrational map}, which is a version of Lemma 4.10 from \cite{Wa17}. This lemma will enable us to `quotient out' the rational relations that are present in a diophantine inequality, leaving behind a purely irrational linear map between spaces of a lower dimension. In particular, we will show that \[T_{F,G,N}^{L,\mathbf{v}}(f_1,\dots,f_d)=\sum\limits_{\widetilde{\mathbf{r}} \in \widetilde{R}}T_{F,G_{\widetilde{\mathbf{r}}},N}^{L^\prime,\mathbf{v}^\prime,\Xi,\widetilde{\mathbf{r}}}(f_1,\dots,f_d),\] where $L^\prime$ is purely irrational, and the vectors $\mathbf{v^\prime}$ and $\widetilde{\mathbf{r}}$, the linear map $\Xi$ and the function $G_{\widetilde{\mathbf{r}}}$ are objects that we may control.

To state the lemma we need to recall explicitly the notion from \cite{GT10} that was mentioned in Remark \ref{Remark dual CS}, namely \emph{finite Cauchy-Schwarz complexity} for linear maps.\footnote{In \cite{Wa17} a notion of degeneracy for pairs of linear maps was useful, but we have structured the present paper in such a way as to avoid requiring this complicated notion.}
\begin{Definition}
[Finite Cauchy-Schwarz complexity]
Let $d,h$ be natural numbers, and let $(\xi_1,\dots,\xi_d): = \Xi:\mathbb{R}^h \longrightarrow \mathbb{R}^d$ be a linear map. We say that $\Xi$ has \emph{infinite Cauchy-Schwarz complexity} if there are two distinct indices $i$ and $j$, and some $\lambda \in \mathbb{R}$, for which $\xi_i = \lambda\xi_j$. If no such $i$ and $j$ exist we say that $\Xi$ has \emph{finite Cauchy-Schwarz complexity}.
\end{Definition}

There is an equivalent definition, which will be more convenient for algebraic manipulations. 
\begin{Definition}[Finite Cauchy-Schwarz complexity, equivalent definition]
\label{Definition finite complexity}
Let $d,h$ be natural numbers. Let $\mathbf{e_1}, \dots,\mathbf{e_d}$ denote the standard basis vectors of $\mathbb{R}^d$, and let $\mathbf{e_1}^\ast, \dots,\mathbf{e_d}^\ast$ denote the dual basis of $(\mathbb{R}^d)^\ast$. Then let $V_{\degen}(h,d)$ denote the set of all linear maps $\Xi:\mathbb{R}^h \longrightarrow \mathbb{R}^d$ for which there exist two indices $i,j \leqslant d$, and some real number $\lambda$, such $\mathbf{e_i} - \lambda \mathbf{e_j}$ is non-zero and $\mathbf{e_i}^* - \lambda \mathbf{e_j}^* \in \ker (\Xi^*)$. If $\Xi \notin V_{\degen}(h,d)$, we say that $\Xi$ has \emph{finite Cauchy-Schwarz complexity}. 
\end{Definition}
\noindent  The equivalence of these definitions is elementary.

For more background on the notion of finite Cauchy-Schwarz complexity, the reader may consult Section 1 of \cite{GT10} or Section 6 of \cite{Wa17}.\\ 

Now we may state and prove the important lemma, which provides the `dimension reduction' of the section title. 
\begin{Lemma}[Generating a purely irrational map]
\label{Lemma generating a purely irrational map}
Let $m,d$ be natural numbers, with $d \geqslant m+2$, and let $C,\eta$ be positive parameters. Let $L: \mathbb{R}^d \longrightarrow \mathbb{R}^m$ be a surjective linear map with algebraic coefficients. Let $u$ be the rational dimension of $L$. Let $F:\mathbb{R}^d \longrightarrow [0,1]$ and $G: \mathbb{R}^m \longrightarrow [0,1]$ be compactly supported functions. Assume that $G$ is smooth, $\Rad(G) \leqslant \eta$, and moreover that $G \in \mathcal{C}(P,\eta)$ for some set of parameters $P$. Let $\mathbf{v} \in\mathbb{R}^m$ be a vector with $\Vert \mathbf{v}\Vert_\infty \leqslant CN$. Then there exists a surjective linear map $\Theta:\mathbb{R}^m \longrightarrow \mathbb{R}^u$, a surjective linear map $L^\prime:\mathbb{R}^{d-u} \longrightarrow \mathbb{R}^{m-u}$, an injective linear map $\Xi:\mathbb{R}^{d-u} \longrightarrow \mathbb{R}^d$, a finite subset $\widetilde{R}\subset \mathbb{Z}^d$, a vector $\mathbf{v}^\prime \in \mathbb{R}^{m-u}$, and, for each $\widetilde{\mathbf{r}} \in \widetilde{R}$, a compactly supported function $G_{\widetilde{\mathbf{r}}}:\mathbb{R}^{m-u} \longrightarrow [0,1]$,  such that
\begin{enumerate}[(1)]
\item $\Theta$ is a rational map for $L$ with algebraic coefficients;
\item $\Xi$ has integer coefficients, depends only on $L$, and satisfies $\im \Xi = \ker \Theta L $ and $\Xi(\mathbb{Z}^{d-u}) = \mathbb{Z}^d \cap \im \Xi$;
\item $\widetilde{R}$ satisfies $\vert \widetilde{R}\vert = O_{L,\eta}(1)$, and $\Vert \widetilde{\mathbf{r}}\Vert_\infty = O_{C,L,\eta}(N)$ for all $\widetilde{\mathbf{r}} \in \widetilde{R}$;
\item for all $\widetilde{\mathbf{r}} \in\widetilde{R}$, the function $G_{\widetilde{\mathbf{r}}}$ is smooth, $\Rad(G) = O_L(\eta)$, and $G_{\widetilde{\mathbf{r}}} \in \mathcal{C}(L,P,\eta)$;
\item $\mathbf{v}^\prime$ satisfies $\Vert \mathbf{v}^\prime\Vert_\infty = O_{C,L}(N)$;
\item for all natural numbers $N$, and for all functions $f_1,\dots,f_d: \mathbb{Z} \longrightarrow \mathbb{R}$, one has \[T_{F,G,N}^{L,\mathbf{v}}(f_1,\dots,f_d)=\sum\limits_{\widetilde{\mathbf{r}} \in \widetilde{R}}T_{F,G_{\widetilde{\mathbf{r}}},N}^{L^\prime,\mathbf{v}^\prime,\Xi,\widetilde{\mathbf{r}}}(f_1,\dots,f_d);\]
\item $L^\prime$ is purely irrational, depends only on $L$, and has algebraic coefficients; 
\item if $L \notin V_{\degen}^*(m,d)$ then $\Xi$ has finite Cauchy-Schwarz complexity. \\

\noindent The above properties suffice for Section \ref{section proof of pseduorandomness}, but three additional properties also hold. We will need these additional properties in Section \ref{section structure of Q}. \\
\item Letting $\mathbf{e_1}, \dots,\mathbf{e_{d-u}}$ denote the standard basis of $\mathbb{R}^{d-u}$, there is a set $\{ \mathbf{x_i}: i \leqslant u\} \subset \mathbb{R}^d$ for which \begin{equation}
 \mathcal{B}: = \{\mathbf{x_i}:i\leqslant u\} \cup \{ \Xi(\mathbf{e_j}):j\leqslant d-u\}
 \end{equation} is a basis for $\mathbb{R}^d$ and a lattice basis for $\mathbb{Z}^d$. Furthermore, $\widetilde{R} \subset \spn(\mathbf{x_i}:i\leqslant u)$ and $\{\Theta L \mathbf{x_i}: i\leqslant u\}$ is a lattice basis for $\Theta L \mathbb{Z}^d$; 
\item if $\eta$ is small enough in terms of $L$, and if $\mathbf{v} = L\mathbf{a}$ for some $\mathbf{a} \in \mathbb{R}^d$, then $\vert \widetilde{R}\vert = 1$ and $\widetilde{\mathbf{r}} \in R$ is a vector that minimises $\Vert \Theta L (\widetilde{\mathbf{r}} + \mathbf{a})\Vert_\infty$ over all $\widetilde{\mathbf{r}} \in \mathbb{Z}^d$;
\item for all $\widetilde{\mathbf{r}} \in \widetilde{R}$ and $\mathbf{x} \in \mathbb{R}^{d-u}$ one has \[ G_{\widetilde{\mathbf{r}}}(L^\prime \mathbf{x} + \mathbf{v}^\prime) = G(L\Xi (\mathbf{x}) + L\widetilde{\mathbf{r}} +\mathbf{v}).\].
\end{enumerate}
\end{Lemma}


\begin{proof}
\textbf{Parts (1) and (2)}: Choose $\Theta:\mathbb{R}^m \longrightarrow \mathbb{R}^u$ to be a rational map for $L$ that has algebraic coefficients. By rank-nullity $\ker (\Theta L)$ is a $d-u$ dimensional subspace of $\mathbb{R}^d$, and also the matrix of $\Theta L$ has integer coefficients. Combining these two facts, we see that $\ker (\Theta L) \cap \mathbb{Z}^d$ is a $d-u$ dimensional lattice, and (by the standard algorithms) one can find a lattice basis $\mathbf{v_1},\dots, \mathbf{v_{d-u}} \in \mathbb{Z}^d$ that satisfies $\Vert \mathbf{v_i}\Vert_\infty = O_{L}(1)$ for every $i$. 

Let $\mathbf{e_1},\dots,\mathbf{e_{d-u}}$ denote the standard basis of $\mathbb{R}^{d-u}$, and then define $\Xi:\mathbb{R}^{d-u} \longrightarrow \mathbb{R}^{d}$ by \[ \Xi(\mathbf{e_i}):= \mathbf{v_i}\] for all $i\leqslant d-u$. Then $\Xi$ satisfies part (2) of the lemma. \\

\textbf{Parts (3), (9), and (10)}: There is a set of vectors $\{ \mathbf{a_1},\dots,\mathbf{a_u}\} \subset \mathbb{Z}^u$ that is an integer basis for the lattice $\Theta L(\mathbb{Z}^d)$ and for which $\Vert \mathbf{a_i}\Vert_\infty = O_{L}(1)$ for each $i$. Furthermore there exists a set of vectors $\{ \mathbf{x_1},\dots,\mathbf{x_u}\} \subset \mathbb{Z}^d$ such that $\Theta L(\mathbf{x_i})= \mathbf{a_i}$ for each $i$, and $\Vert \mathbf{x_i}\Vert_\infty = O_{L}(1)$. By Lemma 4.8 of \cite{Wa17},
\begin{equation}
\label{basis for Rd}
 \mathcal{B}: = \{\mathbf{x_i}:i\leqslant u\} \cup \{ \Xi(\mathbf{e_j}):j\leqslant d-u\}
 \end{equation} is a basis for $\mathbb{R}^d$ and a lattice basis for $\mathbb{Z}^d$. 

Now, if $\mathbf{z} \in\mathbb{R}^m$ and $\Theta(\mathbf{z}) = \mathbf{r}$ then  $ \Vert \mathbf{z} \Vert_\infty =\Omega_L(\Vert\mathbf{r} \Vert_\infty)$. Recall that $\Rad(G) \leqslant \eta$ and that $\Theta L(\mathbb{Z}^d) \subseteq \mathbb{Z}^u$. It follows that there are at most $O_{L,\eta}(1)$ possible vectors $\mathbf{r} \in \mathbb{Z}^u$ for which there exists a vector $\mathbf{n} \in\mathbb{Z}^d$ for which both $G(L\mathbf{n} + \mathbf{v}) \neq 0$ and $\Theta L \mathbf{n} = \mathbf{r}$. Let $R$ denote the set of all such vectors $\mathbf{r}$. Observe that, for all $\mathbf{r} \in R$, $\Vert \mathbf{r}\Vert_\infty  = O_{C,L,\eta}(N)$. 

For each $\mathbf{r} \in R$, there exists a unique vector $\widetilde{\mathbf{r}}\in \spn(\mathbf{x_i}:i\leqslant u)$ such that $\Theta L \widetilde{\mathbf{r}} = \mathbf{r}$. Note that $\Vert \widetilde{\mathbf{r}}\Vert_\infty = O_{C,L,\eta}(N)$. Letting $\widetilde{R}$ denote the set of these $\widetilde{\mathbf{r}}$, we see that $\widetilde{R}$ satisfies part (3).

If $\eta$ is small enough in terms of $L$, then $R$ has size at most $1$. Indeed, if $\mathbf{r^{(1)}}$ and $\mathbf{r^{(2)}}$ are two different vectors in $R$, with respective $\widetilde{\mathbf{r}}^{\mathbf{(1)}}$ and $\widetilde{\mathbf{r}}^{\mathbf{(2)}}$, then $G(L\widetilde{\mathbf{r}}^{\mathbf{(1)}} + \mathbf{v}) \neq 0$ and $G(L\widetilde{\mathbf{r}}^{\mathbf{(2)}} + \mathbf{v}) \neq 0$. Hence $\Vert L(\widetilde{\mathbf{r}}^{\mathbf{(1)}} - \widetilde{\mathbf{r}}^{\mathbf{(2)}})\Vert_\infty \ll\eta$. Yet $\Vert\Theta (L(\widetilde{\mathbf{r}}^{\mathbf{(1)}} - \widetilde{\mathbf{r}}^{\mathbf{(2)}}))\Vert_\infty = \Vert \mathbf{r^{(1)}} - \mathbf{r^{(2)}}\Vert_\infty \gg 1$ (which is a contradiction). In this instance, writing $\mathbf{v}$ in the form $L\mathbf{a}$, we may pick $\widetilde{\mathbf{r}} \in \mathbb{Z}^d$ to be an element in $\spn(\mathbf{x_i}:i\leqslant u)$ that minimises $\Vert  \Theta L(\widetilde{\mathbf{r}} + \mathbf{a})\Vert_{\infty}$ over all $\widetilde{\mathbf{r}} \in \mathbb{Z}^d$\\


\textbf{Parts (4), (5), (6), and (11)}: By the definition of $\widetilde{R}$, and the fact that $\Xi(\mathbb{Z}^{d-u}) = \mathbb{Z}^d \cap \ker (\Theta L)$, we have that $T_{F,G,N}^{L,\mathbf{v}}(f_1,\dots,f_d)$ is equal to 
\begin{equation}
\label{equation getting integer rows}
\sum\limits_{\widetilde{\mathbf{r}} \in \widetilde{R}}\frac{1}{N^{d-m}} \sum\limits_{\mathbf{n} \in \mathbb{Z}^{d-u}} \Big(\prod\limits_{j=1}^d f_j(\xi_j(\mathbf{n}) + \widetilde{r}_j) \Big)F\Big(\frac{\Xi(\mathbf{n}) + \widetilde{\mathbf{r}}}{N}\Big) G(L\Xi(\mathbf{n}) + L\widetilde{\mathbf{r}} + \mathbf{v}).
\end{equation}
\noindent This is very close to being of the form required for part (6), and indeed it can be massaged into exactly the required form.\\

To do this, note that \begin{equation*}
\label{direct sum}
\mathbb{R}^m = \spn(L\mathbf{x_i}:i\leqslant u) \oplus \ker \Theta
\end{equation*} and so there exists an invertible linear map $Q:\mathbb{R}^m \longrightarrow \mathbb{R}^m$ with algebraic coefficients such that \begin{align*}
Q((\spn(L\mathbf{x_i}:i\leqslant u))) &= \mathbb{R}^u \times \{0\}^{m-u}, \\
Q(\ker \Theta) &= \{0\}^{u} \times \mathbb{R}^{m-u}. 
\end{align*} For all $\mathbf{x} \in \mathbb{R}^{d-u}$ we have \[ G(L\Xi(\mathbf{x}) + L\widetilde{\mathbf{r}} + \mathbf{v}) = (G \circ Q^{-1})(QL\Xi(\mathbf{x}) + QL\widetilde{\mathbf{r}} + Q\mathbf{v}). \] We also note that $QL\Xi(\mathbf{x}) \in \{0\}^u \times \mathbb{R}^{m-u}$, and that $QL\widetilde{\mathbf{r}} \in \mathbb{R}^u \times \{0\}^{m-u}$.  

Now, write $\pi_{m-u}:\mathbb{R}^{m}\longrightarrow \mathbb{R}^{m-u}$ for the projection map onto the final $m-u$ coordinates. Define $G_{\widetilde{\mathbf{r}}}:\mathbb{R}^{m-u} \longrightarrow [0,1]$ by 
\begin{equation}
\label{explicit form of G}
G_{\widetilde{\mathbf{r}}}(\mathbf{x}) : = (G\circ Q^{-1})(\mathbf{x_0} + QL\widetilde{\mathbf{r}} + Q\mathbf{v} - (\pi_{m-u} Q \mathbf{v})_\mathbf{0}),
\end{equation} where $\mathbf{x_0}$ is the extension of $\mathbf{x}$ by $0$ in the first $u$ coordinates. Then $G_{\widetilde{\mathbf{r}}}$ satisfies the desired properties of part (3), since $\mathbf{x_0}$ and $QL\widetilde{\mathbf{r}} + Q\mathbf{v} - (\pi_{m-u} Q \mathbf{v})_\mathbf{0}$ are orthogonal. 

Then (\ref{equation getting integer rows}) is equal to 
\begin{equation}
\label{equation end of second part}
\sum\limits_{\widetilde{\mathbf{r}} \in \widetilde{R}}\frac{1}{N^{d-m}} \sum\limits_{\mathbf{n} \in \mathbb{Z}^{d-u}} \Big(\prod\limits_{j=1}^d f_j(\xi_j(\mathbf{n}) + \widetilde{r}_j) \Big)F\Big(\frac{\Xi(\mathbf{n}) + \widetilde{\mathbf{r}}}{N}\Big)G_{\widetilde{\mathbf{r}}}( \pi_{m-u} QL\Xi(\mathbf{n}) + \pi_{m-u} Q\mathbf{v}).
\end{equation} 
\noindent Let 
\begin{equation}
\label{equation definition of L prime}
L^\prime: = \pi_{m-u} QL\Xi, \qquad \mathbf{v}^\prime =  \pi_{m-u} Q\mathbf{v}.
\end{equation} Then $L^\prime:\mathbb{R}^{d-u} \longrightarrow \mathbb{R}^{m-u}$ is surjective, and \[T_{F,G,N}^{L,\mathbf{v}}(f_1,\dots,f_d) = \sum\limits_{\widetilde{\mathbf{r}} \in \widetilde{R}}T_{F,G_{\widetilde{\mathbf{r}}},N}^{L^\prime, \mathbf{v}^\prime,\Xi,\widetilde{\mathbf{r}}}(f_1,\dots,f_d).\] This resolves parts (5) and (6). But furthermore, by the construction of $G_{\widetilde{\mathbf{r}}}$, part (10) is also satisfied. \\

\textbf{Part (7)}: This is immediate from Lemma 4.10 of \cite{Wa17}. To spell it out, suppose for contradiction that there exists some surjective linear map $\varphi:\mathbb{R}^{m-u} \longrightarrow \mathbb{R}$ with $\varphi L^\prime (\mathbb{Z}^{d-u}) \subseteq \mathbb{Z}$, i.e. with $\varphi \pi_{m-u} QL\Xi(\mathbb{Z}^{d-u}) \subseteq \mathbb{Z}$. Then define the map $\Theta^\prime:\mathbb{R}^m \longrightarrow \mathbb{R}^{u+1}$ by \[ \Theta^\prime(\mathbf{x}) : = (\Theta(\mathbf{x}),\varphi \pi_{m-u} Q(\mathbf{x})).\] Then $\Theta^\prime$ is surjective, and $\Theta^\prime L(\mathbb{Z}^d)\subseteq \mathbb{Z}^{u+1}$. This second fact is immediately seen by writing $\mathbb{Z}^d$ with respect to the lattice basis $\mathcal{B}$ from (\ref{basis for Rd}). This contradicts the assumption that $L$ has rational dimension $u$. So $L^\prime$ is purely irrational.  \\

\textbf{Part (8)}: Suppose $L\notin V_{\degen}^*(m,d)$ and suppose for contradiction that $\Xi$ has infinite Cauchy-Schwarz complexity. Letting $\mathbf{e_1},\dots,\mathbf{e_d}$ denote the standard basis of $\mathbb{R}^d$, this means there exists $i,j\leqslant d$ and a non-zero vector $\mathbf{e_i} - \lambda\mathbf{e_j}$ such that $\mathbf{e_i}^* - \lambda\mathbf{e_j}^* \in \ker (\Xi^*)$. But $\ker (\Xi^*) = (\im \Xi)^0 = (\ker \Theta L)^0 = \im (L^* \Theta^*)$. Hence $\mathbf{e_i} - \lambda\mathbf{e_j} \in \im L^*$, which implies that $L \in V_{\degen}^*(m,d)$, contradicting our hypothesis.\\

The lemma is proved. 
\end{proof}

\begin{Remark}
\label{Remark generalising Green Tao}
\emph{Applying Lemma \ref{Lemma generating a purely irrational map} with $f_j = \Lambda^\prime$ for all $j$, and when $L$ has rational dimension $m$, it is evident that estimating $T_{F,G,N}^{L,\mathbf{v}}(\Lambda^\prime,\dots,\Lambda^\prime)$ is equivalent to counting solutions to $\vert \widetilde{R}\vert$ systems of linear equations given by $\Xi$. This is handled by the Main Theorem of \cite{GT10}. In this sense, one may see how our work in this paper generalises Green-Tao's work in \cite{GT10} to the cases in which the rational dimension is not equal to $m$. }
\end{Remark}
\section{Normal form}
\label{section normal form}
In this section we describe, very briefly, what it means for a linear map $(\psi_1,\dots,\psi_t) = \Psi:\mathbb{R}^d\longrightarrow \mathbb{R}^t$ to be in \emph{$s$-normal form}. For a more complete discussion we refer the reader to \cite{GT10} and \cite{Wa17}. 

\begin{Definition}[Normal form]
\label{Definition normal form}
Let $d,t$ be natural numbers, let $s$ be a non-negative integer, and let $(\psi_1,\dots,\psi_t) = \Psi:\mathbb{R}^d\longrightarrow \mathbb{R}^t$ be a linear map. We say that $\Psi$ is in $s$-normal form if for every $i \in [t]$ there exists a collection $J_i \subseteq \{ \mathbf{e_1},\dots,\mathbf{e_d}\}$ of basis vectors of cardinality $\vert J_i\vert\leqslant s+1$ such that $\prod_{\mathbf{e}\in J_i} \psi_{i^\prime}(\mathbf{e})$ is non-zero for $i^\prime = i$ and vanishes otherwise. 
\end{Definition}

The notion of normal form is intimately connected with the notion of finite Cauchy-Schwarz complexity (Definition \ref{Definition finite complexity}). The key proposition was proved\footnote{In \cite{Wa17} we were forced to prove a delicate quantitative version, but this will not be necessary here.} in \cite{GT10}. 

\begin{Lemma}[Normal form extensions]
\label{Lemma normal form algorithm}
Let $d,t$ be natural numbers, and let $(\psi_1,\dots,\psi_t) = \Psi:\mathbb{R}^d \longrightarrow \mathbb{R}^t$ be a linear map with finite Cauchy-Schwarz complexity. Then there is a linear map $\Psi^\prime:\mathbb{R}^{d^\prime}\longrightarrow \mathbb{R}^t$ such that:
\begin{itemize}
\item $d^\prime = O(1)$;
\item for some vectors $\mathbf{f_k}\in \mathbb{R}^d$ that satisfy $\Vert \mathbf{f_k}\Vert_\infty=O_{\Psi}(1)$ for every $k$, the map $\Psi^\prime$ is of the form $$\Psi^\prime(\mathbf{u}, x_1,\dots,x_{d^\prime - d}) = \Psi(\mathbf{u} + x_1\mathbf{f_1}+ \dots + x_{d^\prime - d}\mathbf{f_{d^\prime - d}})$$ for all $\mathbf{u} \in \mathbb{R}^d$;
\item $\Psi^\prime$ is in $s$-normal form, for some $s= O(1)$.
\end{itemize}
\end{Lemma}

\begin{proof}
In \cite[Lemma 4.4]{GT10} this lemma was proved for a linear map over a $\mathbb{Q}$-vector space. The proof over $\mathbb{R}$ is identical. Alternatively one can iterate \cite[Proposition 6.7]{Wa17} over all $i\leqslant t$.
\end{proof}
\begin{Remark}
\label{Remark Cauchy Scwarz complexity}
\emph{In Lemma \ref{Lemma normal form algorithm} one may take $s$ to be the Cauchy-Schwarz complexity of $\Psi$. This notion will not be used in this paper, save for the `finite versus infinite' dichotomy already given in Definition \ref{Definition finite complexity}.}
\end{Remark}
\part{Pseudorandomness}
\label{part pseudorandomness}
Notions of pseudorandomness are crucial to the theory of higher order Fourier analysis. A small Gowers norm is one such notion, as is satisfying the `linear forms condition' of \cite{GT08} and \cite{GT10}. In this part we review what is known about Gowers norms in relation to the primes, and then formulate a `linear inequalities condition', which will be the analogous notion of pseudorandomness for this paper. \\

\section{The $W$-trick and Gowers norms}
\label{section W trick and Gowers norms}

To begin with, let us recall the definition of the Gowers norm over a cyclic group and over $[N]$. Given a function $f:\mathbb{Z}/N\mathbb{Z}\longrightarrow \mathbb{C}$, and a natural number $d$, one defines the Gowers $U^{d}$ norm $\Vert f\Vert_{U^{d}(N)}$ to be the unique non-negative solution to the equation
\begin{equation}
\label{Definition of Gowers norms}
\Vert f\Vert_{U^{d}(N)}^{2^d} = \frac{1}{N^{d+1}}\sum\limits_{x,h_1,\cdots,h_d}\prod\limits_{\boldsymbol{\omega}\in \{0,1\}^d}\mathscr{C}^{\vert \boldsymbol{\omega} \vert} f(x+\mathbf{h}\cdot\boldsymbol{\omega}),
\end{equation}
\noindent where $\vert \boldsymbol{\omega}\vert = \sum_i \omega_i$, $\mathbf{h} = (h_1,\cdots,h_d)$, $\mathscr{C}$ is the complex-conjugation operator, and the summation is over $x,h_1,\cdots,h_d\in \mathbb{Z}/N\mathbb{Z}$. It is not immediately obvious why the right-hand side of  (\ref{Definition of Gowers norms}) is always a non-negative real, nor why the $U^d$ norms are genuine norms if $d\geqslant 2$, but both facts are true. There are many expositions of the standard theory of these norms available in the literature, for example \cite[Chapter 11]{TaVu10} and \cite{Gr07}. For the most general treatment, the reader may consider Appendices B and C of \cite{GT10}. 

In the sequel we will be considering functions defined on $[N]$ rather than on $\mathbb{Z}/N\mathbb{Z}$. However, the Gowers norm of such functions may be easily defined by reference to the cyclic group case. Indeed, if $f:[N] \longrightarrow \mathbb{C}$, and $d$ is a natural number, one chooses a natural number $N^\prime > N$ and then considers $[N]$ as an initial segment of $\mathbb{Z}/N^\prime \mathbb{Z}$ (viewing $[N^\prime]$ as a set of representative classes for $\mathbb{Z}/N^\prime \mathbb{Z}$). One then defines \begin{equation}
\label{equation GN over ZNZ}
\Vert f\Vert_{U^d[N]} := \frac{\Vert f1_{[N]}\Vert_{U^d(N^\prime)}}{\Vert 1_{[N]}\Vert_{U^d(N^\prime)}},
\end{equation} which is independent of $N^\prime$ provided $N^\prime/N$ is large enough in terms of $d$. 

This is as much background as we will give here, and the reader is invited to consult the aforementioned references for more detail. A Gowers norm over $\mathbb{R}$ will also appear later on in this paper, but will be introduced in Section \ref{section transfer} as and when it is needed. \\

We move our consideration to the primes. Given some fixed modulus $q$ the primes are not uniformly distributed across arithmetic progressions modulo $q$ (as almost all the primes are coprime to $q$), and this lack of uniformity is an obstacle when trying to count solutions to equations in primes. Fortunately, there is a technical device, known as the $W$-trick, that has long been used to manage this difficulty. 

This device is usually introduced via the following function.
 
\begin{Definition}
\label{Definition W tricked von mangoldt function}
Let $N$ be a natural number, and let $W$ be as in Section \ref{section conventions}. For any natural number $b$ with $(b,W)=1$, let $\Lambda_{b,W}^\prime: \mathbb{Z} \longrightarrow \mathbb{R}_{\geqslant 0}$ be defined by \[\Lambda^\prime_{b,W}(n) = \begin{cases} \frac{\varphi(W)}{W}\Lambda^\prime(Wn+b) & n\geqslant 1 \\ 0 & \text{otherwise}. \end{cases} \]
\end{Definition}
\noindent The idea from \cite{GT10}, going back to \cite{Gr05} and \cite{GT08}, is that the function \[\frac{1}{\varphi(W)}\sum\limits_{\substack{b \leqslant W \\ \,(b,W) = 1}} \Lambda_{b,W}^\prime\] should act as a proxy for $\Lambda^\prime$, while each $\Lambda_{b,W}^\prime$ enjoys strong pseudorandomness properties. For example we have the following deep result, which is a crucial component of the proof of Theorem \ref{Theorem Green Tao} on linear equations in primes. 

\begin{Theorem}{\cite[Theorem 7.2]{GT10}}
\label{Tool from Green and Tao}
Let $N,s$ be natural numbers, and let $w^*: \mathbb{N} \longrightarrow \mathbb{R}_{ \geqslant 0}$ be any function that satisfies $w^*(n) \longrightarrow \infty$ as $n\rightarrow \infty$ and $w^*(n) \leqslant \frac{1}{2} \log\log n$ for all $n$. Let $b = b(N)$ be a natural number that satisfies $b\leqslant W^*$ and $(b,W^*) = 1$. Then
\begin{equation}
\label{key green tao result}
\Vert \Lambda^\prime_{b,W^*}-1\Vert_{U^{s+1}[N]}=o(1)
\end{equation}
\noindent as $N\rightarrow \infty$, where the $o(1)$ term may depend on the function $w^*$ chosen (but is independent of the choice of $b$). 
\end{Theorem}
\noindent We remind the reader that $s$ is a dimension parameter, and so dependence on $s$ is not denoted explicitly in our implied constants. 

\begin{Remark}
\emph{In \cite{GT10} Theorem \ref{Tool from Green and Tao} is proved conditionally, relying on two other conjectures. But, as we intimated in the introduction, these conjectures were later settled in joint work of Green-Tao and Green-Tao-Ziegler \cite{GT12, GTa12, GTZ12}.} 
\end{Remark}

\begin{Remark}
\label{Remark rescaling w trick}
\emph{We will use Theorem \ref{Tool from Green and Tao} to prove Theorem \ref{Main theorem}. Unfortunately it seems that this cannot be done in the same manner as in \cite{GT10}, i.e. by splitting $[N]$ into arithmetic progressions modulo $W$ at an early stage and then performing subsequent manipulations with the functions $\Lambda^\prime_{b,W}$.}

\emph{As a heuristic, instead of considering an inequality such as 
\begin{equation}
\label{basic inequality}
\Vert L\mathbf{n} + \mathbf{v}\Vert_\infty \leqslant \varepsilon,
\end{equation} for some $L$ with irrational coefficients and some positive $\varepsilon$, \cite{GT10} considers (\ref{basic inequality}) for some $L$ with rational coefficients and sets $\varepsilon$ equal to $0$. Under those assumptions one may rescale the variables $\mathbf{n}$ by a factor of $W$, as required in Definition \ref{Definition W tricked von mangoldt function}, without fundamentally altering the problem. However, in the more general scenario of Theorem \ref{Main theorem}, where $\varepsilon$ is strictly positive, rescaling the variable $\mathbf{n}$ by a factor of $W$ means we must replace $\varepsilon$ by $\varepsilon/W$, and we cannot afford this loss, as the manipulations in Section \ref{section transfer} lose some powers of $\varepsilon$. As far as we have been able to tell, this means that we cannot perform the $W$-trick in this manner.}
\end{Remark}


To circumvent this issue of scaling, we will manipulate with the local von Mangoldt functions $\Lambda_{\mathbb{Z}/W\mathbb{Z}}$ throughout, saving our rescaling for the very end of the argument. Regarding the control on Gowers norms, the following lemma is therefore the more appropriate bound.

\begin{Lemma}
\label{Lemma Corollary of tool from Green-Tao}
Let $N,s$ be natural numbers. Then
$$\Vert \Lambda^\prime - \Lambda_{\mathbb{Z}/W\mathbb{Z}}\Vert_{ U^{s+1}[N]} = o(1)$$ as $N\rightarrow \infty$.
\end{Lemma}
\noindent The proof is a standard deduction from results of \cite{GT10}, achieved by splitting into arithmetic progressions modulo $W$. We would however like to thank the anonymous referee for suggesting a simplification to our original argument. 
\begin{proof}
Let $(\psi_{\bo})_{\bo \in \{0,1\}^{s+1}} = \Psi:\mathbb{R}^{s+2} \longrightarrow \mathbb{R}^{2^{s+1}}$ denote the linear map giving the Gowers norm, i.e. where each $\psi_{\bo}$ is of the form $\psi_{\bo}(x,\mathbf{h}) = x + \bo \cdot \mathbf{h}$. From expression (\ref{equation GN over ZNZ}), we then have
\begin{equation}
\label{general linear form comparison of von mangoldt and local von mangoldt}
\Vert \Lambda^\prime - \Lambda_{\mathbb{Z}/W\mathbb{Z}}\Vert_{ U^{s+1}[N]}^{2^{s+1}}=\frac{1}{\vert Z\vert}\sum\limits_{\mathbf{n}\in Z } \prod\limits_{\bo \in \{0,1\}^{s+1}} (\Lambda^\prime - \Lambda_{\mathbb{Z}/W\mathbb{Z}})(\psi_{\bo}(\mathbf{n}))\end{equation}
\noindent where \[ Z = \{ \mathbf{n} \in \mathbb{\mathbb{Z}}^{s+2} : \Psi(\mathbf{n}) \in [1,N]^{2^{s+1}}\}.\] It is immediate that $\vert Z\vert \asymp N^{s+2}$. 

We now split into arithmetic progressions modulo $W$. To this end let $A \subset [W]^{s+2}$ be the set defined by \[ A = \{ \mathbf{a} \in [W]^{s+2}: (\psi_{\bo}(\mathbf{a}), W) = 1 \text{ for all } \bo \in \{0,1\}^{s+1}\}.\] Then the right-hand side of (\ref{general linear form comparison of von mangoldt and local von mangoldt}) is  
\begin{equation}
\label{gowers correction}
 \asymp \frac{1}{N^{s+2}} \sum\limits_{\mathbf{a} \in A} \sum\limits_{ \substack{ \mathbf{m} \in \mathbb{Z}^{s+2} \\ W\mathbf{m} + \mathbf{a} \in Z }} \prod\limits_{ \bo \in \{0,1\}^{s+1}} (\Lambda^\prime - \Lambda_{\mathbb{Z}/W\mathbb{Z}})( \psi_{\bo}(W\mathbf{m} + \mathbf{a})),
 \end{equation} plus an error of magnitude at most \[ \frac{ \log^{2^{s+1}} W}{N^{s+2}} \sum\limits_{ \mathbf{n} \in Z } \sum\limits_{ \bo \in \{0,1\}^{s+1}} 1_{[0,W]}(\psi_{\bo}(\mathbf{n})).\] This error is $o(1)$. 

By the linearity of $\Psi$, and recalling the definition of $\Lambda_{b,W}^\prime$ from Definition \ref{Definition W tricked von mangoldt function}, we have that expression (\ref{gowers correction}) is equal to
\[ \frac{1}{N^{s+2}} \sum\limits_{\mathbf{a} \in A} \Big(\frac{W}{\varphi(W)}\Big)^{2^{s+1}}\sum\limits_{\substack{\mathbf{m} \in \mathbb{Z}^{s+2} \\ W\mathbf{m} + \mathbf{a} \in Z}} \prod\limits_{\bo \in \{0,1\}^{s+1}} (\Lambda^\prime_{\psi_{\bo}(\mathbf{a}), W} - 1)(\psi_{\bo}(\mathbf{m})).\] Observe that \[\vert A\vert = \beta_W W^{s+2}\frac{\varphi(W)^{2^{s+1}}}{W^{2^{s+1}}},\] where \[\beta_W: = \frac{1}{W^{s+2}}\sum\limits_{\mathbf{m}\in [W]^{s+2}}\prod\limits_{\bo\in \{0,1\}^{s+1}}\Lambda_{\mathbb{Z}/W\mathbb{Z}}(\psi_{\bo}(\mathbf{m}))\] is the local factor associated to the system of forms $\Psi$. Since $\Psi$ has finite Cauchy-Schwarz complexity, we have the bound $\beta_W=O(1)$ (by \cite[Lemma 1.3]{GT10}). This means that the lemma would follow from the bound 
\begin{equation}
\label{what it would follow from}
\frac{1}{(N/W)^{s+2}}\sum\limits_{\substack{\mathbf{m} \in \mathbb{Z}^{s+2} \\ \mathbf{m}   \in (Z-\mathbf{a})/W}} \prod\limits_{\bo \in \{0,1\}^{s+1}} (\Lambda^\prime_{\psi_{\bo}(\mathbf{a}), W} - 1)(\psi_{\bo}(\mathbf{m})) = o(1)
\end{equation}
for each fixed $\mathbf{a} \in A$. What's more, expression (\ref{what it would follow from}) is an immediate consequence of the Gowers-Cauchy-Schwarz inequality when combined with Theorem \ref{Tool from Green and Tao}.

To spell out some of the details, let $M: = \lfloor N/W \rfloor$ and let $M^\prime>M$ be a natural number with $M^\prime /M$ large enough in terms of $s$. Then, recalling the definition of the set $Z$, the left-hand side of (\ref{what it would follow from}) is equal to 
\begin{equation}
\label{next}
\frac{1}{M^{s+2}}\sum\limits_{\substack{\mathbf{m} \in \mathbb{Z}^{s+2} \\ \Psi(\mathbf{m}) \in [M]^{2^{s+1}}}} \prod\limits_{\bo \in \{0,1\}^{s+1}} (\Lambda^\prime_{\psi_{\bo}(\mathbf{a}), W} - 1)(\psi_{\bo}(\mathbf{m})) + o(1).
\end{equation}
Taking the $o(1)$ term as read, this is
\begin{equation} 
\label{about to do GCS}
\ll \frac{1}{(M^\prime)^{s+2}} \sum\limits_{ \mathbf{m} \in (\mathbb{Z}/M^\prime \mathbb{Z})^{s+2}}\prod \limits_{\bo \in \{0,1\}^{s+1}} (\Lambda^\prime_{\psi_{\bo}(\mathbf{a}), W}1_{[M]} - 1_{[M]})(\psi_{\bo}(\mathbf{m})).
\end{equation}
Now, by the Gowers-Cauchy-Schwarz inequality (see \cite[Expression (11.6)]{TaVu10}), expression (\ref{about to do GCS}) is at most
\[\max_{\bo \in \{0,1\}^{s+1}} \Vert \Lambda^\prime_{\psi_{\bo}(\mathbf{a}),W}1_{[M]} - 1_{[M]}\Vert_{U^{s+1}(M^\prime)}^{2^{s+1}} .\] By expression (\ref{equation GN over ZNZ}), this is bounded above by a constant times 
\begin{equation}
\label{nearly amenable}
\max_{\substack{1 \leqslant b \leqslant (s+2)W \\ (b,W) = 1}} \Vert \Lambda^\prime_{b,W} - 1\Vert_{U^{s+1}[M]}^{2^{s+1}} .
\end{equation}

Expression (\ref{nearly amenable}) is directly amenable to Theorem \ref{Tool from Green and Tao}, with the only wrinkle being the fact that Theorem \ref{Tool from Green and Tao} only applied to functions $\Lambda^\prime_{b,W}$ with $b \leqslant W$. But this is easy to deal with. Indeed, for natural numbers $n$ and $k$ we have the identity \[ \Lambda^\prime_{b+kW,W}(n) = \Lambda^\prime_{b,W}(n+k),\] and so one establishes that if $b$ is in the range $1 \leqslant b \leqslant (s+2) W$ then \[\Vert \Lambda^\prime_{b,W} - 1\Vert_{U^{s+1}[M]}^{2^{s+1}} = \Vert \Lambda^\prime_{b^\prime,W} - 1\Vert_{U^{s+1}[M]}^{2^{s+1}} + E,\] where $b^\prime \in [W]$ and $b^\prime \equiv b  \,(\text{mod } W)$, and where the error term $E$ is at most a constant times \[ \frac{\log^{2^{s+1}}M}{M^{s+2}}\sum\limits_{ \mathbf{m} \in [M]^{s+2}} \sum\limits_{ \bo \in \{0,1\}^{s+1}} (1_{[s+2]}(\psi_{\bo}(\mathbf{m})) + 1_{\{M + 1,\dots, M + s+2\}}(\psi_{\bo}(\mathbf{m}))).\] We have $E = o(1)$ and therefore, by Theorem \ref{Tool from Green and Tao}, expression (\ref{nearly amenable}) is $o(1)$. The lemma follows. 
\end{proof}

\section{Inequalities in lattices}
\label{section Inequalities in arithmetic progressions}

This section will be devoted to proving the following technical lemma. This is the only part of the paper in which we pay especial attention to the quantitative aspects of the smooth cut-off functions, as the lemma will be applied in contexts where the functions $F$ and $G$ depend on the asymptotic parameter $N$ (albeit tamely).

\begin{Lemma}[Inequalities in lattices]
\label{Lemma in APs}
Let $N,m,d, h$ be natural numbers, with $d\geqslant h\geqslant m+1$, and let  $\gamma$ be a positive constant. Suppose that $N>2^{\frac{1}{\gamma}}$. Let $P$ be an additional set of parameters.  Let $(\xi_1,\dots,\xi_d) = \Xi:\mathbb{R}^h\longrightarrow \mathbb{R}^d$ be an injective linear map with integer coefficients and let $L:\mathbb{R}^h\longrightarrow \mathbb{R}^m$ be a purely irrational surjective linear map with algebraic coefficients. Let $\mathbf{v}\in \mathbb{R}^{m}$ and let $\widetilde{\mathbf{r}}\in\mathbb{Z}^d$. Let $e_1,\dots,e_d\in \mathbb{N}$ and suppose that $e_j< N^\gamma$ for all $j$. Let $F:\mathbb{R}^h\longrightarrow [0,1]$ and $G:\mathbb{R}^{m}\longrightarrow [0,1]$ be functions in $\mathcal{C}(P)$. Then, provided that $\gamma$ is small enough in terms of $L$, for all positive $K$ we have
\begin{align}
\label{equation removing divisors}
\frac{1}{N^{h - m}}\sum\limits_{\substack{\mathbf{n} \in \mathbb{Z}^h \\ e_j \vert \xi_j(\mathbf{n}) + \widetilde{r}_j \, \forall j \leqslant d}} F(\mathbf{n}/N)G(L\mathbf{n} + \mathbf{v})=\frac{ \alpha_{\mathbf{e},\widetilde{\mathbf{r}}}}{N^{h-m}} \int\limits_{\mathbf{x}\in \mathbb{R}^h} F(\mathbf{x}/N)G(L\mathbf{x} + \mathbf{v})\, d\mathbf{x}\nonumber \\ +O_{K,L,P}(N^{-K}),
\end{align}
\noindent where $\alpha_{\mathbf{e},\widetilde{\mathbf{r}}}$ is the local factor 
\begin{equation}
\label{local factor def}
\alpha_{\mathbf{e},\widetilde{\mathbf{r}}} := \lim\limits_{M\rightarrow \infty}\frac{1}{M^{h}}\sum\limits_{\mathbf{n}\in [M]^h} \prod\limits_{j\leqslant d} 1_{e_j\vert \xi_j(\mathbf{n}) + \widetilde{r}_j}.
\end{equation}
\end{Lemma}
\begin{Remark}
\label{Remark local factor for AP calculation}
\emph{If $h=d$ and if $\Xi:\mathbb{R}^h \longrightarrow \mathbb{R}^d$ is the identity map, then the Chinese Remainder Theorem guarantees that $\alpha_{\mathbf{e},\widetilde{\mathbf{r}}} = (e_1\dots e_d)^{-1}$. In the general case, the local factors $\alpha_{\mathbf{e},\widetilde{\mathbf{r}}}$ are the same objects as those factors $\alpha_{m_1,\dots,m_t}$ considered in \cite[Page 1831]{GT10}. }
\end{Remark}
\begin{proof}[Proof of Lemma \ref{Lemma in APs}]

We assume throughout that $\gamma$ is small enough in terms of $L$, and that $K$ is large enough in terms of the dimensions $m$, $d$, and $h$.

By applying Fourier inversion to $G$, we see that the left-hand side of (\ref{equation removing divisors}) is equal to 
\begin{equation}
\label{after first fourier transform}
\frac{1}{N^{h-m}}\int\limits_{\boldsymbol{\alpha}\in \mathbb{R}^{m}}\widehat{G}(\boldsymbol{\alpha})\sum\limits_{\substack{\mathbf{n} \in \mathbb{Z}^h \\ e_j \vert \xi_j(\mathbf{n}) + \widetilde{r}_j \, \forall j\leqslant d}} F(\mathbf{n}/N)e((L^T\boldsymbol{\alpha})\cdot \mathbf{n} + \boldsymbol{\alpha}\cdot \mathbf{v}) \, d\boldsymbol{\alpha}.
\end{equation}
\noindent To bound this integral, we split $\mathbb{R}^m$ into three ranges. Let $\eta$ be a small positive parameter to be chosen later, which we assume to be small enough in terms of $L$. We then define the so-called `trivial arc' by \[ \mathfrak{t} := \{\ba\in\mathbb{R}^{m}: \Vert \ba \Vert_\infty \geqslant N^\eta\}, \] the `minor arc' by \[\mathfrak{m} := \{\ba\in\mathbb{R}^{m}: N^{-1+\eta}\leqslant \Vert  \ba \Vert_\infty < N^\eta\}, \] and the `major arc' by \[\mathfrak{M} := \{\ba\in\mathbb{R}^{m}:  \Vert  \ba \Vert_\infty < N^{-1+\eta}\}.\] \\

\textbf{Trivial arc:}
By Lemma \ref{Lemma by parts}, $\vert\widehat{G}(\ba)\vert \ll_{K,P}\Vert\ba\Vert_\infty ^{-K}$. Therefore, applying the trivial $O_P(N^h)$ bound to the inner sum, we have \[\frac{1}{N^{h-m}}\int\limits_{\boldsymbol{\alpha}\in \mathfrak{t}}\widehat{G}(\boldsymbol{\alpha})\sum\limits_{\substack{\mathbf{n} \in \mathbb{Z}^h \\ e_j \vert \xi_j(\mathbf{n}) + \widetilde{r}_j \, \forall j}} F(\mathbf{n}/N)e((L^T\boldsymbol{\alpha})\cdot \mathbf{n} + \boldsymbol{\alpha}\cdot \mathbf{v}) \, d\boldsymbol{\alpha} \ll _{K,P}N^{-\eta K + O(1)}.\]\\

\textbf{Minor arc:} Choose $\mathbf{x}\in\mathbb{Z}^h$ to satisfy the simultaneous divisor conditions $e_j\vert \xi_j(\mathbf{x}) + \widetilde{r}_j$ for every $j\leqslant d$. If there is no such $\mathbf{x}\in\mathbb{Z}^h$ then (\ref{equation removing divisors}) is trivially true. Further, we may assume that $\mathbf{x}$ satisfies $\Vert\mathbf{x}\Vert_\infty\leqslant e_1\dots e_d$. Let $\Gamma_{\Xi,\mathbf{e}}$ denote the lattice \[\Gamma_{\Xi,\mathbf{e}} := \{ \mathbf{n} \in \mathbb{Z}^h: e_j\vert \xi_j(\mathbf{n}) \, \text{ for every } j\leqslant d\}.\]  Then
\begin{align}
\label{reformulation}
\sum\limits_{\substack{\mathbf{n} \in \mathbb{Z}^h \\ e_j \vert \xi_j(\mathbf{n}) + \widetilde{r}_j \, \forall j\leqslant d}} F(\mathbf{n}/N)  e((L^T \ba )\cdot \mathbf{n}) =
\sum\limits_{\mathbf{n} \in \Gamma_{\Xi,\mathbf{e}}} F((\mathbf{x}+\mathbf{n})/N)  e((L^T \ba) \cdot ( \mathbf{x} + \mathbf{n})).
\end{align}
\noindent Using this reformulation, we apply Poisson summation (Lemma \ref{Lemma Poisson summation}) to the inner sum of (\ref{after first fourier transform}). Then the contribution to (\ref{after first fourier transform}) from the minor arc $\mathfrak{m}$ is equal to  
\begin{equation}
\label{after Poisson summation}
\frac{N^{d-h+m}}{\operatorname{vol} (\mathbb{R}^h/ \Gamma_{\Xi,\mathbf{e}})} \int\limits_{\ba\in\mathfrak{m}} \widehat{G}(\ba) e(\ba\cdot \mathbf{v})\sum\limits_{\mathbf{c}\in \Gamma_{\Xi,\mathbf{e}}^*} \widehat{F}(N(\mathbf{c} - L^{ T}\ba)) e(\mathbf{x}\cdot\mathbf{c}) \, d\ba,
\end{equation}
where $\Gamma_{\Xi,\mathbf{e}}^*$ is the lattice that is dual to $\Gamma_{\Xi,\mathbf{e}}$ (see Definition \ref{Definition dual lattice}). \\

We need the following obvious lemma. 

\begin{Lemma}
\label{Lemma clearing denominators in dual lattices}
There is a natural number $A$, of size at most $O(N^{O(\gamma)})$, such that $A\Gamma_{\Xi,\mathbf{e}}^*\subset \mathbb{Z}^h$.
\end{Lemma}
\begin{proof}
There is an $h$-dimensional sublattice of $\Gamma_{\Xi,\mathbf{e}}$, namely $(e_1\dots e_d \mathbb{Z})^h$. Therefore, we may choose a lattice basis for $\Gamma_{\Xi,\mathbf{e}}$ all of whose elements $\mathbf{b}$ satisfy $\Vert \mathbf{b}\Vert_\infty  = O(N^{O(\gamma)})$. Let $M$ be the $h$-by-$h$ matrix that has these basis vectors as its columns. Then the columns of the matrix $(M^T)^{-1}$ are a lattice basis for the dual lattice $\Gamma_{\Xi,\mathbf{e}}^*$. The entries in $(M^T)^{-1}$ are rational numbers with numerator and denominator at most $O(N^{O(\gamma)})$. Clearing denominators, the lemma follows. 
\end{proof}

Let $\langle L^T\ba\rangle$ denote some $\mathbf{c}\in \Gamma_{\Xi,\mathbf{e}}^*$ that minimises the expression $ \Vert \mathbf{c} -L^T\ba\Vert_\infty$. We claim that the only term in (\ref{after Poisson summation}) that cannot be easily absorbed into the error term comes from $\mathbf{c} = \langle L^T\ba\rangle$.

Indeed, let $A$ be the quantity provided by Lemma \ref{Lemma clearing denominators in dual lattices}, and let $\langle L^T \ba \rangle_2$ denote the second closest point to $L^T\ba$ in the lattice $\Gamma_{\Xi,\mathbf{e}}^*$. If more than one such point exists, choose arbitrarily. Then \begin{equation}
\label{second closest lattice point}
\Vert\langle L^T \ba \rangle_2 - L^T\ba\Vert_\infty \geqslant a A^{-1},
\end{equation} where $a$ is some positive constant which depends only on $h$. By the triangle inequality and dyadic pigeonholing, one then has
\begin{align}
\label{contribution from distant lattice points}
\Big\vert\sum\limits_{\substack{\mathbf{c}\in\Gamma_{\Xi,\mathbf{e}}^*\\ \mathbf{c}\neq \langle L^T\ba \rangle}} \widehat{F}(N(\mathbf{c} - L^T\ba))e(\mathbf{x}\cdot\mathbf{c})\Big\vert \leqslant \sum\limits_{k=0}^\infty \sum\limits_{\substack{\mathbf{c}\in\Gamma_{\Xi,\mathbf{e}}^* \\ 2^{k} a A^{-1} \leqslant \Vert \mathbf{c} - L^T \mathbf{a}\Vert_\infty \leqslant 2^{k+1} a A^{-1}}}\vert \widehat{F}(N(\mathbf{c} - L^T\ba))\vert.
\end{align}
\noindent By Lemma \ref{Lemma clearing denominators in dual lattices} we also have the estimate 
\begin{equation}
\label{estimate sum over lattice points}
\sum\limits_{\substack{\mathbf{c}\in\Gamma_{\Xi,\mathbf{e}}^* \\ \Vert \mathbf{c} - L^T\ba\Vert_\infty \leqslant R}} 1 \ll \operatorname{max}(1,R^hA^h),
\end{equation}
\noindent which holds for all $R>0$. Using (\ref{estimate sum over lattice points}), Lemma \ref{Lemma by parts}, and the bound $A = O(N^{O(\gamma)})$, the quantity (\ref{contribution from distant lattice points}) is seen to be \begin{equation}
\label{the thing that it follows from}
\ll_{K,P}  N^{-K + O(K\gamma)}.
\end{equation}

This implies that the contribution from these lattice points to (\ref{after Poisson summation}) is at most 
\begin{align}
\label{overall contributoin form distant lattice points}
&\ll_{K,P} N^{-K + O(K\gamma) + O(1)} \int\limits_{\ba \in \mathfrak{m}} \vert \widehat {G}(\ba)\vert \, d\ba \nonumber \\
&\ll_{K,P}  N^{-K + O(K \gamma) + O(1) + O(\eta) }.
\end{align}
\noindent Since $\gamma$ and $\eta$ are small enough, (\ref{overall contributoin form distant lattice points}) is \[\ll_{K,P} N^{-K/2},\] which may be absorbed into the error term of (\ref{equation removing divisors}) after adjusting the implied constant appropriately. \\

It remains to estimate
\begin{align}
\label{after removing far lattice points}
\frac{N^{d-h+m}}{\operatorname{vol} (\mathbb{R}^h/ \Gamma_{\Xi,\mathbf{e}})} \int\limits_{\ba\in\mathfrak{m}} &\widehat{G}(\ba) e(\ba\cdot \mathbf{v})\widehat{F}(N(\langle L^T\ba \rangle - L^T\ba)) e(\mathbf{x}\cdot \langle L^T\ba \rangle)\, d\ba.\\\nonumber
\end{align} We have the following key lemma.
\begin{Lemma}
\label{Lemma key to minor arcs}
Under the assumption that $\eta$ and $\gamma$ are suitably small in terms of $L$, \[\operatorname{sup}\limits_{\ba \in\mathfrak{m}} \vert\widehat{F}(N(\langle L^T\ba \rangle -L^T\ba))\vert \ll_{K,L,P} N^{-K\eta}.\] 
\end{Lemma}
\begin{Remark}
\emph{The proof of this lemma uses the algebraicity of the coefficients of $L$. One should note that the bound (\ref{only point where algebraicity is used}) below, which holds for matrices with algebraic coefficients, also holds for almost all matrices. It is this fact which ultimately leads to our observation in the introduction that the main theorems of this paper hold for almost all matrices (as well as for matrices with algebraic coefficients, as stated).} 
\end{Remark}
\begin{proof} 
Certainly, by rescaling $\ba$ and using Lemmas \ref{Lemma by parts} and \ref{Lemma clearing denominators in dual lattices},
\begin{align}
\label{bound in terms of approximation function}
\operatorname{sup}\limits_{\ba \in\mathfrak{m}} \vert\widehat{F}(N(\langle L^T\ba \rangle -L^T\ba))\vert& \ll_{K,P} N^{-K}( \inf\limits_{\ba \in\mathfrak{m}}\Vert \langle L^T\ba \rangle -L^T\ba \Vert_\infty)^{-K} \nonumber\\
&\ll_{K,P} A^KN^{-K} ( \inf\limits_{\substack{ \bb \in \mathbb{R}^m \\ AN^{-1 + \eta} \leqslant \Vert \bb\Vert_\infty \leqslant AN^{\eta} }} \inf\limits_{\mathbf{n} \in \mathbb{Z}^h} \Vert \mathbf{n} - L^T \bb  \Vert_\infty)^{-K}
\end{align}
\noindent 
The quantity 
\begin{equation}
\label{approximation function expression}
\inf\limits_{\substack{ \bb \in \mathbb{R}^m \\ AN^{-1 + \eta} \leqslant \Vert \bb\Vert_\infty \leqslant A N^{\eta}}} \inf\limits_{\mathbf{n} \in \mathbb{Z}^h} \Vert \mathbf{n} - L^T \bb \Vert_\infty
\end{equation} encodes information about diophantine approximations to the coefficients of $L$. For example, since $L$ is purely irrational, by definition\footnote{The reader may consult Definition \ref{Definition rational space}.} we have $L^T \bb \neq \mathbb{Z}$ for any $\bb \neq \mathbf{0}$. Therefore, since the function \[ \bb \mapsto \inf\limits_{\mathbf{n} \in \mathbb{Z}^h} \Vert L^T \bb - \mathbf{n} \Vert_\infty\] is continuous, (\ref{approximation function expression}) is always non-zero. We will need a quantitative refinement of this fact.

Fortunately, in \cite{Wa17} we extensively analysed expressions such as (\ref{approximation function expression}). Consider Definition 2.8 of \cite{Wa17} in particular, in which we defined the approximation function\footnote{We stress that the notation $A_L$ is unrelated to the parameter $A$ from this section.} $A_L$. In this language, (\ref{approximation function expression}) is equal to \[A_L(AN^{-1+\eta},A^{-1}N^{-\eta}).\] Therefore, since $L$ is purely irrational and has algebraic coefficients, Lemma E.1 of \cite{Wa17} tells us that \begin{equation}
\label{only point where algebraicity is used} \inf\limits_{\substack{ \bb \in \mathbb{R}^m \\ AN^{-1 + \eta} \leqslant \Vert \bb\Vert_\infty \leqslant A N^{\eta}}} \inf\limits_{\mathbf{n} \in \mathbb{Z}^h} \Vert L^T \bb - \mathbf{n} \Vert_\infty \gg_{L} \min ( A N^{-1 + \eta}, A^{-O_L(1)}N^{-O_L( \eta )}).
\end{equation} Since $\eta$ and $\gamma$ are small enough in terms of $L$, and since $A = O(N^{O(\gamma)})$, (\ref{bound in terms of approximation function}) implies that 
\[\operatorname{sup}\limits_{\ba \in\mathfrak{m}} \vert\widehat{F}(N(\langle L^T\ba \rangle -L^T\ba))\vert \ll_{K,L,P} N^{-K\eta},\] as claimed.  
\end{proof}
\noindent The lemma above implies that (\ref{after removing far lattice points}) has size 
\begin{equation}
\label{minor arc total contribution}
O_{K,L,P}(N^{-K \eta + O(1)}),
\end{equation} which is thus our bound for the total contribution from the minor arc $\mathfrak{m}$. \\

\noindent \textbf{Major arc:} Performing the same Poisson summation argument as in the minor arc case, the main term on the left-hand side of (\ref{equation removing divisors}) is equal to 
\begin{align}
\label{major arc expression}
\frac{N^{d-h+m}}{\operatorname{vol} (\mathbb{R}^h/ \Gamma_{\Xi,\mathbf{e}})} \int\limits_{\ba\in\mathfrak{M}} &\widehat{G}(\ba) e(\ba\cdot \mathbf{v})
\widehat{F}(N(\langle L^T\ba \rangle - L^T\ba)) e(\mathbf{x}\cdot \langle L^T\ba \rangle)\, d\ba.
\end{align} 

For $\ba\in\mathfrak{M}$ one has $\Vert L^T\ba\Vert_\infty \ll_L N^{-1+\eta}$, and so $\langle L^T\ba\rangle = \mathbf{0}$. Therefore (\ref{major arc expression}) is equal to 
\begin{equation}
\label{equation before expanding integral}
\frac{N^{d-h+m}}{\operatorname{vol} (\mathbb{R}^h/ \Gamma_{\Xi,\mathbf{e}})}\int\limits_{\ba \in\mathfrak{M}} \widehat{G}(\ba) e(\ba\cdot\mathbf{v}) \widehat{F}(-NL^T\ba) \, d\ba.
\end{equation} Since $L^T:\mathbb{R}^{m}\longrightarrow \mathbb{R}^h$ is injective, one has $\Vert L^T\ba\Vert_\infty \gg_{L} \Vert \ba\Vert_\infty$. Therefore (\ref{equation before expanding integral}) is equal to
\begin{align*}
&\frac{N^{d-h+m}}{\operatorname{vol} (\mathbb{R}^h/ \Gamma_{\Xi,\mathbf{e}})} \Big(\int\limits_{\ba \in\mathbb{R}^{m}} \widehat{G}(\ba) e(\ba\cdot\mathbf{v}) \widehat{F}(- NL^T\ba) \, d\ba\Big) + O_{K,L,P}(N^{-\eta K + O(1)}),
\end{align*}
\noindent which, after the obvious manipulations, equals \[\frac{1}{N^{h-m}\operatorname{vol} (\mathbb{R}^h/ \Gamma_{\Xi,\mathbf{e}})}\Big(\int\limits_{\mathbf{x}\in\mathbb{R}^h} F(\mathbf{x}/N) G(L\mathbf{x} + \mathbf{v}) \, d\mathbf{x}\Big) + O_{K,L,P}(N^{-\eta K + O(1)}).\] 

Fixing suitably small $\eta$ and $\gamma$, and combining the contribution from the trivial, minor, and major arc, we deduce that \begin{align}
\label{final error term}
\frac{1}{N^{h - m}}\sum\limits_{\substack{\mathbf{n} \in \mathbb{Z}^h \\ e_j \vert \xi_j(\mathbf{n}) + \widetilde{r}_j \, \forall j}} F(\mathbf{n}/N)G(L\mathbf{n} + \mathbf{v})
= \frac{1}{N^{h-m}\operatorname{vol} (\mathbb{R}^h/ \Gamma_{\Xi,\mathbf{e}})}\int\limits_{\mathbf{x}\in\mathbb{R}^h} F(\mathbf{x}/N) G(L\mathbf{x} + \mathbf{v}) \nonumber\\+ O_{K,L,P}(N^{-\eta K + O(1)}).
\end{align}

By adjusting the implied constant appropriately, the error term from (\ref{final error term}) is $O_{K,L,P}(N^{-K})$ for all positive real $K$. The final observation is that, considering the definition of the local factor $\alpha_{\mathbf{e},\widetilde{\mathbf{r}}}$ in (\ref{local factor def}) and the fact that we assumed $\alpha_{\mathbf{e},\widetilde{\mathbf{r}}}$ is non-zero, \[\frac{1}{\operatorname{vol} (\mathbb{R}^h/ \Gamma_{\Xi,\mathbf{e}})} = \alpha_{\mathbf{e},\widetilde{\mathbf{r}}}.\]  
The lemma follows. 
\end{proof}

The following estimate will also be needed. 

\begin{Lemma}
\label{Lemma just arithmetic progressions}
Under the same hypotheses as Lemma \ref{Lemma in APs}, for all positive $K$ \begin{equation}
\label{equation just aps}
\frac{1}{N^{h}} \sum\limits_{\substack{ \mathbf{n} \in \mathbb{Z}^h\\ e_j \vert \xi_j(\mathbf{n}) + \widetilde{ r}_j \, \forall j\leqslant d}} F(\mathbf{n}/N) = \frac{\alpha_{\mathbf{e},\widetilde{\mathbf{r}}}} {N^{h}} \int\limits_{\mathbf{x} \in \mathbb{R}^h} F(\mathbf{x}/N) \, d\mathbf{x} + O_{K,L,P}( N^{-K}),
\end{equation}
\noindent where $\alpha_{\mathbf{e},\widetilde{\mathbf{r}}}$ is as in (\ref{local factor def}).
\end{Lemma}
\begin{proof}
By applying Poisson summation, the left-hand side of (\ref{equation just aps}) is equal to 
\begin{equation}
\frac{N^{d-h}}{ \vol(\mathbb{R}^h/ \Gamma_{\Xi,\mathbf{e}})} \sum\limits_{\mathbf{c} \in \Gamma^*_{\Xi,\mathbf{e}}} \widehat{F}(N\mathbf{c}) e(\mathbf{x} \cdot \mathbf{c}),
\end{equation}
where $\mathbf{x}$ and $\Gamma^*_{\Xi,\mathbf{e}}$ are as in (\ref{after Poisson summation}). By applying estimates (\ref{contribution from distant lattice points}) and (\ref{the thing that it follows from}), one shows that the main term of (\ref{equation just aps}) comes from the $\mathbf{c}=\mathbf{0}$ term above. After the obvious manipulations, this concludes the lemma.
\end{proof}

\section{The linear inequalities condition}
\label{section proof of pseduorandomness}

In \cite{GT10}, the key notion of pseudorandomness is the so-called `linear forms condition' (see Definition 6.2 of of that paper). The upshot is that in order to understand the number of solutions to a particular linear equation in primes, it is enough to understand the number of solutions to certain auxiliary linear equations weighted by a sieve weight $\nu$. In this paper an analogous philosophy holds. Indeed we will show that, in order to understand the number of solutions to a particular linear inequality in primes, it is enough to understand the number of solutions to certain auxiliary linear inequalities weighted by a sieve weight $\nu$. \\

Let us proceed with the formal definition. The reader is reminded that $W_j = \prod_{p\leqslant w_j(N)} p$ (see Section \ref{section conventions}). 
\begin{Definition}[Linear inequalities condition]
\label{Definition linear inequalities condition}
Let $m,d$ be natural numbers, and let $L:\mathbb{R}^d \longrightarrow \mathbb{R}^m$ be a linear map. For each natural number $N$, let $\nu_N:\mathbb{Z}\longrightarrow \mathbb{R}$ be a function. We say that the family of functions $(\nu_N)_{N=1}^\infty$ is \emph{$(L,w)$-pseudorandom} if the following holds. For all positive constants $C$ and for all sets of parameters $P$, for all compactly supported smooth functions $F:\mathbb{R}^d\longrightarrow [0,1]$ and $G:\mathbb{R}^m\longrightarrow [0,1]$ such that $F,G\in\mathcal{C}(P)$, and for all functions $w_1,\dots w_d:\mathbb{N}\longrightarrow \mathbb{R}_{\geqslant 0}$ that each satisfy $w_j(n)\rightarrow \infty$ as $n\rightarrow \infty$ and $w_j(n)\leqslant w(n)$ for all $n$, for all $\mathbf{v}\in\mathbb{R}^m$ satisfying $\Vert \mathbf{v} \Vert_\infty \leqslant CN$, and for functions $f_1,\dots,f_d:\mathbb{Z} \longrightarrow \mathbb{R}$ such that each $f_j$ equals either $\nu_N$ or $\Lambda_{\mathbb{Z}/W_j\mathbb{Z}} $,
\begin{equation}
\label{eqn defining pseudorandomness}
T^{L,\mathbf{v}}_{F,G,N}(f_1, \dots,f_d) = T^{L,\mathbf{v}}_{F,G,N}(\Lambda_{\mathbb{Z}/W_1\mathbb{Z}},\dots,\Lambda_{\mathbb{Z}/W_d\mathbb{Z}}) + o(1)
\end{equation}
as $N\rightarrow \infty$, where the $o(1)$ term may depend on the family $(\nu_N)_{N=1}^\infty$, on $C$, $L$, $P$, and on the functions $w_1,\dots,w_d$. 
\end{Definition}

\begin{Remark}
\emph{Equation (\ref{eqn defining pseudorandomness}) might seem to be a slightly curious formulation of a pseudorandomness principle, as it does not claim that the weight $\nu_N$ behaves like the constant $1$ function but rather behaves like the local von Mangoldt function. However, referring to Remark \ref{Remark rescaling w trick}, let us reiterate the comment that we are not performing the $W$-trick in the same manner as \cite{GT10}.} 
\end{Remark}

The aim of this section is to introduce a sieve weight $\nu_{N,w}^\gamma$, and to prove that it is $(L,w)$-pseudorandom for a large class of linear maps $L$. We begin by introducing the sieve weight from \cite[Appendix D]{GT10}. 

\begin{Definition}[Smooth sieve weight]
\label{Definition Smooth sieve weight}
Let $N$ be a natural number, $\gamma$ be a positive real, and define $R:= N^\gamma$. Let $\rho \in \mathcal{C}(\emptyset)$ be the smooth $1$-supported function fixed in Section \ref{section conventions}. Define the function $\Lambda_{\rho,R,2}:\mathbb{Z}\longrightarrow \mathbb{R}_{\geqslant 0}$ by the formula
\begin{equation}
\label{definition of GOldston weight}
\Lambda_{\rho,R,2}(n):=(\log R) \Big(\sum\limits_{d\vert n}\mu(d)\rho\big(\frac{\log d}{\log R}\big)\Big)^2,
\end{equation}
\noindent for non-negative integers $n$, and then by the obvious extension to negative integers. 
\end{Definition}
 We now define the family of majorants themselves.

\begin{Definition}[Pseudorandom majorant]
\label{Definnition pseudorandom majorant}
Let $N$ be a natural number, let $\gamma$ be a positive real, and let $R:= N^\gamma$. Define the constant \[c_{\rho,2} : = \int\limits_{0}^{\infty} \vert \rho^\prime (x)\vert^2 \, dx.\] Then define the weight $\nu_{N,w}^\gamma:\mathbb{Z}\rightarrow \mathbb{R}_{\geqslant 0}$ by \[\nu_{N,w}^\gamma(n) := \frac{1}{2c_{\rho,2}}\Lambda_{\rho,R,2}(n) + \frac{1}{2}\Lambda_{\mathbb{Z}/W\mathbb{Z}}(n).\] 
\end{Definition}
\noindent Note that $\nu_{N,w}^{\gamma}$ also depends on $\rho$, but we suppress that dependence from the notation (as we fixed $\rho$ in Section \ref{section conventions}). \\

We now state our main new result on the pseudorandomness of this sieve weight.

\begin{Theorem}[Pseudorandomness of sieve weights]
\label{Theorem pseudorandomness}
Let $m,d$ be natural numbers, with $d\geqslant m+2$. Let $L:\mathbb{R}^d \longrightarrow \mathbb{R}^m$ be a surjective linear map, and suppose that $L\notin V_{\degen}^*(m,d)$ and that the coefficients of $L$ are algebraic. Assume that $\gamma$ is a positive parameter that is small enough in terms of $L$. Then $\nu_{N,w}^{\gamma}$ is $(L,w)$-pseudorandom. 
\end{Theorem}

\noindent Temporarily dropping the convention that $w(N) = \max(1,\log\log\log N)$, we speculate that the following general result holds.

\begin{Conjecture}[Pseudorandomness conjecture]
\label{Conjecture pseudorandomness}
Let $m,d$ be natural numbers, with $d\geqslant m+2$. Let $L:\mathbb{R}^d \longrightarrow \mathbb{R}^m$ be a surjective linear map, and suppose that $L\notin V^*_{\degen}(m,d)$. Then there is some value of $\gamma$ and some function $w:\mathbb{N} \longrightarrow \mathbb{R}_{\geqslant 0}$, satisfying $w(N)\rightarrow \infty$ as $N\rightarrow \infty$, for which $\nu_{N,w}^{\gamma}$ is $(L,w)$-pseudorandom. 
\end{Conjecture}
\noindent Unfortunately we have not been able to resolve Conjecture \ref{Conjecture pseudorandomness}, but we strongly believe it to be true. If $d$ is large enough in terms of $m$ then the analytic methods of Parsell (see \cite{Pa02} and Appendix \ref{section an analytic argument}) can be used to show that $\nu_{N,w}^\gamma$ is $(L,w)$-pseudorandom without any algebraicity assumptions. But these methods seem harder to apply in the range $d\geqslant m+2$, and we have not been able to establish the appropriate mean value estimate. Resolving Conjecture \ref{Conjecture pseudorandomness} would, after a straightforward adaptation of the methods of this paper, enable one to remove the algebraicity assumption from Theorem \ref{Main theorem simpler version} and Theorem \ref{Main theorem}. 

\begin{Remark}
\label{Remark almost all}
\emph{The proof of Theorem \ref{Theorem pseudorandomness} is the only moment during the proof of the main theorems Theorem \ref{Main theorem simpler version} and Theorem \ref{Main theorem} when we use the fact that the coefficients of the original linear map $L$ are algebraic. Furthermore, we will ultimately only ever appeal to the linear inequalities condition for a certain finite collection of linear maps, which includes the original linear map $L$ itself as well as some auxiliary linear maps that are generated from applications of the Cauchy-Schwarz inequality. Since only the diophantine approximation properties of algebraic numbers are used (witness Lemma \ref{Lemma key to minor arcs} and \cite[Lemma E.1]{Wa17}), and since these properties are satisfied by almost all real numbers, one may show that Theorems \ref{Main theorem simpler version} and \ref{Main theorem} remain true for some explicit set of maps $L$ that has full Lebesgue measure.}
\end{Remark}

To demonstrate our approach to proving Theorem \ref{Theorem pseudorandomness}, we first give the argument under the simplifying additional assumption that $L$ is purely irrational (see Definition \ref{Definition rational space}). 

\begin{Lemma}
\label{Claim local calculation}
Suppose that $F$, $G$, $L$, $\mathbf{v}$ and the functions $w_1,\dots,w_d$ all satisfy the conditions in Definition \ref{Definition linear inequalities condition}. Suppose in addition that $L$ is surjective, purely irrational, and has algebraic coefficients. Then for all positive $K$ we have
\begin{equation}
\label{asymptotic local first}
T_{F,G,N}^{L,\mathbf{v}}(\Lambda_{\mathbb{Z}/ W_1\mathbb{Z}},\dots,\Lambda_{\mathbb{Z}/ W_d\mathbb{Z}}) =  J  + O_{K,L,P}(N^{-K}) 
\end{equation}
where $J$ is the singular integral 
\begin{equation}
\label{equation singular integral}
J:=\frac{1}{N^{d-m}}\int\limits_{\mathbf{x} \in \mathbb{R}^d}F(\mathbf{x}/N)G(L\mathbf{x} + \mathbf{v})\, d\mathbf{x}.
\end{equation} 
\end{Lemma}
\begin{proof}
We have the identity \begin{equation}
\label{equation local von identity}
\Lambda_{\mathbb{Z}/W_j\mathbb{Z}}(n) = \frac{W_j}{\varphi(W_j)}\sum\limits_{\substack{e_j\vert n\\ e_j\vert W_j}} \mu(e_j).
\end{equation} Then the expression $T_{F,G,N}^{L,\mathbf{v}}(\Lambda_{\mathbb{Z}/ W_1\mathbb{Z}},\dots,\Lambda_{\mathbb{Z}/ W_d\mathbb{Z}})$ is equal to
\begin{align}
\label{calculation of comparison term first}
& \Big(\prod\limits_{j=1}^d \frac{W_j}{\varphi(W_j)}\Big)\frac{1}{N^{d-m}} \sum\limits_{\substack{e_1,\dots,e_d\\e_j\vert W_j \,\forall j\leqslant d}} \Big(\prod\limits_{j=1}^{d}\mu(e_j)\Big)\sum\limits_{\substack{\mathbf{n}\in\mathbb{Z}^d\\ e_j\vert n_j \forall j \leqslant d }} F(\mathbf{n}/N) G(L\mathbf{n} + \mathbf{v})\nonumber\\
& = \Big(\prod\limits_{j=1}^d \frac{W_j}{\varphi(W_j)}\Big)\sum\limits_{\substack{e_1,\dots,e_d\\e_j\vert W_j \,\forall j\leqslant d}} \Big(\prod\limits_{j=1}^{d}\mu(e_j)\Big) (J (\prod\limits_{j=1}^d e_j)^{-1}  + O_{K,L,P}(N^{-K})),
\end{align}
\noindent by applying Lemma \ref{Lemma in APs} to the inner sum, where in the statement of that lemma we take $h = d$, the map $\Xi:\mathbb{R}^h \longrightarrow \mathbb{R}^d$ to be the identity, and $\widetilde{\mathbf{r}} = \mathbf{0}$. The local factor $\alpha_{\mathbf{e} ,\widetilde{\mathbf{r}}}$ is equal to $\prod_{j\leqslant d} e_j^{-1}$ in this instance.

Sum the error term in (\ref{calculation of comparison term first}) over all $e_j$. The bound $W_j = O(\log\log N)$ that comes from the prime number theorem controls the resulting error term (with room to spare), and the main term of (\ref{asymptotic local first}) follows from the identity
\begin{equation}
\label{identity mobius function}
\Big(\prod\limits_{j=1}^d \frac{W_j}{\varphi(W_j)}\Big) \sum\limits_{\substack{e_1,\dots,e_d\\e_j\vert W_j \,\forall j\leqslant d}} \Big(\prod\limits_{j=1}^{d}\frac{\mu(e_j)}{e_j}\Big) = 1.
\end{equation}
\end{proof}
To finish the proof of Theorem \ref{Theorem pseudorandomness} (in the case when $L$ is purely irrational, that is) it now suffices to show that
\begin{equation}
\label{sufficing asymptotic}
T_{F,G,N}^{L,\mathbf{v}}(f_1,\dots,f_d) = J + O_{C,L,P,\gamma}(\log ^{-\Omega(1)}N),
\end{equation}
\noindent where each $f_j$ is either $\nu_{N,w}^{\gamma}$ or $\Lambda_{\mathbb{Z}/W_j\mathbb{Z}}$. 
By multiplying out the left-hand side of (\ref{sufficing asymptotic}), we see that it is sufficient to prove that
\begin{equation}
\label{equation local or sieve first}
\frac{1}{N^{d-m}} \sum\limits_{\mathbf{n}\in \mathbb{Z}^d} \Big(\prod\limits_{j=1}^d \nu_j(n_j)\Big)F(\mathbf{n}/N)G(L\mathbf{n}+\mathbf{v}) = J + O_{L,P,\gamma}(\log ^{-\Omega(1)} N)
\end{equation}
\noindent where each $\nu_j$ equals either $c_{\rho,2}^{-1}\Lambda_{\rho,R,2}$ or  $\Lambda_{\mathbb{Z}/W_j\mathbb{Z}}$ (recall $R: = N^{\gamma}$).

After our analysis in Lemma \ref{Lemma in APs}, it turns out that the estimate (\ref{equation local or sieve first}) will follow almost immediately from the sieve calculation performed in \cite[Theorem D.3]{GT10}. To describe the details, it will be useful to introduce the following notation. Let \[ S: = \{ j \leqslant d: \nu_{j} = c_{\rho,2}^{-1}\Lambda_{\rho, R, 2}\} \] and \[ S^\prime:= \{ j \leqslant d: \nu_{j} = \Lambda_{\mathbb{Z}/W_j\mathbb{Z}}\}.\] We may assume that $S \neq \emptyset$, as otherwise the estimate (\ref{equation local or sieve first}) follows from the estimate (\ref{asymptotic local first}).

Each $\nu_j$ may be expressed as a divisor sum, either using Definition \ref{Definition Smooth sieve weight} or expression (\ref{equation local von identity}). Doing this, and swapping orders of summations, we have that the left-hand side of expression (\ref{equation local or sieve first}) is equal to 
\begin{align}
\label{massive expanded out sieve expression}
\Big(\prod\limits_{j \in S^\prime} \frac{W_j}{\varphi(W_j)}\Big)c_{\rho,2}^{-\vert S\vert}(\log  R)^{\vert S\vert} \sum\limits_{\substack{(e_{j,k})_{j\in S,k \in [2]} \in \mathbb{N}^{S \times [2]}\\ (e_{j})_{j\in S^\prime}\in \mathbb{N}
^{S^\prime} \\e_j\vert W_j \, \forall j \in S^\prime}}
&\Big(\prod\limits_{\substack{ j \in S \\ k \in [ 2]}} \mu(e_{j,k})\rho\Big(\frac{\log e_{j,k}}{\log R}\Big)\Big)\Big(\prod\limits_{j\in S^\prime} \mu(e_{j})\Big) \nonumber\\
& \frac{1}{N^{d-m}}\sum\limits_{\substack{\mathbf{n}\in \mathbb{Z}^d\\ e_{j}\vert n_j \forall j\leqslant d}} F(\mathbf{n}/N)G(L\mathbf{n} + \mathbf{v})
\end{align}

\noindent where $R = N^\gamma$, and if $j \in S$ we write $e_{j}$ for the least common multiple $[e_{j,1}, e_{j,2}]$. Using the compact support of the function $\rho$, when analysing the inner sum one may assume that each $e_j$ is at most $N^{2\gamma}$. \\

We apply Lemma \ref{Lemma in APs}. Therefore, provided $\gamma$ is small enough,
\begin{align}
\label{equation analysing inner sum}
&\frac{1}{N^{d-m}} \sum\limits_{\substack{\mathbf{n}\in \mathbb{Z}^d\\ e_j\vert n_j \,\forall j\leqslant d}}  F(\mathbf{n}/N)G(L\mathbf{n} +\mathbf{v} ) =  J\Big(\prod\limits_{j=1}^d e_j\Big)^{-1}  + O_{L,P,\gamma}(N^{-20}), 
\end{align} 
\noindent as in (\ref{calculation of comparison term first}). By the bounds on $W_j$ and $e_j$, the error term from (\ref{equation analysing inner sum}) may be summed over all $e_j$ and remain acceptable. We also have the identity \begin{equation}
\Big(\prod\limits_{j\in S^{\prime}}\frac{W_j}{\varphi(W_j)}\Big) \sum\limits_{\substack{(e_j)_{j\in S^\prime} \in \mathbb{N}^{S^\prime} \\e_j \vert W_j \, \forall j \in S^\prime}} \Big(\prod\limits_{j\in S^\prime}\frac{\mu(e_j)}{e_j}\Big) = 1.
\end{equation} Therefore expression (\ref{equation local or sieve first}) would follow from the asymptotic
\begin{align}
\label{equation rational from irrational in sieve calculation}
&(\log R)^{\vert S\vert}\sum\limits_{(e_{j,k})_{j\in S, k \in [2]} \in \mathbb{N}^{S \times [2]}}
&\Big(\prod\limits_{\substack{ j \in S \\ k \in [2]}} \mu(e_{j,k})\rho\Big(\frac{\log e_{j,k}}{\log R}\Big)\Big)\Big(\prod\limits_{j\in S} e_j\Big)^{-1}  & = c_{\rho,2}^{\vert S\vert} + O(\log^{-1/20} R).
\end{align}
\noindent But this is just expression D.4 of \cite{GT10}, applied to the  identity map $\Psi: \mathbb{R}^{\vert S\vert} \longrightarrow \mathbb{R}^{\vert S\vert}$. Note that the quantity $X$ in expression D.4 of \cite{GT10} is zero, as if $\psi_1,\dots,\psi_{\vert S\vert}:\mathbb{R}^{\vert S\vert} \longrightarrow \mathbb{R}$ are the linear maps given by $\psi_j(\mathbf{x}) : = x_j$ for all $j\leqslant \vert S\vert$ then there are no primes $p$ for which there exist two forms $\psi_i$ and $\psi_j$ that are linearly dependent modulo $p$. This proves (\ref{equation local or sieve first}), and hence resolves Theorem \ref{Theorem pseudorandomness} in the case when $L$ is purely irrational. \\

We now present the detailed proof of Theorem \ref{Theorem pseudorandomness} in full generality. 

\begin{proof}[Proof of Theorem \ref{Theorem pseudorandomness}]
Let $u$ be the rational dimension of $L$ (see Definition \ref{Definition rational space}). Apply Lemma \ref{Lemma generating a purely irrational map} to both the expression $T_{F,G,N}^{L,\mathbf{v}}(f_1,\dots,f_d)$ and the expression\\ $T_{F,G,N}^{L,\mathbf{v}}(\Lambda_{\mathbb{Z}/W_1\mathbb{Z}},\dots,\Lambda_{\mathbb{Z}/W_d\mathbb{Z}})$. Writing $h:=d-u$, where $u$ is the rational dimension of $L$, and renaming $m-u$ as $m$, $L^\prime$ as $L$, $\mathbf{v^\prime}$ as $\mathbf{v}$, and $G_{\widetilde{\mathbf{r}}}$ as $G$, we see that it suffices to prove the following theorem.

\begin{Theorem}
\label{Theorem pseudorandomness after rational reductions}
Let $N,d,h$ be natural numbers, and let $m$ be a non-negative integer. Suppose that $d\geqslant h\geqslant m+2$. Let $C,\gamma$ be positive parameters, and let $P$ be a set of additional parameters. Let $L:\mathbb{R}^h \longrightarrow \mathbb{R}^m$ be a surjective purely irrational linear map with algebraic coefficients, and let $\Xi :\mathbb{R}^{h} \longrightarrow \mathbb{R}^d$ be an injective linear map with integer coefficients. Assume that $\gamma$ is small enough in terms of $L$. Let $\mathbf{v} \in\mathbb{R}^m$ be a vector with $\Vert \mathbf{v} \Vert_\infty  \leqslant CN$, and let $\widetilde{\mathbf{r}} \in \mathbb{Z}^d$ be a vector with $\Vert \widetilde{\mathbf{r}}\Vert_\infty\leqslant CN$. Let $F:\mathbb{R}^d \longrightarrow [0,1]$ and $G: \mathbb{R}^m \longrightarrow [0,1]$ be in $\mathcal{C}(P)$. Let $w_1,\dots w_d:\mathbb{N}\longrightarrow \mathbb{R}_{\geqslant 0}$ be functions that each satisfy $w_j(n)\rightarrow \infty$ as $n\rightarrow \infty$ and $w_j(n)\leqslant w(n)$ for all $n$. \\

\noindent These conditions will be referred to as `the hypotheses of Theorem \ref{Theorem pseudorandomness after rational reductions}'.\\

\noindent Then, if $\Xi$ has finite Cauchy-Schwarz complexity,
\begin{equation}
\label{equation defining pseudorandomness after ratinoal reduction}
T_{F,G,N}^{L,\mathbf{v},\Xi,\widetilde{\mathbf{r}}}(f_1,\dots,f_d) = T_{F,G,N}^{L,\mathbf{v},\Xi,\widetilde{\mathbf{r}}}(\Lambda_{\mathbb{Z}/ W_1\mathbb{Z}},\dots,\Lambda_{\mathbb{Z}/ W_d\mathbb{Z}}) + O_{C,L,P,\Xi,\gamma}((\min_j w_j(N))^{-1/2})
\end{equation}
\noindent where each $f_j$ equals either $\nu_{N,w}^\gamma$ or $\Lambda_{\mathbb{Z}/W_j\mathbb{Z}} $. 
\end{Theorem}
\begin{proof}[Proof of Theorem \ref{Theorem pseudorandomness after rational reductions}]

Let $\Xi: \mathbb{R}^h \longrightarrow \mathbb{R}^d$ have coordinate maps $(\xi_1,\dots,\xi_d)$. Let \begin{equation}
\label{equation singular series} \mathfrak{S}_{\widetilde{\mathbf{r}}}:=\prod\limits_{p}
\frac{1}{p^h} \sum\limits_{\mathbf{m} \in [p]^h} \prod\limits_{j=1}^d \Lambda_{\mathbb{Z}/p\mathbb{Z}}(\xi_j (\mathbf{m}) +\widetilde{r}_j)
\end{equation} be the \emph{singular series}, where $\widetilde{r_j}$ denotes the $j^{th}$ coordinate of $\widetilde{\mathbf{r}}$. Let 
\begin{equation}
\label{equation singular integral again}
J_{\widetilde{\mathbf{r}}}:=\frac{1}{N^{h-m}}\int\limits_{\mathbf{x} \in \mathbb{R}^h}F((\Xi(\mathbf{x}) + \widetilde{\mathbf{r}})/N)G(L\mathbf{x} + \mathbf{v})\, d\mathbf{x}
\end{equation} be the \emph{singular integral}. 

\begin{Lemma}
\label{Lemma singular series and singular integral are bounded}
Under the hypotheses of Theorem \ref{Theorem pseudorandomness after rational reductions}, if $\Xi$ has finite Cauchy-Schwarz complexity then the singular series and singular integral satisfy the bounds \[\mathfrak{S}_{\widetilde{\mathbf{r}}} \ll_{\Xi} 1\] and \[J_{\widetilde{\mathbf{r}}} \ll_{C,L,\Xi} \Rad(F)^{h-m}\Rad(G)^m\Vert G\Vert_\infty .\]
\end{Lemma}
\noindent The reader may find the definition of $\Rad(F)$ and $\Rad(G)$ in Section \ref{section smooth functions}.
\begin{proof}
Since $\Xi$ has finite Cauchy-Schwarz complexity, no two of the forms $\xi_1,\dots,\xi_d$ are parallel. Hence by \cite[Lemma 1.3]{GT10} the singular series $\mathfrak{S}_{\widetilde{\mathbf{r}}}$ converges, and the size may be bounded by a constant depending only on $\Xi$. 

The bound on $J_{\widetilde{\mathbf{r}}}$ follows  directly from Lemma \ref{Lemma general upper bound}.
\end{proof}

We continue with the following lemma, which is a more general version of Lemma \ref{Claim local calculation}.
\begin{Lemma} 
\label{Lemma problem for local von Mangoldt}
Under the hypotheses of Theorem \ref{Theorem pseudorandomness after rational reductions} we have, for every positive real $K$,
\begin{align}
\label{asymptotic local all situations}
T_{F,G,N}^{L,\mathbf{v},\Xi,\widetilde{\mathbf{r}}}(\Lambda_{\mathbb{Z}/W_1\mathbb{Z}},\dots,\Lambda_{\mathbb{Z}/W_d\mathbb{Z}}) = \Big(\frac{1}{(\max W_j)^h} \sum\limits_{\mathbf{m} \in [\max W_j]^h} \prod\limits_{j=1}^d \Lambda_{\mathbb{Z}/W_j\mathbb{Z}}(\xi_j(\mathbf{m}) + \widetilde{r}_j) \Big) J_{\widetilde{\mathbf{r}}}\nonumber \\ + O_{C,K,L,P,\Xi}(N^{-K}).
\end{align}
\noindent If $\Xi$ has finite Cauchy-Schwarz complexity, then
\begin{equation}
\label{asymptotic local}
T_{F,G,N}^{L,\mathbf{v},\Xi,\widetilde{\mathbf{r}}}(\Lambda_{\mathbb{Z}/ W_1\mathbb{Z}},\dots,\Lambda_{\mathbb{Z}/ W_d\mathbb{Z}}) = \mathfrak{S}_{\widetilde{\mathbf{r}}} J_{\widetilde{\mathbf{r}}}  + O_{C,L,P,\Xi}((\min_j w_j(N))^{-1}).
\end{equation}
\end{Lemma}

\begin{proof}
We proceed as in the proof of Lemma \ref{Claim local calculation}. Then $T_{F,G,N}^{L,\mathbf{v},\Xi,\widetilde{\mathbf{r}}}(\Lambda_{\mathbb{Z}/ W_1\mathbb{Z}},\dots,\Lambda_{\mathbb{Z}/ W_d\mathbb{Z}})$ is equal to
\begin{align}
\label{calculation of comparison term}
& \Big(\prod\limits_{j=1}^d \frac{W_j}{\varphi(W_j)}\Big)\frac{1}{N^{h-m}} \sum\limits_{\substack{e_1,\dots,e_d\\e_j\vert W_j \,\forall j\leqslant d}} \Big(\prod\limits_{j=1}^{d}\mu(e_j)\Big)\sum\limits_{\substack{\mathbf{n}\in\mathbb{Z}^h\\ e_j\vert \xi_j(\mathbf{n}) + \widetilde{r}_j \forall j \leqslant d }} F((\Xi(\mathbf{n}) + \widetilde{\mathbf{r}})/N) G(L\mathbf{n} + \mathbf{v})\nonumber\\
& = \Big(\prod\limits_{j=1}^d \frac{W_j}{\varphi(W_j)}\Big)\sum\limits_{\substack{e_1,\dots,e_d\\e_j\vert W_j \,\forall j\leqslant d}} \Big(\prod\limits_{j=1}^{d}\mu(e_j)\Big) \alpha_{\mathbf{e},\widetilde{\mathbf{r}}} J_{\widetilde{\mathbf{r}}} + O_{C,K,L,P,\Xi}(N^{-K}),
\end{align}
\noindent by applying Lemma \ref{Lemma in APs} to the inner sum, where \[ \alpha_{\mathbf{e},\widetilde{\mathbf{r}}} = \lim\limits_{M\rightarrow \infty} \frac{1}{M^h} \sum\limits_{\mathbf{n} \in [M]^h} \prod\limits_{j=1}^d 1_{e_j\vert (\xi_j (\mathbf{n}) +\widetilde{r}_j)}. \] If $m=0$, one should apply Lemma \ref{Lemma just arithmetic progressions} in place of Lemma \ref{Lemma in APs}.

By using the identity (\ref{equation local von identity}) again one obtains 
\begin{align}
\label{first part of lemma finished}
\Big(\prod\limits_{j=1}^d \frac{W_j}{\varphi(W_j)}\Big)\sum\limits_{\substack{e_1,\dots,e_d\\e_j\vert W_j \,\forall j\leqslant d}} \prod\limits_{j=1}^{d}\mu(e_j) \alpha_{\mathbf{e},\widetilde{\mathbf{r}}}&= \lim\limits_{M\rightarrow \infty}\frac{1}{M^{h}}\sum\limits_{\mathbf{n} \in [M]^h} \prod\limits_{j=1}^d \Lambda_{\mathbb{Z}/W_j\mathbb{Z}}(\xi_j (\mathbf{n}) +\widetilde{r}_j) \nonumber \\
& = \frac{1}{(\max W_j)^h} \sum\limits_{\mathbf{n} \in [\max W_j]^h}\prod\limits_{j=1}^d \Lambda_{\mathbb{Z}/W_j\mathbb{Z}}(\xi_j (\mathbf{n}) +\widetilde{r}_j).
\end{align}
\noindent This settles the first part of the lemma. 

For the second part, by the Chinese Remainder Theorem we have that (\ref{first part of lemma finished}) is equal to 
\begin{align}
\label{equation product of local factors in local calculation}
\Big(\prod\limits_{p\leqslant \min_j w_j}\frac{1}{p^h} \sum\limits_{\mathbf{m} \in [p]^h} \prod\limits_{j\leqslant d}& \Lambda_{\mathbb{Z}/p\mathbb{Z}}(\xi_j(\mathbf{m})+ \widetilde{r}_j)\Big) \nonumber \\
&\times \Big(\prod\limits_{\min_j w_j < p \leqslant \max_j w_j} \frac{1}{p^h} \sum\limits_{\mathbf{m} \in [p]^h} \prod\limits_{j\leqslant d }^* \Lambda_{\mathbb{Z}/p\mathbb{Z}}(\xi_j(\mathbf{m})+ \widetilde{r}_j)\Big), 
\end{align} 
\noindent where $\prod^*$ denotes the product over those $j\leqslant d$ for which $p\leqslant w_j$. 

Since $\Xi$ has finite Cauchy-Schwarz complexity there is no pair of forms $\xi_i$ and $\xi_j$ that are parallel. Therefore we may apply the analysis of local factors in \cite[Lemma 1.3]{GT10} to conclude that the first bracket in (\ref{equation product of local factors in local calculation}) is equal to $\mathfrak{S}_{\widetilde{\mathbf{r}}}(1 + O_{\Xi}((\min_j w_j(N))^{-1}))$, and that the second bracket is equal to $(1+ O_{\Xi}((\min_j w_j(N))^{-1})$. Combining these bounds with Lemma \ref{Lemma singular series and singular integral are bounded} gives the second part of the present lemma. 
\end{proof}
\begin{Remark}
\label{Remark asymptotic for local von mangoldt in general}
\emph{As we intimated earlier, in Remark \ref{Remark after statement of Main theorem}, one can use Lemma \ref{Lemma problem for local von Mangoldt} to establish an asymptotic expression for $T_{F,G,N}^{L,\mathbf{v}}(\Lambda_{\mathbb{Z}/W\mathbb{Z}},\dots,\Lambda_{\mathbb{Z}/W\mathbb{Z}})$ in the general case. Indeed, one applies the rational parametrisation process of Lemma \ref{Lemma generating a purely irrational map} and then the asymptotic in Lemma \ref{Lemma problem for local von Mangoldt} to obtain} \[T_{F,G,N}^{L,\mathbf{v}}(\Lambda_{\mathbb{Z}/W\mathbb{Z}},\dots,\Lambda_{\mathbb{Z}/W\mathbb{Z}}) = \sum\limits_{\widetilde{\mathbf{r}} \in \widetilde{R}} \mathfrak{S}_{\widetilde{\mathbf{r}}} J_{\widetilde{\mathbf{r}}} + o_{C,L,P}(1).\] 
\end{Remark}

Now, Theorem \ref{Theorem pseudorandomness after rational reductions} will be settled if we can show that the left-hand side of (\ref{equation defining pseudorandomness after ratinoal reduction}) enjoys the same asymptotic expression as the one present in (\ref{asymptotic local}). By multiplying out the left-hand side of (\ref{equation defining pseudorandomness after ratinoal reduction}), we see that it is sufficient to prove the following lemma.
\begin{Lemma}
\label{Lemma sieve weight asymptotics in general case} Under the hypotheses of Theorem \ref{Theorem pseudorandomness after rational reductions}, if $\Xi$ has finite Cauchy-Schwarz complexity then
\begin{equation}
\label{equation local or sieve}
\frac{1}{N^{h-m}} \sum\limits_{\mathbf{n}\in \mathbb{Z}^h} \Big(\prod\limits_{j=1}^d \nu_j(n_j)\Big)F((\Xi(\mathbf{n}) + \widetilde{\mathbf{r}})/N)G(L\mathbf{n}+\mathbf{v}) =  \mathfrak{S}_{\widetilde{\mathbf{r}}} J_{\widetilde{\mathbf{r}}} + O_{C,L,P,\Xi,\gamma}((\min_j w_j(N))^{-1/2}),
\end{equation}
\noindent where each $\nu_j$ equals either $c_{\rho,2}^{-1}\Lambda_{\rho,R,2}$ or  $\Lambda_{\mathbb{Z}/W_j\mathbb{Z}}.$ 
\end{Lemma}
\begin{proof}[Proof of Lemma]
The first half of the proof of this lemma comprises manipulations that are very similar to those that have appeared previously in this section. Indeed, as before, it will be useful to let \[ S: = \{ j \in [d]: \nu_{j} = c_{\rho,2}^{-1}\Lambda_{\rho, R, 2}\}\] and \[ S^\prime:= \{ j \leqslant d: \nu_{j} = \Lambda_{\mathbb{Z}/W_j\mathbb{Z}}\}.\] We may assume that $S \neq \emptyset$, as otherwise the estimate (\ref{equation local or sieve}) follows from Lemma \ref{Lemma problem for local von Mangoldt}.

Considering (\ref{equation local von identity}) again, and expressing each $\nu_j$ as a divisor sum, we have that the left-hand side of expression (\ref{equation local or sieve}) is equal to 
\begin{align}
\label{massive expanded out sieve expression again}
\Big(\prod\limits_{j \in S^\prime} \frac{W_j}{\varphi(W_j)}\Big)c_{\rho,2}^{-\vert S\vert }(\log  R)^{\vert S\vert} \sum\limits_{\substack{(e_{j,k})_{j \in S, k \in [2]} \in \mathbb{N}^{S \times [2]}\\ (e_{j})_{j\in S^\prime} \in \mathbb{N}^{S^\prime} \\e_j\vert W_j \, \forall j \in S^\prime}}
&\Big(\prod\limits_{\substack{ j \in S \\ k \in [2]}}\mu(e_{j,k})\rho\Big(\frac{\log e_{j,k}}{\log R}\Big)\Big)\Big(\prod\limits_{j\in S^\prime} \mu(e_{j})\Big) \nonumber\\
& \frac{1}{N^{h-m}}\sum\limits_{\substack{\mathbf{n}\in \mathbb{Z}^h\\ e_{j}\vert \xi_j(\mathbf{n}) + \widetilde{r}_j \forall j\leqslant d}} F((\Xi(\mathbf{n}) + \widetilde{\mathbf{r}})/N)G(L\mathbf{n} + \mathbf{v})
\end{align}

\noindent where if $ j \in S$ we write $e_j$ for the least common multiple $[e_{j,1}, e_{j,2}]$. Using the compact support of the function $\rho$, when analysing the inner sum one may assume that each $e_j$ is at most $N^{2\gamma}$.

We apply Lemma \ref{Lemma in APs} (or, if $m=0$, we apply Lemma \ref{Lemma just arithmetic progressions}). Therefore 
\begin{align}
\label{equation analysing inner sum again}
&\frac{1}{N^{h-m}} \sum\limits_{\substack{\mathbf{n}\in \mathbb{Z}^d\\ e_j\vert \xi_j(\mathbf{n}) + \widetilde{r}_j\,\forall j\leqslant d}}  F((\Xi(\mathbf{n}) + \widetilde{\mathbf{r}})/N)G(L\mathbf{n} +\mathbf{v} ) = \alpha_{\mathbf{e},\widetilde{\mathbf{r}}} J_{\widetilde{\mathbf{r}}} + O_{C,L,P,\Xi}(N^{-20}),
\end{align} where \begin{equation}
\label{local factor def again}
\alpha_{\mathbf{e},\widetilde{\mathbf{r}}} := \lim\limits_{M\rightarrow \infty}\frac{1}{M^{h}}\sum\limits_{\mathbf{m}\in [M]^h} \prod\limits_{j=1}^d 1_{e_j\vert (\xi_j(\mathbf{m}) + \widetilde{r}_j)}
\end{equation}

Therefore expression (\ref{equation local or sieve}) (and hence the entirety of Theorem \ref{Theorem pseudorandomness after rational reductions}) would follow from the asymptotic expression
\begin{align}
\label{equation rational from irrational in sieve calculation again}
&\Big(\prod\limits_{j \in S^\prime} \frac{W_j}{\varphi(W_j)}\Big)(\log R)^{\vert S\vert} \sum\limits_{\substack{(e_{j,k})_{j \in S, k \in [2]} \in \mathbb{N}^{S \times [2]}\\(e_{j})_{j\in S^\prime} \in \mathbb{N}^{S^\prime} \\ e_j \vert W_j \, \forall j \in S^\prime}}
&\Big(\prod\limits_{\substack{j \in S \\ k \in[2]}} \mu(e_{j,k})\rho\Big(\frac{\log e_{j,k}}{\log R}\Big)\Big)\Big(\prod\limits_{j\in S^\prime} \mu(e_{j})\Big) \alpha_{\mathbf{e},\widetilde{\mathbf{r}}}\nonumber \\
& = c_{\rho,2}^{\vert S\vert }\mathfrak{S}_{\widetilde{\mathbf{r}}} + O((\min_j w_j(N))^{-1/2}).
\end{align}
\noindent Note that this expression concerns linear forms with integer coefficients. We have removed the irrational information entirely. \\ 

Expression (\ref{equation rational from irrational in sieve calculation again}) follows from the sieve calculation \cite[Theorem D.3]{GT10}, after restricting to suitable arithmetic progressions. Indeed,  let \[ A: = \{ \mathbf{a} \in [W]^h: ((\xi_j(\mathbf{a}) + \widetilde{r}_j),W_j) = 1 \, \forall j \in S^\prime\}.\] Then the left-hand side of (\ref{equation rational from irrational in sieve calculation again}) is equal to \begin{align}
\label{split into arithmetic progressions}
&\Big(\prod\limits_{j \in S^\prime} \frac{W_j}{\varphi(W_j)}\Big)\frac{1}{W^h}\sum\limits_{\mathbf{a} \in A}(\log  R)^{\vert S\vert}\nonumber\\
&\sum\limits_{(e_{j,k})_{j \in S, k \in [2]} \in \mathbb{N}^{S \times [2]}}\Big(\prod\limits_{\substack{ j \in S \\ k \in [2]}} \mu(e_{j,k})\rho\Big(\frac{\log e_{j,k}}{\log R}\Big)\Big)\lim\limits_{M\rightarrow \infty} \frac{1}{M^h} \sum\limits_{\mathbf{m}\in [M]^h} \prod\limits_{j \in S} 1_{e_j\vert (\xi_j(W\mathbf{m} + \mathbf{a}) + \widetilde{r}_j)}.
\end{align}

The expression following the summation in $\mathbf{a}$ is amenable to the estimate (D.4) from \cite{GT10}, applied with $t = \vert S\vert$ and affine linear forms \[ \psi_j(\mathbf{m}) : = \xi_j(W\mathbf{m} + \mathbf{a}) + \widetilde{r}_j ,\qquad j \in S.\] In order to apply this estimate we note first that $t\neq 0$ (since we have previously assumed that $S \neq \emptyset$). We also note again that, by the finite Cauchy-Schwarz complexity assumption, no two of the forms $\psi_j$ are rational multiples of each other.

So, applying the estimate (D.4) from \cite{GT10} we have that the expression in (\ref{split into arithmetic progressions}) following the summation in $\mathbf{a}$ is equal to
\begin{align}
\label{application of D4}
c_{\rho,2}^{\vert S\vert}\prod\limits_{p}\beta_{p,\mathbf{a}} + O_{C,\Xi,\gamma}(e^{O(X)} \log^{-\frac{1}{20}} R),
\end{align}
where \[ \beta_{p,\mathbf{a}} = \frac{1}{p^h}\sum\limits_{\mathbf{m} \in [p]^h} \prod\limits_{j \in S} \Lambda_{\mathbb{Z}/p\mathbb{Z}} (\xi_j(W\mathbf{m} +\mathbf{a}) + \widetilde{r}_j),\] and \[ X : = \sum\limits_{p\in P_{\Xi}} p^{-1/2},\] where $P_{\Xi}$ is the set of `exceptional' primes, i.e. those primes $p$ for which there exist $i$ and $j$ for which the forms $\xi_i(W\mathbf{m} +\mathbf{a}) + \widetilde{r}_i$ and $\xi_j(W\mathbf{m} +\mathbf{a}) + \widetilde{r}_j$ are affinely related modulo $p$.
 \begin{Remark}
 \emph{The reader may have noticed that expression (\ref{application of D4}) is not exactly what was proved in estimate (D.4) of \cite{GT10}. Rather than having an error term depending on $\Xi$ and $C$, that expression has an error term depending on the linear maps $\mathbf{m} \mapsto \xi_j(W\mathbf{m} + \mathbf{a}) + \widetilde{r}_j$ which, one notes, have coefficients that depend on $W$ and that are therefore unbounded. Fortunately, the dependence of the error term on the size of the coefficients is only polynomial, and so any contribution from powers of $W$ may be absorbed into the $\log ^{-\frac{1}{20}} R$ factor. }
 
 \emph{This technical manoeuvre is also required in \cite{GT10} (in the application of Theorem D.3 that follows expression (D.24)), although it is not explicitly stated by the authors.}  
 \end{Remark}
Following on from (\ref{application of D4}) and assuming that $N$ is large enough in terms of $\Xi$, we see that any $p \in P_\Xi$ satisfies $p \leqslant w$ (as $\Xi$ has finite Cauchy-Schwarz complexity). Since $w(N) = \max(1,\log\log\log N)$, the error in (\ref{application of D4}) is therefore $O_{C,\Xi,\gamma}( \log ^{-\Omega(1)} N)$. Furthermore, by \cite[Lemma 1.3]{GT10} we have $\beta_{p,\mathbf{a}} = 1+O(p^{-2})$, and so $\prod\limits_{p>w} \beta_{p,\mathbf{a}} = 1 + O(w^{-1})$. Finally, if $p\leqslant w$ then \[\beta_{p,\mathbf{a}} = \prod\limits_{j \in S} \Lambda_{\mathbb{Z}/p\mathbb{Z}} (\xi_j(\mathbf{a}) + \widetilde{r}_j).\] Therefore expression (\ref{split into arithmetic progressions}), up to an error term of $O_{C,\Xi,\gamma}(w^{-1/2})$, is equal to \begin{align}
\label{long series of local deductions}
& c_{\rho,2}^{\vert S\vert}\Big(\prod\limits_{j\in S^\prime} \frac{W_j}{\varphi(W_j)}\Big) \frac{1}{W^h}\sum\limits_{\mathbf{a}\in A} \prod\limits_{p\leqslant w} \beta_{p,\mathbf{a}}\nonumber\\
 = &\, c_{\rho,2}^{\vert S\vert} \Big(\prod\limits_{j\in S^\prime} \frac{W_j}{\varphi(W_j)}\Big) \frac{1}{W^h} \sum\limits_{\mathbf{a}\in A} \prod\limits_{j \in S}\Lambda_{\mathbb{Z}/W\mathbb{Z}}(\xi_j(\mathbf{a}) + \widetilde{r}_j) \nonumber \\
 = & \,c_{\rho,2}^{\vert S\vert}\frac{1}{W^h} \sum\limits_{\mathbf{a} \in [W]^h} \prod\limits_{j \in S}\Lambda_{\mathbb{Z}/W\mathbb{Z}}(\xi_j(\mathbf{a}) + \widetilde{r}_j)\prod\limits_{j\in S^\prime} \Lambda_{\mathbb{Z}/W_j\mathbb{Z}}(\xi_j(\mathbf{a}) + \widetilde{r}_j)\nonumber\\
= & \,c_{\rho,2}^{\vert S\vert}\prod\limits_{p\leqslant \min_j w_j}\frac{1}{p^h} \sum\limits_{\mathbf{m} \in [p]^h} \prod\limits_{j\leqslant d} \Lambda_{\mathbb{Z}/p\mathbb{Z}}(\xi_j(\mathbf{m}) + \widetilde{r}_j) \times \prod\limits_{\min_j w_j < p \leqslant w} \widetilde{\beta_p},
\end{align}
\noindent where \[ \widetilde{\beta_p}: = \frac{1}{p^h}\sum\limits_{\mathbf{m} \in [p]^h}\prod\limits_{j \in S}\prod\limits_{j\in S^\prime}^*\Lambda_{\mathbb{Z}/p\mathbb{Z}}(\xi_j(\mathbf{m}) + \widetilde{r}_j),\] where $\prod^*$ denotes the product over all $j\in S^\prime$ for which $p\leqslant w_j$.

By invoking \cite[Lemma 1.3]{GT10} again we conclude that $\widetilde{\beta_p} = 1+O(p^{-2})$ and also that the first part of expression (\ref{long series of local deductions}) is equal to $c_{\rho,2}^{\vert S\vert}\mathfrak{S}_{\widetilde{\mathbf{r}}}(1+O_{C,\Xi}(\min_j w_j^{-1}))$. Hence, as in the conclusion of the proof of Lemma \ref{Lemma problem for local von Mangoldt}, we conclude that expression (\ref{long series of local deductions}) is equal to $c_{\rho,2}^{\vert S\vert}\mathfrak{S}_{\widetilde{\mathfrak{r}}} + O_{C,\Xi,\gamma}(\min_j w_j^{-1/2})$. This establishes expression (\ref{equation rational from irrational in sieve calculation again}), and so Lemma \ref{Lemma sieve weight asymptotics in general case} is proved. 
\end{proof}
Therefore Theorem \ref{Theorem pseudorandomness after rational reductions} is resolved. 
\end{proof}
Hence Theorem \ref{Theorem pseudorandomness} is settled as well, i.e. we conclude that the weight $\nu_{N,w}^\gamma$ is $(L,w)$-pseudorandom. 
\end{proof}

We finish this section by noting a corollary of the theorems above, which will be useful in its own right. 

\begin{Corollary}[Upper bound for linear inequalities]
\label{Corollary upper bound}
Let $N,m,d$ be natural numbers, with $d\geqslant m+2$, and let $C,\varepsilon,\gamma$ be positive reals. Let $L:\mathbb{R}^d \longrightarrow \mathbb{R}^m$ be a surjective linear map with algebraic coefficients, and suppose that $L\notin V_{\degen}^*(m,d)$ and that the coefficients of $L$ are algebraic. Let $u$ be the rational dimension of $L$. Let $w_1,\dots w_d:\mathbb{N}\longrightarrow \mathbb{R}_{\geqslant 0}$ be functions that satisfy $w_j(n)\rightarrow \infty$ as $n\rightarrow \infty$ for all $j$ and satisfy $w_j(n)\leqslant w(n)$ for all $j$ and for all $n$. If $\gamma$ is small enough in terms of $L$, then for all functions $F:\mathbb{R}^d \longrightarrow [0,1]$ supported on $[-C,C]^d$, for all functions $G:\mathbb{R}^m \longrightarrow [0,1]$ supported on $[-\varepsilon,\varepsilon]^m$, and for all $\mathbf{v}\in\mathbb{R}^m$ satisfying $\Vert \mathbf{v} \Vert_\infty \leqslant CN$,  one has \[ T_{F,G,N}^{L,\mathbf{v}}(f_1,\dots,f_d) \ll_{C,L} \Vert G\Vert_\infty \varepsilon^{m-u} + o(1)\] as $N\rightarrow \infty$, where each $f_j$ equals either $\nu_{N,w}^\gamma$ or $\Lambda_{\mathbb{Z}/W_j\mathbb{Z}} $. The $o(1)$ term may depend on $C$, $L$,  $\varepsilon$, $\gamma$, and the choice of functions $w_1,\dots,w_d$.
\end{Corollary}
\begin{proof}
Using Lemma \ref{Lemma smooth approximations}, replace both $F$ and $G$ by compactly supported smooth majorants $F_1$ and $G_1$ for which \[ 1_{[-C_1,C_1]^d} \leqslant F_1 \leqslant 1_{[-2C_1,2C_1]^d}\] and \[ 1_{[-\varepsilon,\varepsilon]^m} \leqslant G_1 \leqslant 1_{[-2\varepsilon,2\varepsilon]^m}.\]  We have $F_1 \in \mathcal{C}(C_1)$ and $G_1 \in \mathcal{C}(\varepsilon)$. Then, by Theorem \ref{Theorem pseudorandomness}, 
\begin{align*}
T_{F,G,N}^{L,\mathbf{v}}(f_1,\dots,f_d) &\ll T_{F_1,G_1,N}^{L,\mathbf{v}}(f_1,\dots,f_d) \\
& = T_{F_1,G_1,N}^{L,\mathbf{v}}(\Lambda_{\mathbb{Z}/W_1\mathbb{Z}},\dots,\Lambda_{\mathbb{Z}/W_d\mathbb{Z}}) + o(1),
\end{align*} 
where the error term may depend on $C$, $L$, $\varepsilon$, $\gamma$, and the functions $w_1,\dots,w_d$.

In Remark \ref{Remark asymptotic for local von mangoldt in general}, we noted that \[T_{F_1,G_1,N}^{L,\mathbf{v}}(\Lambda_{\mathbb{Z}/W_1\mathbb{Z}},\dots,\Lambda_{\mathbb{Z}/W_d\mathbb{Z}}) = \sum\limits_{\widetilde{\mathbf{r}} \in \widetilde{R}} \mathfrak{S}_{\widetilde{\mathbf{r}}} J_{\widetilde{\mathbf{r}}} + o(1),\] where the error term depends on the parameters mentioned above, and where $\mathfrak{S}_{\widetilde{\mathbf{r}}}$ and $J_{\widetilde{\mathbf{r}}}$ are of the form (\ref{equation singular series}) and (\ref{equation singular integral again}). The corollary then follows from the bounds in Lemma~\ref{Lemma singular series and singular integral are bounded}. 
\end{proof}

This result is to be compared with the following statement.   

\begin{Lemma}[Weak upper bound]
\label{Lemma trivial upper bound}
Let $N,m,d$ be natural numbers, with $d\geqslant m$, and let $C,\varepsilon$ be positive parameters. Let $L:\mathbb{R}^d \longrightarrow \mathbb{R}^m$ be a surjective linear map. Then, for all functions $F: \mathbb{R}^d \longrightarrow [0,1]$ supported on $[-C,C]^d$, for all functions $G: \mathbb{R}^m \longrightarrow [0,1]$ supported on $[-\varepsilon,\varepsilon]^m$, for all $\mathbf{v} \in \mathbb{R}^m$, and for all functions $f_1,\dots,f_d: \mathbb{Z} \longrightarrow \mathbb{R}$, \[ T_{F,G,N}^{L,\mathbf{v}}(f_1,\dots,f_d) \ll_{C, L,\varepsilon} \Vert G\Vert_\infty \sup\limits_{\substack{j\leqslant d \\ \vert n\vert \leqslant \Rad(F)N}} \vert f_j(n)\vert^d .\]
\end{Lemma}
\noindent The bound in Lemma \ref{Lemma trivial upper bound} is weaker than the bound in Corollary \ref{Corollary upper bound}, but has the advantage of holding for all surjective maps $L$, which is a situation that will be needed later.\\ 
\begin{proof}
This is essentially identical to Lemma 3.2 of \cite{Wa17}. Indeed, one sees immediately that \[ T_{F,G,N}^{L,\mathbf{v}}(f_1,\dots,f_d) \ll \frac{1}{N^{d-m}}\Big(\sup\limits_{\substack{j\leqslant d \\ \vert n\vert \leqslant \Rad(F)N}} \vert f_j(n)\vert^d \Big) \times \sum\limits_{\substack{ \mathbf{n} \in [-CN,CN]^d \\ \Vert L\mathbf{n} + \mathbf{v}\Vert_\infty \leqslant \varepsilon}} 1.\] Since $L$ is surjective, without loss of generality we may assume that the first $m$ columns of $L$ form an invertible matrix. If the variables $n_{m+1}$ to $n_d$ are fixed, there are only $O_{\varepsilon,L}(1)$ possible choices for $n_1 ,\dots,n_m$ for which the inequality $\Vert L\mathbf{n} + \mathbf{v} \Vert_\infty \leqslant \varepsilon$ is satisfied. Summing over $n_{m+1}$ to $n_{d}$, the lemma follows. 
\end{proof}

\part{The structure of inequalities}
\label{part the structure of inequalities}

Before embarking upon this part of the argument, we remind the reader of the following basic notion from functional analysis. A linear map $L:(V, \Vert \cdot \Vert_V) \longrightarrow (W, \Vert \cdot \Vert_W)$ between two normed spaces will be called a \emph{bounded operator} if there exists a constant $C_L$ such that for all $\mathbf{v} \in V$ one has $\Vert L \mathbf{v}\Vert_{W} \leqslant C_L \Vert \mathbf{v} \Vert _V$. It is a standard fact that all linear maps between two finite dimensional normed spaces are bounded. \\

\section{An alternative formulation}
So far all of our theorems and lemmas have been phrased in terms of linear inequalities that are written in the form $T_{F,G,N}^{L,\mathbf{v}}(f_1,\dots,f_d)$. In Section \ref{section General proof of the real variable von Neumann Theorem} the auxiliary inequalities will appear in a different form, but, as is shown in Lemma \ref{Lemma different forms of inequalities} below,  these different forms are more-or-less equivalent. The statement of this lemma is unfortunately rather technical, but the proof is straightforward. The reader may wish in the first instance to consider the special case in which $l = 0$ and $\Phi$ is injective. 

\begin{Lemma}[Alternative formulation]
\label{Lemma different forms of inequalities}
Let $m,d,l$ be natural numbers, with $d\geqslant m$, and let $C,\sigma,\eta$ be positive parameters. Let $P$ be another set of parameters. Let $k$ be a non-negative integer, and suppose that $\eta$ is small enough in terms of $m$, $d$, $k$ and $l$. Let $\Phi:\mathbb{R}^{d-m+k}\longrightarrow \mathbb{R}^d$ be a linear map, and suppose that $k = \dim \ker \Phi$. Let $I:\mathbb{R}^d \longrightarrow [0,1]$ and $H:\mathbb{R}^{d-m+k +l} \longrightarrow [0,1]$ be smooth functions, where $\Rad(I) \leqslant \eta$ and $\Rad(H) \leqslant C$. Assume that the Lipschitz constant of $H$ is at most $\sigma^{-1}$ and assume further that $H,I \in \mathcal{C}(P)$. Then 

\begin{enumerate}
\item there exists a surjective linear map $L:\mathbb{R}^d \longrightarrow \mathbb{R}^m$ such that $\ker L = \im \Phi$ and $\Vert L\Vert_\infty  = O_{\Phi}(1)$. If $\Phi$ has algebraic coefficients then $L$ can be chosen to have algebraic coefficients. 
\item for any $L$ satisfying part (1), if $\Phi$ has finite Cauchy-Schwarz complexity then $L \notin V_{\degen}^*(m,d)$.
\item for any $L$ satisfying part (1), if $\eta$ is small enough in terms of $L$ and $\Phi$ then there exist smooth functions $F:\mathbb{R}^{d+l} \longrightarrow \mathbb{R}_{\geqslant 0}$ and $G:\mathbb{R}^{m} \longrightarrow \mathbb{R}_{\geqslant 0}$, with $F \in \mathcal{C}(P,\Phi)$, $G \in\mathcal{C}(L,P,\Phi)$ and $\Rad(G) \ll_L\eta $, such that for all $\mathbf{v} \in \mathbb{R}^l$, $\mathbf{z} \in \mathbb{R}^d$, and natural numbers $N$,
\begin{equation}
\label{equation different forms of inequalities}
\frac{1}{N^k} \int\limits_{\mathbf{x} \in \mathbb{R}^{d-m+k}} I(\mathbf{z} - \Phi(\mathbf{x}))H((\mathbf{v},\mathbf{x})/N) \, d\mathbf{x} =  F((\mathbf{v},\mathbf{z})/N) G(L\mathbf{z}) + E(\mathbf{z},N),\end{equation}
\noindent where $E(\mathbf{z},N)$ is an error term of size at most \[
O_{C,\Phi}(\eta \sigma^{-1} N^{-1} 1_{[0,O_{C,\Phi}(N)]}(\Vert \mathbf{z}\Vert_\infty )1_{[0,O_L(\eta)]} (\Vert L\mathbf{z} \Vert_\infty)).\] 
\end{enumerate}
\end{Lemma}

\begin{proof}
Part (1) of the lemma is immediate. Indeed, one has the quotient map $\pi:\mathbb{R}^d \longrightarrow \mathbb{R}^d/ \im \Phi$. Choosing an isomorphism $\iota: \mathbb{R}^d/ \im \Phi \cong \mathbb{R}^m$, we may define $L:= \iota \circ \pi$. If $\Phi$ has algebraic coefficients then choosing such an $\iota$ with algebraic coefficients gives a suitable $L$ with algebraic coefficients. 

For part (2), suppose that $\Phi$ has finite Cauchy-Schwarz complexity. If $L$ were in $V_{\degen}^*(m,d)$ then there would exist $i,j\leqslant d$ and a real number $\lambda$ for which $\mathbf{e_i} - \lambda \mathbf{e_j}$ is non zero and $\mathbf{e_i^*} - \lambda\mathbf{e_j^*} \in L^*((\mathbb{R}^{m})^*)$, which would imply $\mathbf{e_i^*} - \lambda\mathbf{e_j^*} \in (\im \Phi)^0$, which would imply that $\Phi$ has infinite Cauchy-Schwarz complexity, contradicting the hypothesis. \\

It remains to prove part (3). Let $\{\mathbf{u^{(1)}},\dots,\mathbf{u^{(k)}}\} \subset \mathbb{R}^{d-m+k}$ be an orthonormal basis for $\ker \Phi$, and extend this to an orthonormal basis $\{\mathbf{u^{(1)}},\dots,\mathbf{u^{(d-m+k)}}\}$ for $\mathbb{R}^{d-m+k}$. Then define the linear map $\Psi:\mathbb{R}^{d-m+k}\longrightarrow \mathbb{R}^{d-m+k}$ by \[\Psi(\mathbf{y}) =  \sum\limits_{j=1}^{d-m+k} y_j\mathbf{u^{(j)}}.\] By changing variables, we have that the left-hand side of (\ref{equation different forms of inequalities}) is equal to \[\frac{1}{N^k}\int\limits_{\mathbf{y}\in \mathbb{R}^{d-m+k}} I(\mathbf{z} - \Phi(\Psi(\mathbf{y})))H((\mathbf{v},\Psi(\mathbf{y}))/N)\, d\mathbf{y},\] which equals
\begin{equation}
\label{after second rescaling}
\frac{1}{N^k}\int\limits_{\mathbf{y}\in \mathbb{R}^{d-m+k}} I(\mathbf{z} - \Phi(\Psi(\mathbf{0},\mathbf{y_{k+1}^{d-m+k}})))H((\mathbf{v},(\Psi(\mathbf{y_1^k},\mathbf{0}) + \Psi(\mathbf{0},\mathbf{y_{k+1}^{d-m+k}})))/N) \, d\mathbf{y}.
\end{equation}
\noindent Recall, from Section \ref{section conventions}, that we use the notation $\mathbf{y_1^k}$ to refer to the vector $(y_1,\dots,y_k)^T \in \mathbb{R}^k$, etcetera.\\

We make some observations. Firstly, we observe that (\ref{after second rescaling}) is equal to $0$ unless $\Vert \mathbf{z} \Vert_\infty = O_{C,\Phi}(N)$. Indeed, if $\Vert z \Vert_\infty \geqslant C_1N$ then for all $y_{k+1},\dots,y_{d-m+k}$ that give a non-zero contribution to (\ref{after second rescaling}) we have \[\Vert \Phi(\Psi(\mathbf{0}, \mathbf{y_{k+1}^{d-m+k}}))\Vert_\infty \geqslant \frac{1}{2} C_1 N,\] if $\eta$ is small enough. This means that \[\Vert \Psi(\mathbf{0}, \mathbf{y_{k+1}^{d-m+k}})\Vert_\infty \gg_{\Phi} C_1 N,\] which if $C_1$ is large enough in terms of $C$ and $\Phi$ means that \[ H((\mathbf{v},(\Psi(\mathbf{y_1^k},\mathbf{0}) + \Psi(\mathbf{0},\mathbf{y_{k+1}^{d-m+k}})))/N) = 0\] for all $y_1,\dots,y_k$. [Note that $\Psi(\mathbf{y_1^k},\mathbf{0})$ and $\Psi(\mathbf{0},\mathbf{y_{k+1}^{d-m+k}})$ are orthogonal.]

Secondly, we observe that \[\Vert \mathbf{z} - \Phi(\Psi(\mathbf{0},\mathbf{y_{k+1}^{d-m+k}}))\Vert_\infty \ll \eta\] for all $y_{k+1},\dots,y_{d-m+k}$ that give a non-zero contribution to the integral (\ref{after second rescaling}). Write $\mathbf{z} = \mathbf{z_1} + \mathbf{z_2}$, where $\mathbf{z_1} \in \im \Phi$ and $\mathbf{z_2}\in (\im \Phi)^\perp$. By orthogonality, we conclude that \[\Vert \mathbf{z_1} - \Phi(\Psi(\mathbf{0},\mathbf{y_{k+1}^{d-m+k}}))\Vert_\infty \ll \eta.\] Since $(\Phi|_{(\ker \Phi)^\perp})^{-1}:\im \Phi \longrightarrow (\ker \Phi)^\perp$ is a bounded linear map, this in turn means that \[ \Vert (\Phi|_{(\ker \Phi)^\perp})^{-1}(\mathbf{z_1}) - \Psi(\mathbf{0},\mathbf{y_{k+1}^{d-m+k}}))\Vert_\infty\ll_\Phi \eta.\] Since $H$ is Lipschitz, with Lipschitz constant at most $\sigma^{-1}$, this all means that (\ref{after second rescaling}) is equal to 
\begin{align}
\label{two brackets}
&\Big(\int\limits_{\mathbf{y_{k+1}^{d-m+k}} \in \mathbb{R}^{d-m}} I(\mathbf{z} - \Phi(\Psi(\mathbf{0},\mathbf{y_{k+1}^{d-m+k}}))) \, d\mathbf{y_{k+1}^{d-m+k}}\Big)\nonumber \\
&\times \frac{1}{N^k}\Big(\int\limits_{\mathbf{y_1^k}\in \mathbb{R}^k}H((\mathbf{v},(\Psi(\mathbf{y_1^k},\mathbf{0})  + (\Phi|_{(\ker \Phi)^\perp})^{-1}(\mathbf{z_1})))/N) \,d\mathbf{y_1^k}\Big),
\end{align}
\noindent plus an error of size at most \[O_{C,\Phi}(\eta \sigma^{-1} N^{-1} 1_{[0,O_{C,\Phi}(N)]}(\Vert \mathbf{z}\Vert_\infty)1_{[0,O(\eta)]}(\dist(\mathbf{z}, \im \Phi))).\]

We proceed to analyse the terms of (\ref{two brackets}) separately. Firstly, by shifting the variables $y_{k+1},\dots,y_{d-m+k}$ we see that the first bracket of (\ref{two brackets}) is equal to 
\begin{equation}
\label{first bracket}
\int\limits_{\mathbf{y_{k+1}^{d-m+k}} \in \mathbb{R}^{d-m}} I(\mathbf{z_2} - \Phi(\Psi(\mathbf{0},\mathbf{y_{k+1}^{d-m+k}}))) \, d\mathbf{y_{k+1}^{d-m+k}}. 
\end{equation}  
\noindent Now let $L:\mathbb{R}^d \longrightarrow \mathbb{R}^{m}$ be any surjective linear map that satisfies $\ker L = \im \Phi$. Note that $L|_{(\im \Phi)^\perp}$ is an injective linear map, and thus (\ref{first bracket}) is equal to \[\int\limits_{\mathbf{y_{k+1}^{d-m+k}} \in \mathbb{R}^{d-m}}I((L|_{(\im \Phi)^\perp})^{-1}L\mathbf{z_2} -  \Phi (\Psi(\mathbf{0},\mathbf{y_{k+1}^{d-m+k}}))) \, d \mathbf{y_{k+1}^{d-m++k}}.\] Differentiating inside the integral, one sees that this expression is equal to $G(L\mathbf{z_2})$ for some smooth compactly supported function $G:\mathbb{R}^{m}\longrightarrow \mathbb{R}_{\geqslant 0}$ satisfying $G \in \mathcal{C}(L,P,\Phi)$. Moreover, $G$ is supported on $[O_L(\eta),O_L(\eta)]^m$, since $(L|_{(\im \Phi)^\perp})^{-1}L\mathbf{z_2}$ and $\Phi (\Psi(\mathbf{0},\mathbf{y_{k+1}^{d-m+k}}))$ are orthogonal. Note that $L\mathbf{z_2} = L\mathbf{z}$, so the expression is equal to $G(L\mathbf{z})$.\\

We move to the second term of (\ref{two brackets}). Choose $\iota: \im \Phi \longrightarrow \mathbb{R}^{d-m}$ to be an isomorphism with $\Vert \iota(\mathbf{x})\Vert_\infty \asymp _{\Phi} \Vert\mathbf{x}\Vert_\infty$. Then the second term of (\ref{two brackets}) is equal to $F_1((\mathbf{v},\iota(\mathbf{z_1}))/N)$ for some smooth function $F_1:\mathbb{R}^{d-m+l}\longrightarrow \mathbb{R}_{\geqslant 0}$ satisfying $F_1 \in \mathcal{C}(P,\Phi)$. Note that $F_1$ is indeed compactly supported, since $(\Phi|_{(\ker \Phi)^\perp})^{-1}(\mathbf{z_1})$ and $\Psi(\mathbf{y_1^k},\mathbf{0})$ are orthogonal vectors.\\

In summary, we have shown that (\ref{equation different forms of inequalities}) is equal to 
\begin{equation}
F_1((\mathbf{v},\iota(\mathbf{z_1}))/N) G(L\mathbf{z}),
\end{equation} plus an error of size \[O_{C,\Phi}(\eta \sigma^{-1} N^{-1} 1_{[0,O_{C,\Phi}(N)]}(\Vert \mathbf{z}\Vert_\infty)1_{[0,O(\eta)]}(\dist(\mathbf{z}, \im \Phi))).\] By the construction of $L$, this error is bounded by \[O_{C,\Phi}(\eta \sigma^{-1} N^{-1}1_{[0,O_{C,\Phi}(N)]}(\Vert \mathbf{z}\Vert_\infty) 1_{[0,O_L(\eta)]}(\Vert L\mathbf{z}\Vert_\infty)).\] The term $F_1(\mathbf{v},\iota(\mathbf{z_1})/N) G(L\mathbf{z})$ is not quite of the required form, since $F_1(\mathbf{v},\iota(\mathbf{z_1})/N)$ is not compactly supported as a function of $\mathbf{z}$. However, it may be easily massaged into this form. Indeed, from the above discussion we know that $G(L\mathbf{z}) \neq 0$ implies that $\Vert\mathbf{z_2}\Vert_\infty \leqslant C_1\eta$, for some constant $C_1$ that satisfies $C_1 = O_{L,\Phi}(1)$. Let $b:\mathbb{R}\longrightarrow [0,1]$ be a $1/2$-supported function (in the sense of Definition \ref{Defintion eta supported}), and let $B:\mathbb{R}^{d}\longrightarrow [0,1]$ be defined by $B(\mathbf{x}) = \prod_{j=1}^d b(x_j)$. Then let $F:\mathbb{R}^d \longrightarrow \mathbb{R}_{\geqslant 0}$ be defined by \[F(\mathbf{v},\mathbf{z}): = F_1(\mathbf{v},\iota (\mathbf{z_1})) B(\mathbf{z_2}).\] Then $F \in \mathcal{C}(P,\Phi)$, and if $\eta \leqslant 1/2C_1$ we have \[ F_1((\mathbf{v},\iota(\mathbf{z_1}))/N) G(L\mathbf{z}) = F((\mathbf{v},\mathbf{z})/N) G(L\mathbf{z}).\] The lemma is proved. 
\end{proof}

This reformulation allows us to deduce Corollary \ref{Corollary switching functions} below. This is a corollary of Theorem \ref{Theorem pseudorandomness} and is the result on inequalities and sieve weights that we will actually use in Section \ref{section Cauchy Schwarz argument}. In order to state this inequality, we introduce the following convention. 

\begin{Definition}[Convolution]
\label{Definition convolution}
If $f:\mathbb{Z}\longrightarrow \mathbb{R}$ has finite support, and $g:\mathbb{R}\longrightarrow [0,1]$ is a measurable function, we may define the convolution  $(f\ast g)(x) : \mathbb{R}\longrightarrow \mathbb{R}$ by \[ (f\ast g)(x): = \sum\limits_{n\in \mathbb{Z}} f(n) g(x -n).\] 
\end{Definition}

Recall from Section \ref{section conventions} that, for some positive parameter $\eta$, the function $\chi:\mathbb{R} \longrightarrow [0,1]$ denotes a fixed $\eta$-supported function.

\begin{Corollary}[Switching functions]
\label{Corollary switching functions}
Let $N,m,d$ be natural numbers, with $d\geqslant m+2$, and let $k$ be a non-negative integer. Let $C,\gamma,\eta$ be positive parameters, and let $P$ be a set of further parameters. Suppose that $\eta$ is small enough in terms of $m$, $d$, and $k$. Let $(\varphi_1,\dots,\varphi_d) = \Phi:\mathbb{R}^{d-m+k} \longrightarrow \mathbb{R}^d$ be a linear map with algebraic coefficients, and suppose that $k=\dim \ker \Phi$. Suppose that $\Phi$ has finite Cauchy-Schwarz complexity. Let $H:\mathbb{R}^{d-m+k} \longrightarrow [0,1]$ be a smooth function in $\mathcal{C}(P)$. For $j\leqslant d$, let $w_1,\dots,w_d$ be any functions with $w_j(n)\leqslant w(n)$ for all $n$ and for which $w_j(n) \rightarrow \infty$ as $n \rightarrow \infty$. For each $j\leqslant d$ let the function $f_j:\mathbb{Z}\longrightarrow \mathbb{R}_{\geqslant 0}$ be equal to either $\nu_{N,w}^\gamma$ or $\Lambda_{\mathbb{Z}/W_j\mathbb{Z}}$. Let $\mathbf{r} \in \mathbb{R}^d$ be any vector satisfying $\Vert \mathbf{r} \Vert_\infty \leqslant C N$. 

Then, if $\gamma$ is small enough in terms of $\Phi$, the expression \begin{equation}
\label{upper bounding count} \frac{1}{N^{d-m+k}} \int\limits_{\mathbf{x} \in \mathbb{R}^{d-m+k}} \Big( \prod\limits_{j=1}^{d} (f_j \ast \chi)(\varphi_j(\mathbf{x}) - r_j)\Big) H(\mathbf{x}/N) \, d\mathbf{x}
\end{equation}
\noindent is independent of the choices of the functions $f_j$, up to an error of size $o(1)$ as $N\rightarrow \infty$. This $o(1)$ term may depend on $C$, $P$, $\Phi$, $\eta$, $\gamma$, and on the functions $w_1,\dots,w_d$.
\end{Corollary}
\begin{proof}
Expanding out the definition of $f_j \ast \chi$, one observes that the left-hand side of (\ref{upper bounding count}) is equal to 
\begin{equation}
\label{equation bounding alternative count}\frac{1}{N^{d-m+k}}\sum\limits_{n_1,\dots,n_d \in \mathbb{Z}} \Big(\prod\limits_{j=1}^d f_j(n_j) \Big)\int\limits_{\mathbf{x} \in \mathbb{R}^{d-m+k}} \Big(\prod\limits_{j=1}^d \chi(\varphi_j(\mathbf{x}) - n_j - r_j) \Big) H(\mathbf{x}/N)\, d\mathbf{x}. 
\end{equation} By applying Lemma \ref{Lemma different forms of inequalities} to the inner integral, we get a surjective linear map $L:\mathbb{R}^d \longrightarrow \mathbb{R}^m$ with algebraic coefficients, and smooth functions $F:\mathbb{R}^d \longrightarrow \mathbb{R}_{\geqslant 0}$ and $G:\mathbb{R}^m \longrightarrow \mathbb{R}_{\geqslant 0}$ supported on $[-O_{P,\Phi}(1),O_{P,\Phi}(1)]^d$ and $[-O_{\Phi}(\eta),O_{\Phi}(\eta)]^m$ respectively and with $F,G \in \mathcal{C}(P,\Phi,\eta)$, such that (\ref{equation bounding alternative count}) is equal to 
\begin{equation}
\label{main term and error term}\frac{1}{N^{d-m}} \sum\limits_{n_1,\dots,n_d \in \mathbb{Z}} \Big(\prod\limits_{j=1}^d f_j(n_j) \Big)F(\mathbf{n}/N)G(L\mathbf{n} + L\mathbf{r})
\end{equation} plus an error of size 

\begin{equation}
\label{plus an error of size}
O_{P,\Phi}\Big(\frac{1}{N^{d-m+1}} \sum\limits_{\substack{n_1,\dots,n_d \ll_{C,\Phi} N\\ \Vert L\mathbf{n} + L\mathbf{r}\Vert_\infty = O_{C,\Phi}(\eta)}} \prod\limits_{j=1}^d  f_j(n_j)\Big).
\end{equation} Furthermore, $L \notin V_{\degen}^*(m,d)$. 

Now apply Theorem \ref{Theorem pseudorandomness} to the main term (\ref{main term and error term}). As written this theorem applies to functions $F$ and $G$ that take values in $[0,1]$, but by the obvious rescaling we may nonetheless apply the theorem to the present functions $F$ and $G$. This shows immediately that (\ref{main term and error term}) is independent of the particular choices of $f_1,\dots,f_d$, up to an error of size $o(1)$. The $o(1)$ term has the appropriate dependencies.

For the error term (\ref{plus an error of size}), we apply the upper bound in Corollary \ref{Corollary upper bound}. This shows that (\ref{plus an error of size}) is $o(N^{-1})$, so may be absorbed into the $o(1)$ term above. Corollary \ref{Corollary switching functions} is proved. 
\end{proof}
An upper bound in this setting will also be convenient. 
\begin{Corollary}
\label{Corollary more upper bounds}
Under the same hypotheses as Corollary \ref{Corollary switching functions}, \begin{equation}
\frac{1}{N^{d-m+k}} \int\limits_{\mathbf{x} \in \mathbb{R}^{d-m+k}} \Big( \prod\limits_{j=1}^{d} (f_j \ast \chi)(\varphi_j(\mathbf{x}) + r_j)\Big) H(\mathbf{x}/N) \, d\mathbf{x} \ll 1,
\end{equation}
\noindent where the implied constant may depend on $C$, $P$, $\Phi$, $\eta$, $\gamma$, and on the functions $w_1,\dots,w_d$. 
\end{Corollary}
\begin{proof}
Proceed as in the previous proof to get to expression (\ref{main term and error term}). Then apply the upper bound in Corollary \ref{Corollary upper bound}.
\end{proof}
\section{Variation in parameters}
\label{section structure of Q}
This section will be devoted to proving Lemma \ref{Lemma approximation of Q} below. This technical lemma shows that the number of solutions to certain inequalities, weighted by the local von Mangoldt function, is a quantity that behaves well when the underlying parameters are perturbed. The slightly esoteric notation, in which we introduce a dimension $d$ only to consider $\mathbf{x} \in \mathbb{R}^{d-1}$, is designed to correspond to the moment in Section \ref{section Cauchy Schwarz argument} in which this lemma will be applied.
\begin{Lemma}
\label{Lemma approximation of Q}
Let $d,l,N,s$ be natural numbers, with $d\geqslant 2$, and let $C,\eta$ be positive parameters. Let $(\varphi_1,\dots,\varphi_l) = \Phi:\mathbb{R}^{d-1} \longrightarrow \mathbb{R}^l$ and $(\psi_1,\dots,\psi_l) = \Psi: \mathbb{R}^{s+2} \longrightarrow \mathbb{R}^l$ be linear maps with algebraic coefficients. Let $P$ be a set of parameters, and let $b\in\mathcal{C}(C,P,\eta,\Phi,\Psi)$ be an arbitrary smooth function. Let $w^*:\mathbb{N} \longrightarrow \mathbb{R}$ be a function such that $w^*(n) \rightarrow \infty$ as $n\rightarrow \infty$ and $w^*(n) \leqslant w(n)$ for all $n$. Let $\mathbf{a} \in \mathbb{R}^l$ be a vector satisfying $\Vert \mathbf{a}\Vert_\infty \leqslant CN$. For $\mathbf{y} \in \mathbb{R}^{s+1}$, define 
\begin{equation}
\label{equation definition of Qy}
Q_{\mathbf{a},N}(\mathbf{y}) : = \frac{1}{N^{d-1}} \int\limits_{\mathbf{x} \in \mathbb{R}^{d-1}} \Big (\prod\limits_{j\leqslant l} (\Lambda_{\mathbb{Z}/W^*\mathbb{Z}} \ast \chi) ( \varphi_j(\mathbf{x}) +  \Psi(\mathbf{y}) + a_j ) \Big) b((\mathbf{x},\mathbf{y})/N) \, d\mathbf{x},
\end{equation} where $a_j$ is the $j^{th}$ coordinate of $\mathbf{a}$. Then, if $\eta$ is sufficiently small in terms of $\Phi$ and $\Psi$, there is a function $f_{1}:\mathbb{Z}^{l} \longrightarrow \mathbb{C}$, satisfying $\Vert f_1\Vert_\infty \ll (\log \log W^*)^{O(1)}$, such that  \[Q_{\mathbf{a},N}(\mathbf{y}) = b_{\mathbf{a},N}(\mathbf{y}/N) \sum\limits_{\Vert \mathbf{k}\Vert_{\infty}\leqslant (\log\log W^*)^{O(1)}} f_1(\mathbf{k})e\Big(\frac{\mathbf{k} \cdot (\Psi(\mathbf{y}) + \mathbf{a})}{W^*}\Big) +o_{C,P,\eta,\Phi,\Psi}(1).\] Here $b_{\mathbf{a},N} \in \mathcal{C}(C,P,\eta,\Phi,\Psi)$, though it may also depend on $\mathbf{a}$ and $N$. 
\end{Lemma}

None of the methods required to prove this lemma will be particularly deep, but the technical manoeuvres will be a little intricate. In particular, we will need to apply the approximation in Lemma \ref{Lemma different forms of inequalities} multiple times within the same argument. \\

The proof of Lemma \ref{Lemma approximation of Q} will require the preliminary result below, namely Lemma \ref{Lemma functions with small support}. To state this lemma, we define a metric on $\mathbb{R}^d/K\mathbb{Z}^d$ by the formula \[ \Vert \mathbf{x}\Vert _{\mathbb{R}^d/K\mathbb{Z}^d} := \min\limits_{\mathbf{n} \in K\mathbb{Z}^d } \Vert \mathbf{x} - \mathbf{n}\Vert_\infty.\] Lipschitz constants of functions $\mathfrak{F}:\mathbb{R}^d/K\mathbb{Z}^d \longrightarrow \mathbb{R}$ will be considered with respect to this metric. 

\begin{Lemma}
\label{Lemma functions with small support}
Let $d,m,K$ be natural numbers, and let $\eta,\sigma$ be positive parameters. Let $S:\mathbb{R}^d \longrightarrow \mathbb{R}^m$ be a surjective linear map with integer coefficients, and let $G:\mathbb{R}^m \longrightarrow [0,1]$ be a Lipschitz function supported on $[-\eta,\eta]^m$, with Lipschitz constant at most $\sigma^{-1}$. Let $\mathfrak{F}:S\mathbb{Z}^d \longrightarrow \mathbb{R}$ be any function for which \[\mathfrak{F}(S\mathbf{x}) = \mathfrak{F}(S\mathbf{x} + S\mathbf{n})\] for all $\mathbf{x} \in \mathbb{Z}^d$ and all $\mathbf{n} \in K\mathbb{Z}^d$.

For each $\mathbf{a} \in \mathbb{R}^d$, define $\widetilde{\mathbf{a}} \in \mathbb{Z}^d$ to be some vector with integer coordinates for which \[\Vert S(\widetilde{\mathbf{a}} - \mathbf{a})\Vert_\infty = \min \limits_{\mathbf{n} \in \mathbb{Z}^d} \Vert S(\mathbf{n} - \mathbf{a})\Vert_\infty.\] Then, provided $\eta$ is small enough in terms of $S$, the function 
\begin{equation}
\label{function that will be lipschitz}
\mathbf{a} \mapsto \mathfrak{F}(S\widetilde{\mathbf{a}})G(S(\widetilde{\mathbf{a}} - \mathbf{a}))
\end{equation}
\begin{itemize}
\item depends only on the value of $\mathbf{a}$ modulo $K\mathbb{Z}^d$;
\item is Lipschitz when viewed as a function on $\mathbb{R}^d/K\mathbb{Z}^d$, with Lipschitz constant at most \[O_S(\Vert \mathfrak{F}\Vert_\infty(\eta^{-1} + \sigma^{-1})).\]
\end{itemize}
\end{Lemma}

\begin{Remark}
\emph{The expression $\min \limits_{\mathbf{n} \in \mathbb{Z}^d} \Vert S(\mathbf{n} - \mathbf{a})\Vert_\infty$ is well-defined, since $S\mathbb{Z}^d$ is a lattice.}
\end{Remark}

\begin{proof}
To prove the first part of the lemma, let $\mathbf{a} \in \mathbb{R}^d$ and first suppose that there is a unique vector $\mathbf{x} \in S \mathbb{Z}^d$ for which \[ \Vert \mathbf{x} - S\mathbf{a}\Vert_\infty = \min \limits_{\mathbf{n} \in \mathbb{Z}^d} \Vert S\mathbf{n} - S\mathbf{a}\Vert_\infty.\] In this case, by the uniqueness of $\mathbf{x}$, we have $\mathbf{x} = S \widetilde{\mathbf{a}}$. By translation, we know that \[S(\widetilde{\mathbf{a} + \mathbf{n}}) - S \mathbf{n} = \mathbf{x}\] for all $\mathbf{n} \in \mathbb{Z}^d$, and hence \[ S(\widetilde{\mathbf{a} + \mathbf{n}}) - S \mathbf{n} = S\widetilde{\mathbf{a}} \] for all $\mathbf{n} \in \mathbb{Z}^d$. Hence \[ G(S(\widetilde{\mathbf{a}} - \mathbf{a})) = G(S(\widetilde{\mathbf{a} + \mathbf{n}} - (\mathbf{a} + \mathbf{n}))),\] and so the function \[ \mathbf{a} \mapsto G(S(\widetilde{\mathbf{a}} - \mathbf{a}))\] depends only on the value of $\mathbf{a}$ modulo $\mathbb{Z}^d$. Furthermore, if $\mathbf{n} \in K\mathbb{Z}^d$, \[ \mathfrak{F}(S(\widetilde{\mathbf{a} + \mathbf{n}})) = \mathfrak{F}(S\widetilde{\mathbf{a}} + S\mathbf{n}) =  \mathfrak{F}(S\widetilde{\mathbf{a}}),\] by the invariance properties of $\mathfrak{F}$. Hence the function (\ref{function that will be lipschitz}) only depends on the value of $\mathbf{a}$ modulo $K\mathbb{Z}^d$. 

Now suppose that there were two  distinct vectors $\mathbf{x_1}$, $\mathbf{x_2} \in S \mathbb{Z}^d$ for which \[ \Vert \mathbf{x_i} - S\mathbf{a}\Vert_\infty = \min \limits_{\mathbf{n} \in \mathbb{Z}^d} \Vert S\mathbf{n} - S\mathbf{a}\Vert_\infty\] for $i=1,2$. Then in fact $G(S(\widetilde{\mathbf{a}} - \mathbf{a})) = 0$. Indeed, if this were not the case then we would have $\Vert \mathbf{x_1} - \mathbf{x_2}\Vert_\infty \leqslant O(\eta)$, which is impossible if $\eta$ is small enough, since $\mathbf{x_1}$ and $\mathbf{x_2}$ are two distinct elements of $\mathbb{Z}^m$. By translation, we may also conclude that $G(S(\widetilde{\mathbf{a} + \mathbf{n}} - (\mathbf{a} + \mathbf{n}))) = 0$ for all $\mathbf{n} \in \mathbb{Z}^d$. So again, the function (\ref{function that will be lipschitz}) depends only on the value of $\mathbf{a}$ modulo $K\mathbb{Z}^d$.\\

Regarding the second part of the lemma, the idea of the proof is similar to the above. Indeed, the only aspect of the function (\ref{function that will be lipschitz}) that could lead to a large Lipschitz constant is of course the term $S\widetilde{\mathbf{a}}$, which could, one fears, jump sharply for small changes in $\mathbf{a}$. However, when such jumps occur, the function $G(S(\widetilde{\mathbf{a}} - \mathbf{a}))$ is always equal to zero.

Let us proceed with the full proof. Indeed, let $\mathbf{a_0}$,$\mathbf{a_1} \in \mathbb{R}^d$ and suppose first that \[\Vert \mathbf{a_0} - \mathbf{a_1}\Vert_{\mathbb{R}^d/K\mathbb{Z}^d} \leqslant \eta.\] By choosing suitable coset representatives, without loss of generality we may assume that \[\Vert \mathbf{a_0} - \mathbf{a_1}\Vert_{\mathbb{R}^d/K\mathbb{Z}^d} = \Vert \mathbf{a_0} - \mathbf{a_1}\Vert_{\infty}.\] 

Then either $S\widetilde{\mathbf{a_0}} =S\widetilde{\mathbf{a_1}}$ or $S\widetilde{\mathbf{a_0}} \neq S\widetilde{\mathbf{a_1}}$. If $S\widetilde{\mathbf{a_0}} =S\widetilde{\mathbf{a_1}}$ then 
\begin{align}
\vert G(S(\widetilde{\mathbf{a_0}} - \mathbf{a_0})) - G(S(\widetilde{\mathbf{a_1}} - \mathbf{a_1}))\vert &\leqslant \sigma^{-1} \Vert S \mathbf{a_0} - S\mathbf{a_1}\Vert_\infty\nonumber \\
& \ll_S \sigma^{-1} \Vert \mathbf{a_0} - \mathbf{a_1}\Vert_\infty \nonumber \\
& \ll_S \sigma^{-1} \Vert \mathbf{a_0} - \mathbf{a_1}\Vert_{\mathbb{R}^d /K\mathbb{Z}^d}.
\end{align}
\noindent Therefore 
\begin{align*}
\vert \mathfrak{F}(S\widetilde{\mathbf{a_0}}) G(S(\widetilde{\mathbf{a_0}} - \mathbf{a_0})) - \mathfrak{F}(S\widetilde{\mathbf{a_1}})G(S(\widetilde{\mathbf{a_1}} - \mathbf{a_1}))\vert &= \vert \mathfrak{F}(S\widetilde{\mathbf{a_0}})\vert \vert G(S(\widetilde{\mathbf{a_0}} - \mathbf{a_0})) - G(S(\widetilde{\mathbf{a_1}} - \mathbf{a_1}))\vert \nonumber \\
&\ll_S \Vert \mathfrak{F}\Vert_\infty \sigma^{-1}\Vert \mathbf{a_0} - \mathbf{a_1}\Vert_{\mathbb{R}^d /K\mathbb{Z}^d}.
\end{align*} That resolves the lemma in this case. 

 If on the other hand $S\widetilde{\mathbf{a_0}} \neq S \widetilde{\mathbf{a_1}}$, we may conclude that both 
 \begin{equation}
\label{claim equations}
 \Vert S \widetilde{\mathbf{a_0}} - S\mathbf{a_0}\Vert_\infty \geqslant 10\eta
\end{equation} and 
\begin{equation}
\label{claim equations 2}
 \Vert S \widetilde{\mathbf{a_1}} - S\mathbf{a_1}\Vert_\infty \geqslant 10\eta.
\end{equation} Indeed, if $\Vert S \widetilde{\mathbf{a_0}} - S\mathbf{a_0}\Vert_\infty \leqslant 10\eta$, say, then \[\Vert S\widetilde{\mathbf{a_0}}-S\mathbf{a_1}\Vert_\infty \leqslant 10\eta + \Vert S\mathbf{a_0} - S\mathbf{a_1} \Vert_\infty \ll_S \eta.\] If $\eta$ is small enough, this implies that $S\widetilde{\mathbf{a_0}}$ must be the unique element of $S\mathbb{Z}^d$ for which \[ \Vert  S\widetilde{\mathbf{a_0}} - S\mathbf{a_1}\Vert_\infty = \min \limits_{\mathbf{n} \in \mathbb{Z}^d} \Vert S\mathbf{n} - S\mathbf{a_1}\Vert_\infty,\]and hence that $S\widetilde{\mathbf{a_0}} = S\widetilde{\mathbf{a_1}}$, contradicting the assumption. 

If $\eta$ is small enough, expressions (\ref{claim equations}) and (\ref{claim equations 2}) imply that 
 \begin{align}
 \label{lipschitz align two}
G(S\widetilde{\mathbf{a_0}} - S\mathbf{a_0}) =G(S\widetilde{\mathbf{a_1}} - S\mathbf{a_1}) = 0,
 \end{align}
 \noindent and so \[\vert \mathfrak{F}(S\widetilde{\mathbf{a_0}}) G(S(\widetilde{\mathbf{a_0}} - \mathbf{a_0})) - \mathfrak{F}(S\widetilde{\mathbf{a_1}})G(S(\widetilde{\mathbf{a_1}} - \mathbf{a_1}))\vert = 0.\] That resolves the lemma in this case.\\
 
The only remaining case to consider is when \[ \Vert \mathbf{a_0} - \mathbf{a_1}\Vert_{\mathbb{R}^d/ K\mathbb{Z}^d} \geqslant \eta.\] In this case we bound the Lipschitz constant very crudely, as $O( \eta^{-1}\Vert \mathfrak{F}\Vert_\infty\Vert G\Vert_\infty)$, which is $O(\Vert \mathfrak{F}\Vert_\infty \eta^{-1})$, since $\Vert G\Vert_\infty \leqslant 1$. This settles the lemma.
\end{proof}
We are now ready to prove Lemma \ref{Lemma approximation of Q}.
\begin{proof}[Proof of Lemma \ref{Lemma approximation of Q}]
For this proof we make the following conventions. Any implied constant may depend on $C$, $\Phi$ and $\Psi$, and we will use the notation $b$, $b_1$, $b_2$ etc. to denote a function in $\mathcal{C}(C,P,\eta,\Phi, \Psi)$, that may change from line to line. 

The first part of the proof will involve establishing an asymptotic formula for $Q_{\mathbf{a},N}(\mathbf{y})$, namely the expression $Q_{\mathbf{a},N}(\mathbf{y}) = \mathfrak{S}_{\mathbf{a},N}(\mathbf{y}) I_{\mathbf{a},N}(\mathbf{y}) + o_{P,\eta}(1)$ in (\ref{Q asymp formula}) below. Indeed, expanding out the definition of $\Lambda_{\mathbb{Z}/W^*\mathbb{Z}}\ast \chi$ (see Definition \ref{Definition convolution}) we have
\begin{align}
\label{expanding out Q z h} Q_{\mathbf{a},N}(\mathbf{y}) = \frac{1}{N^{d-1}}\sum\limits_{\mathbf{n} \in \mathbb{Z}^l} \Big(\prod\limits_{j\leqslant l} \Lambda_{\mathbb{Z}/W^*\mathbb{Z}}(n_j)\Big)\int\limits_{\mathbf{x} \in \mathbb{R}^{d-1}}  \boldsymbol{\chi}(\mathbf{n} - \Phi(\mathbf{x}) - \Psi(\mathbf{y}) - \mathbf{a}) 
b((\mathbf{x},\mathbf{y})/N)\, d\mathbf{x},
\end{align}
\noindent where $\bc: \mathbb{R}^{l} \longrightarrow [0,1]$ is defined by $\bc(\mathbf{z}): = \prod_{j\leqslant l} \chi(z_j)$. Let $k := \dim \im \Phi$, and note that $k\leqslant d-1$. 

The inner integral of (\ref{expanding out Q z h}) may be analysed using Lemma \ref{Lemma different forms of inequalities}. The following table indicates which objects in (\ref{expanding out Q z h}) play which role in Lemma \ref{Lemma different forms of inequalities}.
\begin{center}
\begin{tabular}{c|c}
Notation of Lemma \ref{Lemma different forms of inequalities} & Objects in (\ref{expanding out Q z h})\\
\hline
$\mathbf{z}$ & $\mathbf{n} - \Psi(\mathbf{y}) - \mathbf{a}$ \\
$\mathbf{v}$ & $\mathbf{y}$\\
$\Phi$ & $\Phi$\\
$H$ & $b$\\
$I$ & $\bc$
\end{tabular}
\end{center}  So, applying Lemma \ref{Lemma different forms of inequalities}, one sees that (\ref{expanding out Q z h}) is equal to
\begin{align}
\label{equation writing Q as an inequality}
Q_{\mathbf{a},N}(\mathbf{y}) = \frac{1}{N^{d-1}}\sum\limits_{\mathbf{n} \in \mathbb{Z}^l} \Big(\prod\limits_{j\leqslant l} \Lambda_{\mathbb{Z}/W^*\mathbb{Z}}(n_j)\Big)b_1((\mathbf{y},\mathbf{n})/N) b_2(L(\mathbf{n} - \Psi(\mathbf{y}) - \mathbf{a})) + E.
\end{align} Here, $L:\mathbb{R}^{l} \longrightarrow \mathbb{R}^{l-k}$ is a surjective linear map with algebraic coefficients, that depends only on $\Phi$, $\Rad(b_2) = O(\eta)$, and the error term $E$ may be bounded above by 
\begin{equation}
\label{easy error term}
\ll_P \Big\vert \frac{1}{N^{d}} \sum^* \Big(\prod\limits_{j\leqslant l} \Lambda_{\mathbb{Z}/W^*\mathbb{Z}}(n_j)\Big) \Big\vert,
\end{equation} 
where the summation $\sum^*$ denotes summation over the set \[\{\mathbf{n} \in \mathbb{Z}^l: \Vert \mathbf{n} \Vert_\infty \ll_P N, \, \Vert L\mathbf{n} - L(\Psi(\mathbf{y})) - L\mathbf{a}\Vert_\infty \ll_P \eta\}.\] The error term $E$ is easy to bound. Indeed, by Lemma \ref{Lemma trivial upper bound}, expression (\ref{easy error term}) may be bounded by $O_P(N^{k-d} (\log \log W^*)^d)$. Since $k\leqslant d-1$, this is an $o_P(1)$ error. \\

It remains to analyse the main term in (\ref{equation writing Q as an inequality}), which we will do with the help of Lemma \ref{Lemma generating a purely irrational map}. The reader is invited to consult Section \ref{section linear algebra and dimension reduction} for the statement of this result, and for the definitions of rational map, rational dimension, etcetera. 

Now, let $u$ be the rational dimension of $L$, and let $\Theta: \mathbb{R}^{l - k} \longrightarrow \mathbb{R}^u$ be a rational map for $L$ with algebraic coefficients. Then, there exists an injective linear map $(\xi_1,\dots,\xi_l) = \Xi:\mathbb{R}^{l - u} \longrightarrow \mathbb{R}^l$ with integer coefficients, satisfying $\Xi \mathbb{Z}^{l-u} = \mathbb{Z}^l \cap \ker \Theta L$, and a vector $\widetilde{\mathbf{r}}(\mathbf{a},\mathbf{y})\in \mathbb{Z}^{l}$, such that the main term of (\ref{equation writing Q as an inequality}) is equal to
\begin{align}
\label{in structure of Q rational manip}
\frac{1}{N^{d-1}} \sum\limits_{\mathbf{n} \in \mathbb{Z}^{l - u}} &\Big(\prod\limits_{j=1}^{l}  \Lambda_{\mathbb{Z}/W^*\mathbb{Z}}(\xi_j(\mathbf{n}) + \widetilde{r}(\mathbf{a},\mathbf{y})_j)\Big)\nonumber \\
 &b_1((\mathbf{y},\Xi(\mathbf{n}) + \widetilde{\mathbf{r}}(\mathbf{a},\mathbf{y}))/N) b_2(L(\Xi(\mathbf{n}) + \widetilde{\mathbf{r}}(\mathbf{a},\mathbf{y}) -\Psi(\mathbf{y}) - \mathbf{a})),
\end{align} 
\noindent where $\widetilde{r}(\mathbf{a},\mathbf{y})_j$ is the $j^{th}$ coordinate of $\mathbf{\widetilde{r}}(\mathbf{a},\mathbf{y})$. Note how we've appealed to part (11) of Lemma \ref{Lemma generating a purely irrational map} for the particular form of the argument of $b_2$. Note also how, since $\eta$ is sufficiently small, we have been able to apply part (10) of the lemma to establish that $\widetilde{R}$ consists of a single element $\mathbf{\widetilde{r}}(\mathbf{a},\mathbf{y})$.

 Moreover, from part (10) of the lemma again, we have that $\widetilde{\mathbf{r}}(\mathbf{a},\mathbf{y})$ is an element of $\mathbb{Z}^{l}$ for which \[\Vert \Theta L(\widetilde{\mathbf{r}}(\mathbf{a},\mathbf{y}) - \Psi(\mathbf{y}) - \mathbf{a})\Vert_{\infty} = \min \limits_{\mathbf{m} \in \mathbb{Z}^l} \Vert \Theta L(\mathbf{m} - \Psi(\mathbf{y}) - \mathbf{a})\Vert_\infty.\] From part (9) of Lemma \ref{Lemma generating a purely irrational map}, letting $\{ \mathbf{e_1},\dots,\mathbf{e_{l-u}} \}$ be the standard basis vectors of $\mathbb{R}^{l-u}$, we have a set \[ \mathcal{B} = \{\mathbf{x_i}: i\leqslant u\} \cup \{ \Xi(\mathbf{e_j}): j\leqslant l - u \}\] which is a lattice basis for $\mathbb{Z}^{l}$ and for which $\{ \Theta L\mathbf{x_i}:i\leqslant u\}$ is a lattice basis for $\Theta L \mathbb{Z}^l$. Letting $U = \spn( \{\mathbf{x_i}: i\leqslant u\})$, we have that $\widetilde{\mathbf{r}}(\mathbf{a},\mathbf{y})\in U$.

By applying the first part of Lemma \ref{Lemma problem for local von Mangoldt} to expression (\ref{in structure of Q rational manip}), one immediately derives
\begin{equation}
\label{Q asymp formula}
Q_{\mathbf{a},N}(\mathbf{y}) =  \mathfrak{S}_{\mathbf{a},N}(\mathbf{y}) I_{\mathbf{a},N}(\mathbf{y}) + o_{P,\eta}(1),
\end{equation} where 
\begin{equation}
\label{singular series which depends on z and h}\mathfrak{S}_{\mathbf{a},N}(\mathbf{y}):=\frac{1}{(W^*)^{l-u}} \sum\limits_{\mathbf{m} \in [W^*]^{l -u}} \prod\limits_{j=1}^{l} \Lambda_{\mathbb{Z}/W^*\mathbb{Z}}(\xi_j (\mathbf{m}) +\widetilde{r}(\mathbf{a},\mathbf{y})_j)
\end{equation} and $I_{\mathbf{a},N}(\mathbf{y})$ is equal to
\begin{equation}
\label{singular integral which depends on z and h}
\frac{1}{N^{d-1}}\int\limits_{\mathbf{x} \in \mathbb{R}^{l -u}}b_1((\mathbf{y},\Xi(\mathbf{x}) + \widetilde{\mathbf{r}}(\mathbf{a},\mathbf{y}))/N)b_2(L \Xi(\mathbf{x}) + L\widetilde{\mathbf{r}}(\mathbf{a},\mathbf{y}) - L(\Psi(\mathbf{y})) -L\mathbf{a})\, d\mathbf{x}. 
\end{equation}
\noindent Note that \begin{equation}
\label{kernel inequalities}\dim \ker L\Xi \leqslant \dim \ker L = k \leqslant d-1.
\end{equation}

The remainder of the proof of Lemma \ref{Lemma approximation of Q} will consist of analysing expressions (\ref{singular series which depends on z and h}) and (\ref{singular integral which depends on z and h}) for $\mathfrak{S}_{\mathbf{a},N}(\mathbf{y})$ and $I_{\mathbf{a},N}(\mathbf{y})$. \\

We begin with $I_{\mathbf{a},N}(\mathbf{y})$, aiming for expression (\ref{approximation of singular integral}). Letting $V = \im \Xi$, we have that $\mathbb{R}^l = U \oplus V$. For any vector $\mathbf{v} \in \mathbb{R}^{l}$ let $\mathbf{v}|_U$ and $\mathbf{v}|_V$ be the components in $U$ and $V$ respectively. Then we have that 
\begin{equation}
\label{equation U approx}
\Vert \widetilde{\mathbf{r}}(\mathbf{a},\mathbf{y}) - \Psi(\mathbf{y})|_U - \mathbf{a}|_U\Vert_\infty = O(1),
\end{equation} since \[\Vert\Theta L(\widetilde{\mathbf{r}}(\mathbf{a},\mathbf{y}) - \Psi(\mathbf{y}) - \mathbf{a})\Vert_\infty = O(1).\] By the bound on the Lipschitz constant of $b_1$, we may replace $b_1((\mathbf{y},\Xi(\mathbf{x}) + \widetilde{\mathbf{r}}(\mathbf{a},\mathbf{y}))/N)$ with $b_1((\mathbf{y},\Xi(\mathbf{x}) + \Psi(\mathbf{y})|_U +\mathbf{a}|_U)/N)$ in (\ref{singular integral which depends on z and h}), up to an error of $O_{P,\eta}(N^{-1})$. Also, note that \[ \frac{1}{N^{d-1}}\int\limits_{\substack{\mathbf{x} \in \mathbb{R}^{l -u}\\ \Vert\mathbf{x}\Vert_\infty \ll N}}b_2(L \Xi(\mathbf{x}) + L\widetilde{\mathbf{r}}(\mathbf{a},\mathbf{y}) - L(\Psi(\mathbf{y})) -L\mathbf{a})\, d\mathbf{x} = O_{P,\eta}(N^{\dim \ker L\Xi -d+1}),\] by Lemma \ref{Lemma general upper bound}. This is $O_{P,\eta}(1)$, since $\dim\ker L\Xi\leqslant d-1$ by (\ref{kernel inequalities}). Therefore we may replace (\ref{singular integral which depends on z and h}) by the expression \begin{equation}
\label{altered version of I}
\frac{1}{N^{d-1}}\int\limits_{\mathbf{x} \in \mathbb{R}^{l-u}}b_1((\mathbf{y},\Xi(\mathbf{x}) + \Psi(\mathbf{y})|_U + \mathbf{a}|_U)/N)b_2(L \Xi(\mathbf{x}) + L\widetilde{\mathbf{r}}(\mathbf{a},\mathbf{y}) - L(\Psi(\mathbf{y})) -L\mathbf{a})\, d\mathbf{x},
\end{equation}
\noindent plus an error of size $o_{P,\eta}(1)$. \\

The expression (\ref{altered version of I}) is in a form that is amenable to Lemma \ref{Lemma different forms of inequalities}. The following table indicates which objects from our present discussion play which role in the notation of Lemma \ref{Lemma different forms of inequalities}. 

\begin{center}
\begin{tabular}{ c|c } 
Notation of Lemma \ref{Lemma different forms of inequalities}  & Objects related to (\ref{altered version of I}) \\
\hline
 $\mathbf{z}$ &  $L\widetilde{\mathbf{r}}(\mathbf{a},\mathbf{y}) - L(\Psi(\mathbf{y})) - L\mathbf{a}$ \\
 $\mathbf{v}$ & $\mathbf{y}$ \\
 $\Phi$ & $-L\Xi$\\
 $L$ & $\Theta$\\
 
  $ (\mathbf{v},\mathbf{x}) \mapsto H(\mathbf{v}, \mathbf{x})$ & $(\mathbf{y},\mathbf{x}) \mapsto b_1(\mathbf{y}, \Xi(\mathbf{x}) + \frac{\Psi(\mathbf{y})|_U}{N} + \frac{\mathbf{a}|_U}{N})$\\
  $I$ & $b_2$ 
\end{tabular}
\end{center}
\noindent This is a valid application of Lemma \ref{Lemma different forms of inequalities}, since $\ker \Theta = \im L\Xi$ and the final two functions in the right-hand column are compactly supported smooth functions of their arguments (as $\Xi$ is injective, $\Xi(\mathbf{x}) \in V$, and $V$ is an algebraic complement to $U$). Recalling that $\Theta$ has algebraic coefficients, by the third part of Lemma \ref{Lemma different forms of inequalities} we may therefore replace (\ref{altered version of I}) by an expression of the form 
\begin{equation}
\label{getting closer}
b_1((\mathbf{y}, L(\widetilde{\mathbf{r}}(\mathbf{a},\mathbf{y}) - \Psi(\mathbf{y}) - \mathbf{a}))/N)b_2(\Theta L(\widetilde{\mathbf{r}}(\mathbf{a},\mathbf{y}) -\Psi(\mathbf{y}) - \mathbf{a})) + o_{P,\eta}(1)  
\end{equation}
\noindent where $\Rad(b_2) = O(\eta)$. 

The argument of the function $b_1$ above doesn't depend smoothly on $\mathbf{y}$, but this may be easily rectified. Indeed, by (\ref{equation U approx}) and the fact that $b_1$ is Lipschitz and $b_2$ is bounded, (\ref{getting closer}) is equal to \[ b_1((\mathbf{y}, -L(\Psi(\mathbf{y})|_V + \mathbf{a}|_V))/N)b_2(\Theta L(\widetilde{\mathbf{r}}(\mathbf{a},\mathbf{y}) -\Psi(\mathbf{y}) - \mathbf{a})) + o_{P,\eta}(1),\] i.e. is equal to \begin{equation}
\label{approximation of singular integral}
b_{\mathbf{a},N,1}(\mathbf{y}/N) b_2(\Theta L(\widetilde{\mathbf{r}}(\mathbf{a},\mathbf{y}) -\Psi(\mathbf{y}) - \mathbf{a})) + o_{P,\eta}(1),
\end{equation} where $\Rad(b_2) = O(\eta)$. \\

In summary then, since $\mathfrak{S}_{\mathbf{a},N}(\mathbf{y}) \ll (\log\log W^*)^{O(1)}$ we have shown that 
\begin{equation}
\label{important staging post}
Q_{\mathbf{a},N}(\mathbf{y}) = \mathfrak{S}_{\mathbf{a},N}(\mathbf{y}) b_{\mathbf{a},N,1}(\mathbf{y}/N) b_2(\Theta L(\widetilde{\mathbf{r}}(\mathbf{a},\mathbf{y}) -\Psi(\mathbf{y}) - \mathbf{a})) + o_{P,\eta}(1).\\
\end{equation} 
The function 
\begin{equation}
\label{function too that will be lipschitz}
(\Psi(\mathbf{y}) + \mathbf{a}) \mapsto \mathfrak{S}_{\mathbf{a},N}(\mathbf{y})b_2(\Theta L(\widetilde{\mathbf{r}}(\mathbf{a},\mathbf{y}) -\Psi(\mathbf{y}) - \mathbf{a}))
\end{equation} is of the form considered in Lemma \ref{Lemma functions with small support} in expression (\ref{function that will be lipschitz}). Indeed, one first notes that (\ref{function too that will be lipschitz}) is a well-defined mapping, since $\widetilde{\mathbf{r}}(\mathbf{a},\mathbf{y})$ is determined only by $\Psi(\mathbf{y}) + \mathbf{a}$ and $\mathfrak{S}_{\mathbf{a},N}(\mathbf{y})$ depends on $\mathbf{a}$ and $\mathbf{y}$ only through the value of $\widetilde{\mathbf{r}}(\mathbf{a},\mathbf{y})$ (see (\ref{singular series which depends on z and h})). Then, one takes the map $S$ from Lemma \ref{Lemma functions with small support} to be the map $\Theta L:\mathbb{R}^l \longrightarrow \mathbb{R}^u$ here, ones takes $K$ from that lemma to be $W^*$ here, and one takes the map $G$ from that lemma to be $b_2$ here, and one takes the map $\mathfrak{F}: \Theta L \mathbb{Z}^l \longrightarrow \mathbb{R}$ from that lemma to be \[\mathfrak{F}(\mathbf{x}) =  \frac{1}{(W^*)^{l-u}} \sum\limits_{\mathbf{m} \in [W^*] ^{l-u}} \prod\limits_{j=1}^l \Lambda_{\mathbb{Z}/W^* \mathbb{Z}} (\xi_j(\mathbf{m}) + (\Theta L|_U)^{-1}(\mathbf{x})_j)\] here. The definition of $\mathfrak{F}$ is valid since $\Theta L |_U: U \longrightarrow \mathbb{R}^u$ is indeed a bijection, and by part (9) of Lemma \ref{Lemma generating a purely irrational map} we have $(\Theta L|_U)^{-1}(\Theta L (\mathbb{Z}^l)) = \mathbb{Z}^l \cap U$.  Consulting expression (\ref{singular series which depends on z and h}) for $\mathfrak{S}_{\mathbf{a},N}(\mathbf{y})$, one sees that \[ \mathfrak{F}(\Theta L \widetilde{\mathbf{r}}(\mathbf{a},\mathbf{y})) = \mathfrak{S}_{\mathbf{a},N}(\mathbf{y})\] and so (\ref{function too that will be lipschitz}) is indeed of the form (\ref{function that will be lipschitz}) as we have claimed. The only hypothesis of Lemma \ref{Lemma functions with small support} that we haven't already verified is the invariance of $\mathfrak{F}$ under translation by elements of $\Theta L(W^*\mathbb{Z}^l)$, but this is immediate from the definition of $\mathfrak{F}$, since $(\Theta L |_U)^{-1}: \mathbb{R}^u \longrightarrow U$ is linear and $\Lambda_{\mathbb{Z}/W^* \mathbb{Z}}$ is $W^*$-periodic. Therefore, by applying Lemma \ref{Lemma functions with small support}, we conclude that the function (\ref{function too that will be lipschitz}) is Lipschitz on $\mathbb{R}^l/W^*\mathbb{Z}^l$, with Lipschitz constant $O_{P,\eta}((\log\log W^*)^{O(1)}).$\\

The proof of Lemma \ref{Lemma approximation of Q} is nearly complete, since Lipschitz functions enjoy good approximation by short exponential sums. Indeed, by Lemma A.9 of \cite{GT08a}, for all $X>2$ there exists a function $f_1:\mathbb{Z}^{l} \longrightarrow \mathbb{C}$ such that $\Vert f_1\Vert_\infty \ll (\log \log W^*)^{O(1)}$ and \[\mathfrak{S}_{\mathbf{a},N}(\mathbf{y})b_2(\Theta L(\widetilde{\mathbf{r}}(\mathbf{a},\mathbf{y}) -\Psi(\mathbf{y}) - \mathbf{a}))\] equals \[ \sum\limits_{\Vert \mathbf{k}\Vert_{\infty}\leqslant X} f_1(\mathbf{k}) e\Big(\frac{\mathbf{k} \cdot (\Psi(\mathbf{y}) + \mathbf{a})}{W^*}\Big) + O_{P,\eta}( (\log\log W^*)^{O(1)} (\log X)/X).\] Then, picking $X$ to be a suitably large power of $\log\log W^*$, Lemma \ref{Lemma approximation of Q} follows. 
\end{proof}

\part{The main argument}
\label{part gen von neu}
Having completed all the preparatory material, the main thrust of the proof can begin in earnest. 
\section{Controlling by Gowers norms}
\label{section Controlling by Gowers norms}
In this section we state a type of result that has become known as a `generalised von Neumann theorem', which uses Gowers norms to bound the number of solutions to a diophantine inequality. For readers familiar with \cite{GT10}, the procedure is routine. We will then show that this result implies the main theorem (Theorem \ref{Main theorem}).



\begin{Theorem}[Generalised von Neumann Theorem]
\label{Theorem generalised von neumann}
Let $N,m,d$ be natural numbers, satisfying $d\geqslant m+2$, and let $C,\gamma,\varepsilon$ be positive parameters. Let $L:\mathbb{R}^d \longrightarrow \mathbb{R}^m$ be a surjective linear map with algebraic coefficients, and assume that $L\notin V_{\degen}^*(m,d)$ and that $\gamma$ is small enough (depending on $L$). Let $\mathbf{v}\in\mathbb{R}^{m}$ satisfy $\Vert \mathbf{v}\Vert_\infty \leqslant CN$. Let $F:\mathbb{R}^d\longrightarrow [0,1]$ and $G:\mathbb{R}^{m}\longrightarrow [0,1]$ be functions with Lipschitz constants at most $\sigma^{-1}$, and suppose that $F$ is supported on $[-1,1]^d$ and $G$ is supported on $[-\varepsilon,\varepsilon]^m$. Let $f_1,\dots,f_d:[N]\longrightarrow \mathbb{R}$ be arbitrary functions, satisfying $\vert f_j(n)\vert\leqslant \nu_{N,w}^\gamma(n)$ for all $j\leqslant d$ and for all $n\leqslant N$.

Then there exists an $s = O(1)$ such that, if \[\min_{j\leqslant d} \Vert f_j \Vert_{U^{s+1}[N]} =o(1)\] as $N\rightarrow \infty$, then \[\vert T^{L,\mathbf{v}}_{F,G,N}(f_1,\dots,f_d)\vert = o(1)\] as $N \rightarrow \infty$. The second $o(1)$ term may also depend on $C$, $L$, $\gamma$, $\varepsilon$, $\sigma$, and the rate of decay of the first $o(1)$ term.
\end{Theorem}

\begin{proof}[Proof of Theorem \ref{Main theorem} assuming Theorem \ref{Theorem generalised von neumann}] 


Assume the hypotheses of Theorem \ref{Main theorem}. By telescoping we have that \[T_{F,G,N}^{L,\mathbf{v}}(\Lambda^\prime,\dots,\Lambda^\prime) - T_{F,G,N}^{L,\mathbf{v}}(\Lambda_{\mathbb{Z}/ W\mathbb{Z}}^+,\dots,\Lambda_{\mathbb{Z}/ W\mathbb{Z}}^+)\] is equal to
\begin{align}
\label{telescoping expression}
&T_{F,G,N}^{L,\mathbf{v}}(\Lambda^\prime - \Lambda_{\mathbb{Z}/W\mathbb{Z}}^+,\Lambda_{\mathbb{Z}/W\mathbb{Z}}^+,\dots,\Lambda_{\mathbb{Z}/W\mathbb{Z}}^+) +\nonumber \\ & T_{F,G,N}^{L,\mathbf{v}}(\Lambda^\prime, \Lambda^\prime - \Lambda_{\mathbb{Z}/W\mathbb{Z}}^+,\Lambda_{\mathbb{Z}/W\mathbb{Z}}^+,\dots,\Lambda_{\mathbb{Z}/W\mathbb{Z}}^+) + \dots + T_{F,G,N}^{L,\mathbf{v}}(\Lambda^\prime, \dots, \Lambda^\prime,\Lambda^\prime - \Lambda_{\mathbb{Z}/W\mathbb{Z}}^+).
\end{align}
\noindent Since $F$ is supported on $[-1,1]^d$, we may restrict the functions $\Lambda^\prime$ and $\Lambda^+_{\mathbb{Z}/W\mathbb{Z}}$ to $[N]$ without altering the size of expression (\ref{telescoping expression}).

By the construction of the sieve weight $\nu_{N,w}^\gamma$ we have \[\vert \Lambda^\prime(n) + \Lambda_{\mathbb{Z}/W\mathbb{Z}}^+(n)\vert \ll_{\gamma}  \nu_{N,w}^\gamma(n)\] for all $n \leqslant N$. Therefore, after rescaling, we may apply Theorem \ref{Theorem generalised von neumann} in this setting.

Recall that, by Lemma \ref{Lemma Corollary of tool from Green-Tao},  \[\Vert\Lambda^\prime  - \Lambda_{\mathbb{Z}/W\mathbb{Z}}^+  \Vert_{U^{s+1}[N]} = \Vert\Lambda^\prime  - \Lambda_{\mathbb{Z}/W\mathbb{Z}}\Vert_{U^{s+1}[N]} = o(1)\] as $N\rightarrow \infty$, for all $s\leqslant d-2$. So, applying Theorem \ref{Theorem generalised von neumann} to each term of (\ref{telescoping expression}) separately, we derive \[\vert T_{F,G,N}^{L,\mathbf{v}}(\Lambda^\prime,\dots,\Lambda^\prime) - T_{F,G,N}^{L,\mathbf{v}}(\Lambda_{\mathbb{Z}/ W\mathbb{Z}}^+,\dots,\Lambda_{\mathbb{Z}/ W\mathbb{Z}}^+)\vert =o_{C,L,\gamma,\varepsilon,\sigma}(1)\] as $N\rightarrow \infty$. By fixing a suitably small value of $\gamma$, we conclude Theorem \ref{Main theorem}.
\end{proof}

\section{Transferring from $\mathbb{Z}$ to $\mathbb{R}$}
\label{section transfer}
In this section we begin the proof of Theorem \ref{Theorem generalised von neumann}. Following the programme set out in \cite{Wa17}, our first step will be to transfer the problem from the setting of functions on $\mathbb{Z}$ to functions on $\mathbb{R}$. 

\begin{Definition}
\label{Definition really continuous solution count}
Let $N,m,d$ be natural numbers. Let $L:\mathbb{R}^{d}\longrightarrow \mathbb{R}^m$ be a linear map, let $\mathbf{v} \in \mathbb{R}^m$, and let $F:\mathbb{R}^{d}\rightarrow [0,1]$ and $G:\mathbb{R}^m\rightarrow [0,1]$ be compactly supported measurable functions. Then, for all bounded measurable functions $g_1,\dots,g_{d}:\mathbb{R}\longrightarrow  \mathbb{R}$ we define
\begin{equation}
\label{definiton of the continuous solution count form}
\widetilde{T}^{L,\mathbf{v}}_{F,G,N}(g_1,\dots,g_{d}) := \frac{1}{N^{d-m}}\int\limits_{\mathbf{x}\in \mathbb{R}^d}\Big(\prod\limits_{j=1}^{d}g_j(x_j)\Big)F(\mathbf{x}/N)G(L\mathbf{x}+ \mathbf{v})\, d\mathbf{x}.
\end{equation}
\end{Definition}

We now state the key lemma. For the definition of $f \ast \chi$, where $\chi$ is the function we determined in Section \ref{section conventions}, the reader may consult Definition \ref{Definition convolution}. 
\begin{Lemma}[Transfer]
\label{Lemma transfer equation}
Let $N,m,d$ be natural numbers, with $d\geqslant m+2$, and let $C$, $\varepsilon$, $\gamma$, $\eta$, $\sigma$ be positive constants.  Let $L:\mathbb{R}^d \longrightarrow \mathbb{R}^m$ be a surjective linear map, and let $\mathbf{v} \in \mathbb{R}^m$ be a vector satisfying $\Vert \mathbf{v}\Vert_\infty \leqslant CN$. Let $F:\mathbb{R}^d\longrightarrow [0,1]$ and $G:\mathbb{R}^{m}\longrightarrow [0,1]$ be compactly supported Lipschitz functions, with Lipschitz constants at most $\sigma^{-1}$. Suppose that $F$ is supported on $[-1,1]^d$, and $G$ is supported on $[-\varepsilon,\varepsilon]^m$. Then there exists some positive real number $C_{\chi}$, satisfying $C_\chi \asymp 1$, such that the following holds. Let $f_1,\dots,f_d:[N]\longrightarrow \mathbb{R}$ be arbitrary functions that satisfy $\vert f_j(n)\vert \leqslant \nu_{N,w}^\gamma(n)$ for all $j\leqslant d$ and for all $n\leqslant N$.  Assume that $\eta \leqslant \min (1,\varepsilon)$ and that $\gamma$ is small enough depending on $L$. Then
\begin{equation}
\label{eqn transfer equation}
T_{F,G,N}^{L,\mathbf{v}}(f_1,\dots,f_d) = \frac{1}{C_{\chi}\eta^{d}}\widetilde{T}_{F,G,N}^{L,\mathbf{v}}(f_1\ast \chi,\dots,f_d\ast \chi) + O(\eta \sigma^{-1}) + o(1)
\end{equation}
\noindent as $N\rightarrow \infty$. The implied constant in the $O(\eta \sigma^{-1})$ term may depend on $C$, $L$, and $\varepsilon$, and the $o(1)$ term may depend on all these parameters together with $\gamma$ and $\sigma$.  
\end{Lemma}

\begin{proof}
The proof is very similar to the proof of Lemma 5.4 in \cite{Wa17}, although we do have to insert various estimates that are only proved in this paper.  

Indeed, let $\boldsymbol{\chi}:\mathbb{R}^d \longrightarrow [0,1]$ denote the function $\mathbf{x}\mapsto \prod\limits_{j=1}^d \chi(x_j)$. We choose \[ C_{\chi} : = \frac{1}{\eta^d}\int\limits_{\mathbf{x} \in \mathbb{R}^d} \bc(\mathbf{x}) \, d\mathbf{x} .\] Since $\chi$ is $\eta$-supported, $C_{\chi} \asymp 1$ . Then, expanding the definition of the convolutions $f_i \ast \chi$, \[\frac{1}{C_{\chi}\eta^{d}}\widetilde{T}_{F,G,N}^{L, \mathbf{v},}(f_1\ast \chi,\dots,f_d\ast \chi)\] equals
\begin{align}
\label{equation transferring zero}
&\frac{1}{N^{d-m}}\sum\limits_{\mathbf{n}\in \mathbb{Z}^d}\Big(\prod\limits_{j=1}^d f_j(n_j)\Big) \frac{1}{C_{\chi}\eta^d} \int\limits_{\mathbf{y}\in \mathbb{R}^d} F(\mathbf{y}/N) G(L\mathbf{y} + \mathbf{v}) \boldsymbol{\chi}(\mathbf{y}  - \mathbf{n}) \, d\mathbf{y}. 
\end{align}
This is equal to 
\begin{equation}
\label{extra equation to be referenced in transfer argument}
\frac{1}{N^{d-m}}\sum\limits_{\substack{\mathbf{n}\in \mathbb{Z}^d \\ \Vert \mathbf{n} \Vert_\infty \leqslant 2 N}} \Big(\prod\limits_{j=1}^d f_j(n_j)\Big) \frac{1}{C_{\chi}\eta^d}\int\limits_{\mathbf{y}\in \mathbb{R}^d}  (F(\mathbf{n}/N) + O(\eta\sigma^{-1} N^{-1}))G(L\mathbf{y}+\mathbf{v}) \boldsymbol{\chi}(\mathbf{y} - \mathbf{n}) \, d\mathbf{y}.
\end{equation}
Indeed, the inner integrand is only non-zero when  $\Vert \mathbf{y} - \mathbf{n}\Vert_\infty\leqslant\eta$, and $F$ has Lipschitz constant $O(\sigma^{-1})$.\\

Continuing, expression (\ref{extra equation to be referenced in transfer argument}) is equal to 
\begin{equation}
\label{equation transferring}
\frac{1}{N^{d-m}} \sum\limits_{\substack{\mathbf{n}\in\mathbb{Z}^d \\ \Vert \mathbf{n}\Vert_\infty \leqslant 2N}}\Big(\prod\limits_{j=1}^d f_j(n_j)\Big) F(\mathbf{n}/N) H(L\mathbf{n} + \mathbf{v}) + E
\end{equation}
\noindent where \[H(\mathbf{x}) = \frac{1}{C_{\chi}\eta^d}\int\limits_{\mathbf{y}\in\mathbb{R}^d}  \boldsymbol{\chi}(\mathbf{y})G(\mathbf{x} + L\mathbf{y})\ \, d\mathbf{y}\] and $E$ is a certain error, which may be bounded above by a constant times 
\begin{equation}
\label{transfer error bounded above}
\frac{\eta}{\sigma N} \frac{1}{N^{d-m}} \sum\limits_{\substack{\mathbf{n}\in \mathbb{Z}^d \\ \Vert \mathbf{n} \Vert_\infty \leqslant 2 N}}\Big(\prod\limits_{j=1}^d \nu_{N,w}^\gamma (n_j)\Big)H(L\mathbf{n} + \mathbf{v}).\\
\end{equation}  

Let us deal with the first term of (\ref{equation transferring}), in which we wish to replace $H$ with $G$. We therefore consider
\begin{equation*}
\Big\vert\frac{1}{N^{d-m}} \sum\limits_{\mathbf{n}\in\mathbb{Z}^d}\Big(\prod\limits_{j=1}^d f_j(n_j)\Big) F(\mathbf{n}/N) (G(L\mathbf{n} + \mathbf{v}) - H(L\mathbf{n} + \mathbf{v})) \Big\vert,
\end{equation*}
which is
\begin{equation}
\label{equation transferring main term differencing}
\leqslant \frac{1}{N^{d-m}} \sum\limits_{\mathbf{n}\in\mathbb{Z}^d}\Big(\prod\limits_{j=1}^d \nu_{N,w}^\gamma(n_j)\Big)F(\mathbf{n}/N)  \vert G-H\vert(L\mathbf{n} + \mathbf{v}).
\end{equation} Observe that $\Vert G - H\Vert_\infty = O(\eta \sigma^{-1})$. Indeed, 
\begin{align*}
&G(\mathbf{x}) - \frac{1}{C_{\chi}\eta^d}\int\limits_{\mathbf{y}\in\mathbb{R}^d} G(\mathbf{x}+ L\mathbf{y}) \chi(\mathbf{y}) \, d\mathbf{y} \\
&= G(\mathbf{x}) - \frac{1}{C_{\chi}\eta^d} \int\limits_{\mathbf{y} \in\mathbb{R}^d} (G(\mathbf{x}) + O(\eta \sigma^{-1}))\boldsymbol{\chi}(\mathbf{y}) \, d\mathbf{y}\\
& = O(\eta\sigma^{-1}),\nonumber
\end{align*}
\noindent by the definition of $C_{\chi}$ and the Lipschitz property of $G$. The function $\vert G - H\vert$ is compactly supported, with $\operatorname{Rad}(\vert G - H\vert ) \ll \varepsilon + \eta \ll \varepsilon$.

Of course $\vert G - H\vert$ needn't be smooth, but we may nonetheless apply Corollary \ref{Corollary upper bound}, concluding that expression (\ref{equation transferring main term differencing}) is at most \[ O_{C,L,\varepsilon}( \eta \sigma^{-1} ) + o_{C,L,\gamma,\varepsilon, \sigma}(1).\] 

Turning to the error $E$ from (\ref{equation transferring}), we've already remarked that it may be bounded above by expression (\ref{transfer error bounded above}). Applying Corollary \ref{Corollary upper bound} again, expression (\ref{transfer error bounded above}) is $o(1)$ (with the appropriate dependencies on $C$, $L$, etc.).  

The lemma then follows. 
\end{proof}
We will need to show that the operation of replacing $f$ by $f\ast \chi$ is compatible with Gowers norms. 

Firstly, if $g:[-N,N]\longrightarrow \mathbb{R}$ is a bounded measurable function, we define the Gowers norm over the reals $\Vert g\Vert_{U^{d}(\mathbb{R},N)}$ by 
\begin{equation}
\label{real gowers norm}
\Vert g \Vert_{U^{d}(\mathbb{R},N)}^{2^{d}}:= \frac{1}{(2N)^{d+1}}\int\limits_{(x,\mathbf{h}) \in\mathbb{R}^{d+1}} \prod\limits_{\boldsymbol{\omega}\in \{0,1\}^d}\mathscr{C}^{\vert \boldsymbol{\omega} \vert} g(x+\sum\limits_{i=1}^d h_i\omega_i) \, dx\,d \mathbf{h}.
\end{equation} More detail about this quantity may be found in Appendix A of \cite{Wa17}. 

Secondly, we note that $\Vert f\Vert_{U^{d}[N]}$ and $\Vert f \ast \chi\Vert_{U^{d}(\mathbb{R},2N)}$ may be related. 
\begin{Lemma}[Relating different Gowers norms]
\label{Lemma linking different Gowers norms}
Let $s$ be a natural number, and assume that $\eta$ is a positive parameter that is small enough in terms of $s$.  Let $N$ be a natural number, and let $f:[N]\longrightarrow \mathbb{R}$ be an arbitrary function. Then we have \begin{equation}
\label{linking different gowers norms}
\Vert f\ast \chi\Vert_{U^{s+1}(\mathbb{R},2N)}\ll \eta^{\frac{s+2}{2^{s+1}}} \Vert f\Vert_{U^{s+1}[N]}. 
\end{equation}
\end{Lemma}
\begin{proof}
This is Lemma 5.5 of \cite{Wa17}. 
\end{proof}
\section{Parametrising the kernel}
\label{section General proof of the real variable von Neumann Theorem}
In this section we will convert the expression $T_{F,G,N}^{L,\mathbf{v}}(f_1,\dots,f_d)$ into an expression that is tailored to the subsequent manipulations. We begin with a lemma that is very similar to Proposition 8.2 of \cite{Wa17}.

\begin{Lemma}[Separating out the kernel]
\label{Proposition separating out the kernel}
Let $N,m,d$ be natural numbers, with $d\geqslant m+2$, and let $C,\varepsilon,\sigma$ be positive constants. Let $L:\mathbb{R}^d \longrightarrow \mathbb{R}^m$ be a surjective linear map with algebraic coefficients, and assume further that $L \notin V_{\degen}^*(m,d)$. Let $F:\mathbb{R}^d\longrightarrow [0,1]$ be a Lipschitz function supported on $[-1,1]^d$, with Lipschitz constant at most $\sigma^{-1}$, and let $G:\mathbb{R}^{m}\longrightarrow [0,1]$ be any function supported on $[-\varepsilon,\varepsilon]^m$. Then there exists an injective linear map $(\psi_1,\dots,\psi_d) = \Psi:\mathbb{R}^{d-m}\longrightarrow \mathbb{R}^d$ with algebraic coefficients (depending only on $L$), and a Lipschitz function $F_1:\mathbb{R}^{d-m}\longrightarrow [0,1]$ with Lipschitz constant $O_{L}(\sigma^{-1})$ and with $\Rad(F_1) = O_{C,L,\varepsilon}(1)$, such that, if $g_1,\dots,g_d:\mathbb{R}\longrightarrow \mathbb{R}$ are arbitrary bounded measurable functions, 
\begin{equation}
\vert\widetilde{T}_{F,G,N}^{L,\mathbf{v}}(g_1,\dots,g_d)\vert\ll_{L,\varepsilon} \Big\vert\frac{1}{N^{d-m}}\int\limits_{\mathbf{x} \in \mathbb{R}^{d-m}} F_1(\mathbf{x}/N)\Big(\prod\limits_{j=1}^d g_j(\psi_j(\mathbf{x}) + a_j)\Big) \, d\mathbf{x}\Big\vert ,
\end{equation}
\noindent where, for each $j$, $a_j$ is some real number that satisfies $a_j = O_{C,L,\varepsilon}(N)$.

Furthermore, $\Psi$ has finite Cauchy-Schwarz complexity (see Definition \ref{Definition finite complexity}).
\end{Lemma}

\begin{proof}[Proof of Lemma \ref{Proposition separating out the kernel}]

For ease of notation, let 
\[ \beta: = \widetilde{T}_{F,G,N}^{L,\mathbf{v}}(g_1,\dots,g_d).\]
Noting that $\ker L$ is a vector space of dimension $d-m$, define $\{\mathbf{u^{(1)}},\dots,\mathbf{u^{(d-m)}}\} \subset \mathbb{R}^d$ to be an orthonormal basis for $\ker L$ consisting of vectors with algebraic coordinates. Then the map $(\psi_1,\dots,\psi_d) = \Psi:\mathbb{R}^{d-m} \longrightarrow \mathbb{R}^d$, defined by
\begin{equation}
\label{parametrise the kernel}
\Psi(\mathbf{x}) := \sum\limits_{i=1}^{d-m}x_i \mathbf{u}^{\mathbf{(i)}},
\end{equation}
\noindent is an injective map that parametrises $\ker L$. Furthermore $\Psi$ has finite Cauchy-Schwarz complexity, since otherwise there would exist $i \neq j\leqslant d$ and a real number $\lambda$ such that $\mathbf{e_i}^* - \lambda \mathbf{e_j}^* \in (\im \Psi)^0$, i.e. $\mathbf{e_i}^* - \lambda \mathbf{e_j}^* \in (\ker L)^{0}$. This implies that $\mathbf{e_i}^* - \lambda \mathbf{e_j}^* \in L((\mathbb{R}^m)^*)$, which, by definition, implies that $L \in V_{\degen}^*(m,d)$, contradicting our hypotheses. 


Now, extend the orthonormal basis $\{\mathbf{u^{(1)}},\dots,\mathbf{u^{(d-m)}}\}$ for $\ker L$ to an orthonormal basis $\{\mathbf{u^{(1)}},\dots,\mathbf{u^{(d)}}\}$ for $\mathbb{R}^d$. By implementing a change of basis, we may rewrite 
\begin{equation}
\label{separating out the kernel}
\beta=\frac{1}{N^{d-m}} \int\limits_{\mathbf{x}\in \mathbb{R}^d} F(\sum\limits_{i=1}^{d}x_i\mathbf{u^{(i)}}/N)G(L(\sum\limits_{i=1}^{d}x_i\mathbf{u^{(i)}}) + \mathbf{v})\Big(\prod\limits_{j=1}^{d} g_j(\psi_j(\mathbf{x})+\sum\limits_{i=d-m+1}^{d} x_i u^{(i)}_j)\Big) \, d\mathbf{x},
\end{equation}
where $u^{(i)}_j$ is the $j^{th}$ coordinate of $\mathbf{u}^{(\mathbf{i})}$.

We wish to remove the presence of the variables $x_{d-m+1},\dots,x_{d}$. To set this up, note that, by the choice of the vectors $\mathbf{u^{(i)}}$, $$G(L(\sum\limits_{i=1}^{d}x_i\mathbf{u^{(i)}}) + \mathbf{v}) = G(L(\sum\limits_{i=d-m+1}^{d}x_i\mathbf{u^{(i)}}) + \mathbf{v}). $$ The vector $\sum_{i=d-m+1}^{d}x_i\mathbf{u^{(i)}}$ is in $(\ker L)^{\perp}$ and so, since $L|_{(\ker L)^{\perp}}$ is a bounded invertible operator, $G(L(\sum_{i=d-m+1}^{d}x_i\mathbf{u^{(i)}}) + \mathbf{v})$ is equal to zero unless $(x_{d-m+1},\dots,x_d)^T\in D$, for some domain $D \subseteq \mathbb{R}^{m}$ of diameter $O_{L}(\varepsilon)$ and satisfying $\sup_{\mathbf{x} \in D} \Vert \mathbf{x}\Vert_\infty = O_{C,L}(N + \varepsilon)$. 

We can use this observation to bound the right-hand side of (\ref{separating out the kernel}). Indeed, we have 
\begin{align}
\label{equation bound beta by a sup}
\beta \ll \operatorname{vol} D \times \sup\limits_{\mathbf{x_{d-m+1}^d}\in D}\frac{1}{N^{d-m}} \Big\vert \int\limits_{\mathbf{x_{1}^{d-m}} \in \mathbb{R}^{d-m}} F(\sum\limits_{i=1}^{d}x_i\mathbf{u^{(i)}}/N)G(L(\sum\limits_{i=d-m+1}^{d}x_i\mathbf{u^{(i)}}) + \mathbf{v}) \nonumber\\
\prod\limits_{j=1}^{d} g_j(\psi_j(\mathbf{x_{1}^{d-m}})+\sum\limits_{i=d-m+1}^{d} x_iu^{(i)}_j) \, d \mathbf{x_1^{d-m}} \Big\vert.
\end{align} See Section \ref{section conventions} for explanation of $\mathbf{x_1^{d-m}}$ notation. So there exists some fixed vector\\ $(x_{d-m+1},\dots,x_d)^T$ in $D$ such that \begin{align}
\label{equation bound beta second}
\beta \ll_{L,\varepsilon}\frac{1}{N^{d-m}} \Big\vert \int\limits_{\mathbf{x_{1}^{d-m}} \in \mathbb{R}^{d-m}} F(\sum\limits_{i=1}^{d}x_i\mathbf{u^{(i)}}/N)G(L(\sum\limits_{i=d-m+1}^{d}x_i\mathbf{u^{(i)}}) + \mathbf{v}) \nonumber\\
\prod\limits_{j=1}^{d} g_j(\psi_j(\mathbf{x_{1}^{d-m}})+\sum\limits_{i=d-m+1}^{d} x_iu^{(i)}_j) \, d \mathbf{x_1^{d-m}} \Big\vert.
\end{align}  

Define the function $F_1:\mathbb{R}^{d-m} \longrightarrow [0,1]$ by \[ F_1(\mathbf{x_1^{d-m}}) : = F(\Psi(\mathbf{x_1^{d-m}}) + (\sum\limits_{i=d-m+1}^{d} x_i\mathbf{u^{(i)}}/N))\] and for each $j$ at most $d$ a shift \[a_j :=\sum\limits_{i=d-m+1}^{d} x_i u^{(i)}_j.\] Then
\begin{equation}
\label{just before we put everything into normal form}
\beta \ll _{L,\varepsilon}\Big\vert\frac{1}{N^{d-m}}\int\limits_{\mathbf{x} \in \mathbb{R}^{d-m}}F_1(\mathbf{x}/N) \Big(\prod\limits_{j=1}^d g_j(\psi_j(\mathbf{x})+a_j)\Big) \, d\mathbf{x}\Big\vert,
\end{equation} 
\noindent and $F_1$ and $a_j$ satisfy the conclusions of the proposition.
\end{proof}

The next proposition is essentially identical to an argument that appears in \cite{Wa17} at the end of Section 8 of that paper. Unfortunately that argument is not in an easily citable form, and so we have found it necessary to state and prove the precise version that we need here. For readers unfamiliar with the notion of normal form, we included a brief summary in Section \ref{section normal form}.

\begin{Lemma}[Parametrising by normal form]
\label{Proposition parametrising by normal form}
Following on from above, there exists a $d^\prime = O(1)$, a linear map $(\psi_1^\prime,\dots,\psi_d^\prime) : = \Psi^\prime:\mathbb{R}^{d^\prime} \longrightarrow \mathbb{R}^d$ with algebraic coefficients that is in $s$-normal form for some $s = O(1)$, and a Lipschitz function $F_2:\mathbb{R}^{d^\prime} \longrightarrow [0,1]$ with Lipschitz constant $O_L(\sigma^{-1})$ and with $\Rad(F_2) = O_{C,L,\varepsilon}(1)$ such that 
\begin{equation}
\label{before normal form}
\Big\vert\frac{1}{N^{d-m}}\int\limits_{\mathbf{x} \in \mathbb{R}^{d-m}} F_1(\mathbf{x}/N)\prod\limits_{j=1}^d g_j(\psi_j(\mathbf{x}) + a_j) \, d\mathbf{x}\Big\vert
\end{equation}
\noindent is bounded above by a constant times
\begin{equation}
\label{just after introducing R primed}
\ll \Big\vert\frac{1}{N^{d^\prime}}\int\limits_{\mathbf{x} \in \mathbb{R}^{d^\prime}}F_2(\mathbf{x}/N)\prod\limits_{j=1}^d g_j(\psi_j^\prime(\mathbf{x}) + a_j) \, d\mathbf{x}\Big\vert.
\end{equation}

\end{Lemma}
\begin{proof}
We apply Lemma \ref{Lemma normal form algorithm} to $\Psi$. Therefore, there is a natural number $k = O(1)$ such that, for \emph{any} real numbers $y_1,\dots,y_{k}$, (\ref{before normal form}) is equal to 
\begin{equation}
\label{Fixed dummey variables}
\Big\vert \frac{1}{N^{d-m}}\int\limits_{\mathbf{x} \in \mathbb{R}^{d-m}}F_1((\mathbf{x} + \sum\limits_{i=1}^{k} y_i \mathbf{f_i})/N)\prod\limits_{j=1}^d g_j(\psi_j^\prime(\mathbf{y},\mathbf{x}) + a_j) \, d\mathbf{x}\Big\vert,
\end{equation}
\noindent where 
\begin{itemize}
\item $\mathbf{f_1},\cdots,\mathbf{f_{k}} \in \mathbb{R}^{d-m}$ are some vectors that satisfy $\Vert\mathbf{f_i}\Vert_\infty = O_{\Psi}(1)$ for each $i$ at most $k$;
\item for each $j$ at most $d$, $\psi^\prime_j:\mathbb{R}^{k}\times\mathbb{R}^{d-m} \longrightarrow \mathbb{R}$ is linear, and $(\psi_1^\prime,\cdots,\psi^\prime_d): = \Psi^\prime: \mathbb{R}^{k} \times \mathbb{R}^{ d-m} \longrightarrow \mathbb{R}^d$ is defined by \[ \Psi^\prime(\mathbf{y},\mathbf{x}): = \psi(\mathbf{x} + \sum\limits_{i=1}^{k} y_i\mathbf{f_i});\]
\item $\Psi^\prime$ is in $s$-normal form, for some $s=O(1)$.
\end{itemize}
\noindent We remark that the right-hand side of expression (\ref{Fixed dummey variables}) is independent of $\mathbf{y}$, as it was obtained by applying the change of variables $\mathbf{x} \mapsto \mathbf{x} + \sum_{i=1}^{k} y_i\mathbf{f_i} $.\\

Now, with $\rho$ as fixed in Section \ref{section conventions}, let $P:\mathbb{R}^{k}\longrightarrow [0,1]$ be defined by \[P(\mathbf{y}) := \prod\limits_{i=1}^{k} \rho(y_i).\] Integrating over $\mathbf{y}$, we have that (\ref{Fixed dummey variables}) is at most a constant times
\begin{align}
 &
\frac{1}{N^{d-m+k}} \int\limits_{\mathbf{y} \in \mathbb{R}^{k}} P(\mathbf{y}/N) \Big\vert\int\limits_{\mathbf{x}\in\mathbb{R}^{d-m}} F_1((\mathbf{x} + \sum\limits_{i=1}^{k} y_i \mathbf{f_i})/N)\prod\limits_{j=1}^{d} g_j(\psi_j^\prime(\mathbf{y},\mathbf{x}) + a_j) \, d\mathbf{x}\Big\vert \, d\mathbf{y} \nonumber \\
\label{equation after itnegrating over w}
 &\ll \Big\vert\frac{1}{N^{d-m+k}} \int\limits_{\substack{\mathbf{x} \in \mathbb{R}^{d-m}\\ \mathbf{y} \in \mathbb{R}^{k}}} F_2((\mathbf{y},\mathbf{x})/N) \prod\limits_{j=1}^{d} g_j(\psi_j^\prime(\mathbf{y},\mathbf{x}) + a_j) \, d\mathbf{x} \, d\mathbf{y}\Big\vert,
\end{align}
where the function $F_2:\mathbb{R}^{d-m+k} \longrightarrow [0,1]$ is defined by \[F_2(\mathbf{y},\mathbf{x}):= F_1(\mathbf{x} + \sum\limits_{i=1}^{k} y_i \mathbf{f_i}) P(\mathbf{y}). \] Notice in (\ref{equation after itnegrating over w}) that we were able to move the absolute value signs outside the integral, as $P$ is positive and the integral over $\mathbf{x}$ is independent of $\mathbf{y}$ (so in particular has constant sign).

Letting $d^\prime: = d-m+k$, the lemma is proved.
\end{proof}

\section{Gowers-Cauchy-Schwarz argument}
\label{section Cauchy Schwarz argument}
This section will be devoted to proving the following theorem, which lies at the heart of the proof of our main results.  
\begin{Theorem}[Gowers-Cauchy-Schwarz argument]
\label{Theorem Cauchy}
Let $N,t,d,s$ be natural numbers, and let $\gamma, \eta,\sigma,C$ be positive constants. Let $a_1,\dots,a_d$ be fixed real numbers that satisfy $\vert a_j\vert \leqslant CN$ for all $j$. Let $(\psi_1,\dots,\psi_t) = \Psi:\mathbb{R}^{d} \longrightarrow \mathbb{R}^t$ be a linear map with algebraic coefficients, which is in $s$-normal form. Let $F:\mathbb{R}^d \longrightarrow [0,1]$ be a Lipschitz function supported on $[-1,1]^d$ and with Lipschitz constant at most $\sigma^{-1}$. Let $g_1,\dots,g_t:[-2N,2N]\longrightarrow \mathbb{R}$ be any bounded measurable functions that satisfy $\vert g_j(x)\vert \leqslant (\nu_{N,w}^\gamma\ast \chi)(x+a_j)$ for all $x$. Suppose that \[\min\limits_{j\leqslant d}\Vert g_j \Vert _{U^{s+1}(\mathbb{R},2N)} =o(1)\] as $N \rightarrow \infty$. Then if $\eta$ and $\gamma$ are small enough in terms of $\Psi$ and the dimensions $t$, $d$, and $s$, 
\begin{equation}
\label{equation abstract Gen von Neu}
\frac{1}{N^{d}} \int\limits_{\mathbf{x}\in \mathbb{R}^d} \Big(\prod\limits_{j=1}^t g_j (\psi_j(\mathbf{x}))\Big)  F(\mathbf{x}/N)\, d\mathbf{x}=  o(1)
\end{equation}
as $N\rightarrow \infty$, where the error term can depend on $C$, $\sigma$, $\eta$, $\gamma$, $\Psi$, and the first $o(1)$ term.
\end{Theorem}
\noindent For the definition of $\Vert g_j\Vert_{U^{s+1}(\mathbb{R},2N)}$, the reader may consult expression (\ref{real gowers norm}).\\

Theorem \ref{Theorem Cauchy} is closely analogous to \cite[Proposition $7.1^{\prime\prime}$]{GT10}, and the first half of our proof will follow the proof of that proposition closely (and in particular will contain no new ideas). However, new technicalities will become apparent as the argument progresses. In particular it will become important to understand the structure of a function that we will come to denote by $Q_{\mathbf{a},N}(z,\mathbf{h})$, and this will not be easy, in that we will have to appeal to the highly technical Lemma \ref{Lemma approximation of Q}. This observation and the subsequent analysis constitute the main new elements of the proof of Theorem \ref{Theorem Cauchy}. 
\begin{proof}
We begin by replacing $F$ with a cut-off function that will be easier to work with during the subsequent manipulations. Indeed, let us pick a positive parameter $\delta \in (0,1]$. By Lemma \ref{Lemma approximating Lipschitz functions by smooth boxes} there is some parameter $k = O(\delta^{-d})$ and some smooth functions $F_1,\dots,F_k:\mathbb{R}^d \longrightarrow [0,1]$ such that \[\Vert F - \sum\limits_{i=1}^k F_i \Vert_\infty = O(\delta \sigma^{-1})\] and each $F_i$ is of the form \[ F_i(\mathbf{x}) = c_{i,F} \prod\limits_{j=1}^d b_j^i(x_j/N),\] where $\vert c_{i,F}\vert \leqslant 1$ and the functions $b_j^i:\mathbb{R} \longrightarrow [0,1]$ are smooth, supported on $[-2,2]$, and satisfy $b_{j}^i \in \mathcal{C}(\delta)$.

Therefore, we may write the left-hand side of (\ref{equation abstract Gen von Neu}) as the sum of $O(\delta^{-d})$ expressions of the form
\begin{equation}
\label{stone weierstrass reduction}
c_{i,F}\frac{1}{N^{d}} \int\limits_{\mathbf{x}\in\mathbb{R}^d} \prod\limits_{l=1}^t g_l (\psi_l(\mathbf{x})) \prod\limits_{j=1}^d b_j^i(x_j/N) \, d\mathbf{x},
\end{equation}
plus an error of size at most
\begin{equation}
\label{error from stone weierstrass}
\delta \sigma^{-1}\frac{1}{N^{d}} \int\limits_{\substack{\mathbf{x}\in\mathbb{R}^d \\ \Vert \mathbf{x}\Vert_\infty \ll N} } \prod\limits_{l=1}^t (\nu_{N,w}^\gamma\ast\chi) (\psi_l(\mathbf{x}) + a_l).
\end{equation} 
\noindent Since $\Psi$ is in $s$-normal form, for some finite $s$, it follows that $\Psi$ has finite Cauchy-Schwarz complexity (see Definition \ref{Definition finite complexity}). Therefore, by Corollary \ref{Corollary more upper bounds}, expression (\ref{error from stone weierstrass}) has size $O_{C,\eta,\gamma}(\delta \sigma^{-1})$.  \\

We now arrange our notation for the rest of the proof, in part to mimic the notation that is used in the proof of \cite[Proposition $7.1^{\prime\prime}$]{GT10}. This will hopefully increase the readability for those who are familiar with \cite{GT10}. Indeed, without loss of generality we may assume that \[\min_{j\leqslant d}(\Vert g_j\Vert_{U^{s+1}[N]}) = \Vert g_1\Vert_{U^{s+1}[N]}.\] Since $\Psi$ is in $s$-normal form there is a set $J \subset \{\mathbf{e_1},\dots,\mathbf{e_d}\}$ of standard basis vectors with $\vert J\vert \leqslant s+1$ and for which $\prod_{j\in J} \psi_i(\mathbf{e_j})$ vanishes for $i \neq 1$ and is nonzero for $i=1$. By the nested property of Gowers norms we may assume that $\vert J\vert = s+1$, and by reordering the variables we can assume without loss of generality that $\prod_{j=1}^{s+1} \psi_i(\mathbf{e_j})$ vanishes for $i\neq 1$ and is nonzero for $i=1$. It will be useful to rename the first $s+1$ variables $\mathbf{x}$ and the remainder as $\mathbf{y}$. If $d = s+1$ then the variable $\mathbf{y}$ is trivial. Note that the coefficients $\psi_1(\mathbf{e_j})$ are non-zero for all $j \in [s+1]$, so, by rescaling the variables $\mathbf{x}$, we may assume that \[ \psi_1(\mathbf{x},\mathbf{y}) = x_1+ \dots + x_{s+1} + \psi_1(\mathbf{0},\mathbf{y}).\] 

For $i \leqslant t$, let $\Omega(i)$ denote\footnote{This is the notation used in \cite{GT10}. In this paper it will never risk being confused with the meaning of $\Omega$ in asymptotic notation.} the set \[ \Omega(i): = \{j\in [s+1]: \psi_i(\mathbf{e_j}) \neq 0\}.\] Note that $\Omega(1) = [s+1]$ and $\Omega(i)\neq [s+1]$ for $i=2,\dots,t$. \\

Now, for any set $B\subseteq [s+1]$ and vector $\mathbf{x}\in \mathbb{R}^{s+1}$, we define the vector $\mathbf{x_B}$ to be the restriction of $\mathbf{x}$ to the coordinates in $B$. Then, for any set $B\subseteq [s+1]$ and vector $\mathbf{y} \in \mathbb{R}^{d-s-1}$, we define \[ G_{B,\mathbf{y}}(\mathbf{x_B}): = \prod\limits_{i\in [t] : \Omega(i) = B} g_i(\psi_i(\mathbf{x_B},\mathbf{y})),\] where we have abused notation slightly in viewing $\psi_i$ only as a function of those variables $x_j$ on which it depends. \\

We also use $b:\mathbb{R}^a \longrightarrow \mathbb{R}$ (for some implied dimension parameter $a$) to denote a smooth function in $\mathcal{C}(C,\delta,\eta,\gamma,\Psi)$. The exact function may change from line to line.

With this notation, by picking $\delta$ to be a suitably slowly decaying function of $N$ we see that Theorem \ref{Theorem Cauchy} would follow from the upper bound 
\begin{align}
\label{equation first omega product}
\frac{1}{N^{d-s-1}} \int\limits_{\mathbf{y}\in \mathbb{R}^{d-s-1}}\frac{1}{N^{s+1}}\int\limits_{\mathbf{x}\in \mathbb{R}^{s+1}}\prod\limits_{B\subseteq [s+1]} G_{B,\mathbf{y}}(\mathbf{x_B}) \prod\limits_{j=1}^{s+1} b_j(x_j/N) \prod\limits_{k=s+2}^{d} b_k(y_{k-s-1}/N) \, d\mathbf{x} \, d\mathbf{y} \nonumber \\=o_{C,\delta, \eta,\gamma,\Psi}(1).
\end{align}  Our entire task is now to establish (\ref{equation first omega product}). From this point onwards, we will allow any error term or implied constant to depend on $C$, $\delta$, $\eta$, $\gamma$, and $\Psi$, without notating so explicitly. \\

We proceed by considering the following version of \cite[Corollary B.4]{GT10}. 

\begin{Proposition}[The weighted generalised von Neumann theorem]
\label{Proposition weighted gen von Neu}
Let $A$ be a finite set, and let $(\mu_\alpha)_{\alpha\in A}$ be a finite collection of compactly supported Borel probability measures on $\mathbb{R}$. For every $B\subseteq A$, let $\mu_B$ denote the product measure $\bigotimes\limits_{\alpha\in B} \mu_\alpha$ on $\mathbb{R}^B$, and let $f_B: \mathbb{R}^B\longrightarrow \mathbb{C}$ and $\theta_B:\mathbb{R}^B\longrightarrow \mathbb{R}_{\geqslant 0}$ be integrable functions such that $\vert f_B(\mathbf{x_B})\vert\leqslant \theta _B(\mathbf{x_B})$ for all $\mathbf{x_B} \in \mathbb{R}^B$. Then 
\begin{equation}
\label{equation weighted generalised von Neumann theorem}
\Big\vert \int\limits_{\mathbf{x_A} \in \mathbb{R}^A} \Big(\prod\limits_{B\subseteq A} f_B (\mathbf{x_B}) \Big)\, d\mu_A(\mathbf{x_A}) \Big\vert \leqslant \Vert f_A\Vert _{\square^A(\theta;\mu_A)} \prod\limits_{B\subsetneq A} \Vert \theta_B\Vert _{\square^B(\theta;\mu_B)} ^{2^{\vert B\vert - \vert A\vert}},
\end{equation}
where for any $B\subseteq A$ and $h_B:\mathbb{R}^B\longrightarrow \mathbb{C}$ we define $\Vert h_B\Vert _{\square^B(\theta;\mu_B)}$ to be the unique nonnegative real number satisfying 
\begin{align}
\label{equation definition of weighted gowers norm}
&\Vert h_B\Vert _{\square^B(\theta;\mu_B)}^{2^{\vert B\vert }}: = \nonumber \\
 &\int\limits_{\mathbf{x_B^{(0)}},\mathbf{x_B^{(1)}}\in \mathbb{R}^B} \Big (\prod\limits_{\boldsymbol{\omega_B} \in \{0,1\}^B} \mathscr{C}^{\vert \boldsymbol{\omega_B}\vert } h_B(\mathbf{x_B^{(\boldsymbol{\omega_B})}}) \Big ) \prod\limits_{C\subsetneq B} \prod\limits_{\boldsymbol{\omega_C} \in \{0,1\}^C} \theta_C(\mathbf{x_C^{(\boldsymbol{\omega_C})}}) \, d\mu_B(\mathbf{x_B^{(0)}}) \, d\mu_B(\mathbf{x_B^{(1)}}).
\end{align}
\noindent Here, as before, we use $\mathbf{x_C}$ to denote the restriction of $\mathbf{x_B}$ to $\mathbb{R}^C$. 
\end{Proposition}
\begin{proof}
The proof is identical to the proof of \cite[Corollary B.4]{GT10}, replacing all summations with integrals, and is a consequence of the Gowers-Cauchy-Schwarz inequality. 
\end{proof}

We now apply this proposition to the left-hand side of (\ref{equation first omega product}) above. Observe that we have the pointwise bounds $\vert G_{B,\mathbf{y}}(\mathbf{x_B})\vert \ll \theta_{B,\mathbf{y}} (\mathbf{x_B})$, where \[ \theta_{B,\mathbf{y}}(\mathbf{x_B}): = \prod\limits_{i\in [t]: \Omega(i) = B} (\nu_{N,w}^\gamma\ast \chi) (\psi_i(\mathbf{x_B},\mathbf{y}) + a_i).\] Therefore, applying Proposition \ref{Proposition weighted gen von Neu} by taking $A$ to be the set $[s+1]$, each $\mu_\alpha$ to be proportional to $(1/N)b_j(x_j/N) dx_j$, and $\theta_B$ to be the function $\theta_{B,\mathbf{y}}$, we establish that the left-hand side of (\ref{equation first omega product}) is
\begin{align}
\label{after Gowers cauchy Schwarz}
\ll \frac{1}{N^{d-s-1}} \int\limits_{\mathbf{y} \in \mathbb{R}^{d-s-1}} \Vert G_{[s+1],\mathbf{y}} \Vert _{\square^{[s+1]}(\theta_{\mathbf{y}};\mu_{[s+1]})}\prod\limits_{B\subsetneq [s+1]} \Vert \theta_{B,\mathbf{y}} \Vert _{\square^B (\theta_{\mathbf{y}}; \mu_B)}^{2^{\vert B\vert - s- 1}}\nonumber \\ \prod\limits_{k=s+2}^{d} b_k(y_{k-s-1}/N) \, d\mathbf{y}. 
\end{align}
Observe that \[G_{[s+1],\mathbf{y}}(\mathbf{x_{[s+1]}}) = g_1(x_1+\dots+x_{s+1} + \psi_1 (0,\mathbf{y})),\] and so all the functions $g_j$ other than $g_1$ have  been eliminated. Experienced readers will note that, so far, we have been following \cite[Appendix C]{GT10} almost verbatim.\\

After applying H\"{o}lder's inequality to (\ref{after Gowers cauchy Schwarz}), we see that to establish (\ref{equation first omega product}) it suffices to prove
\begin{equation}
\label{very tricky thing involving inequalities}
\frac{1}{N^{d-s-1}}\int\limits_{\mathbf{y} \in\mathbb{R}^{d-s-1}}\Vert G_{[s+1],\mathbf{y}}\Vert^{2^{s+1}}_{\square^{[s+1]}(\theta_{\mathbf{y}}; \mu_{[s+1]})}\prod\limits_{k=1}^{d-s-1} b_k(y_{k}/N) \, d\mathbf{y}=o(1),
\end{equation}
and, for all $B\subsetneq [s+1]$,
\begin{equation}
\label{slightly less tricky thing involving inequalities}
\frac{1}{N^{d-s-1}}\int\limits_{\mathbf{y}\in\mathbb{R}^{d-s-1}} \Vert \theta_{B,\mathbf{y}}\Vert^{2^{\vert B\vert}}_{\square^B(\theta_{\mathbf{y}}; \mu_B)}\prod\limits_{k=1}^{d-s-1} b_k(y_{k}/N) \, d\mathbf{y} \ll 1.
\end{equation}
\noindent These two expressions correspond respectively to expressions (C.10) and (C.11) of \cite{GT10}. \\

Establishing (\ref{slightly less tricky thing involving inequalities}) is straightforward. Indeed, we expand the left-hand side, yielding (up to a multiplicative constant factor) the expression 
\begin{align}
\label{equation expanded out nu term}
\frac{1}{N^{d-s-1}} \int\limits_{\mathbf{y}\in \mathbb{R}^{d-s-1}} \frac{1}{N^{2\vert B\vert }} \int\limits_{\mathbf{x_{B}^{(0)}},\mathbf{x_{B}^{(1)}} \in \mathbb{R}^B} \prod\limits_{C\subseteq B} \prod\limits_{\boldsymbol{\omega_C} \in \{0,1\}^C} \prod\limits_{i\in [t]: \Omega(i) = C} (\nu_{N,w} \ast \chi)(\psi_i(\mathbf{x_C^{(\boldsymbol{\omega_C})}}, \mathbf{y})+ a_i) \nonumber \\
 \prod\limits_{j\in B} b_j(x_j^{(0)}/N)b_j(x_j^{(1)}/N)\prod\limits_{k=1}^{d-s-1} b_{k+s+1}(y_{k}/N) \, d\mathbf{x_{B}^{(0)}} \,d\mathbf{x_{B}^{(1)}} \, d\mathbf{y}.
\end{align} 

As noted in \cite[p. 1824]{GT10}, the system of forms given by \[ (\mathbf{y},\mathbf{x_B^{(0)}},\mathbf{x_B^{(1)}}) \mapsto \psi_i(\mathbf{x_C^{(\boldsymbol{\omega_C})}},\mathbf{y}),\] for each $C\subseteq B$, $\bo_\mathbf{C} \in \{0,1\}^C$ and $i \in [t]$ such that $\Omega(i) = C$, has finite Cauchy-Schwarz complexity (since $\Psi$ does). We may therefore apply the upper bound in Corollary \ref{Corollary more upper bounds} to expression (\ref{equation expanded out nu term}), and this immediately yields (\ref{slightly less tricky thing involving inequalities}).  \\

 It remains to prove (\ref{very tricky thing involving inequalities}), which will be a much more major undertaking. We introduce some space-saving notation, namely for any subset $B\subseteq [s+1]$ we define the indexing set \[I_B: = \{ (C,\boldsymbol{\omega_C},i): C\subsetneq B, \, \boldsymbol{\omega_C}\in \{0,1\}^C,\, \Omega(i) = C\}.\] If a product is taken over triples $\mathfrak{t} \in I_B$, we interpret $C$, $\bo_C$ and $i$ as coming from the triple $\mathfrak{t} = (C,\bo_C,i)$. For notational expedience we will also identify the space $\mathbb{R}^{I_B}$ with the space $\mathbb{R}^{\vert I_B\vert}$. 
 
 With this notation, the left-hand side of (\ref{very tricky thing involving inequalities}) expands as 
 \begin{align}
 \label{expanding tricky}
 &\frac{1}{N^{d+s+1}}\int\limits_{\substack{\mathbf{y}\in\mathbb{R}^{d-s-1} \\\mathbf{x_{[s+1]}^{(0)}},\mathbf{x_{[s+1]}^{(1)}} \in\mathbb{R}^{s+1}}} \Big(\prod\limits _{\boldsymbol{\omega}\in \{0,1\}^{s+1}} g_1\Big(\sum\limits_{j=1}^{s+1}x_j^{(\omega_j)}+\psi_1(\mathbf{0},\mathbf{y})\Big)\Big) \nonumber \\ 
 &\Big(\prod\limits_{\mathfrak{t} \in I_{[s+1]}} (\nu_{N,w}^\gamma \ast \chi)(\psi_i(\mathbf{x_C^{(\boldsymbol{\omega_C})}}, \mathbf{y}) + a_i)\Big) b((\mathbf{x_{[s+1]}^{(0)}},\mathbf{x_{[s+1]}^{(1)}},\mathbf{y})/N) \,d\mathbf{x_{[s+1]}^{(0)}}\,d\mathbf{x_{[s+1]}^{(1)}} \, d\mathbf{y}.
 \end{align}

We make the substitution $\mathbf{h}: = \mathbf{x_{[s+1]}^{(1)}} - \mathbf{x_{[s+1]}^{(0)}}$ and $z: = x_1^{(0)} + \dots + x_{s+1}^{(0)} + \psi_1(\mathbf{0},\mathbf{y})$. Given $\mathbf{h}$, $z$, $\mathbf{x_{[s]}^{(0)}}$ and $\mathbf{y}$ one can recover $\mathbf{x_{[s+1]}^{(0)}}$, $\mathbf{x_{[s+1]}^{(1)}}$ and $\mathbf{y}$, so the change of variables is invertible. Therefore we may bound (\ref{expanding tricky}) above by a constant (the Jacobian of the change of variables) times 
\begin{equation}
\label{expression with P}
\Big\vert\frac{1}{N^{s+2}}\int\limits_{(z,\mathbf{h}) \in \mathbb{R}^{s+2}} P_{\mathbf{a},N}(z,\mathbf{h}) \Big(\prod\limits_{\boldsymbol{\omega}\in\{0,1\}^{s+1}}g_1(z + \sum\limits_{j=1}^{s+1}\omega_j h_j)\Big) \, dz\, d\mathbf{h}\Big\vert
\end{equation}
where $P_{\mathbf{a},N}(z,\mathbf{h})$ is equal to 
\begin{equation}
\label{equation defining P} \frac{1}{N^{d-1}}\int\limits_{\mathbf{(x_{[s]}^{(0)}},\mathbf{y})\in \mathbb{R}^{d-1}} \Big(\prod\limits_{\mathfrak{t} \in I_{[s+1]}} (\nu_{N,w}^\gamma \ast \chi) (\varphi_{\mathfrak{t}}(z,\mathbf{h},\mathbf{x_{[s]}^{(0)}},\mathbf{y}) + a_i)\Big) b((z,\mathbf{h},\mathbf{x_{[s]}^{(0)}},\mathbf{y})/N) \, d\mathbf{x_{[s]}^{(0)}} \, d\mathbf{y}
\end{equation} for some linear functions $\varphi_{\mathfrak{t}} : \mathbb{R}^{d+s+1}\longrightarrow \mathbb{R}$.

To be precise, if $\mathfrak{t} = (C,\boldsymbol{\omega},i)$ then the expression $\varphi_{\mathfrak{t}}(z,\mathbf{h},\mathbf{x_{[s]}^{(0)}},\mathbf{y})$ is equal to \[ \sum\limits_{k=1}^{s} (\psi_i(\mathbf{e_k}) - \psi_i(\mathbf{e_{s+1}}))x_k^{(0)} -\psi_i(\mathbf{e_{s+1}}) \psi_1(0,\mathbf{y}) + \psi_i(0,\mathbf{y}) + c(z,\mathbf{h})_{\mathfrak{t}},\] where \[ c(z,\mathbf{h})_{\mathfrak{t}} = \psi_i(\mathbf{e_{s+1}})z + \sum\limits_{k=1}^{s+1} \psi_i(\mathbf{e_k})\omega_k  h_k.\] This expression is analogous to expression (C.14) of \cite{GT10}. We let $\mathbf{c}(z,\mathbf{h}) \in \mathbb{R}^{I_{[s+1]}}$ denote the vector $(c(z,\mathbf{h})_{\mathfrak{t}})_{\mathfrak{t} \in I_{[s+1]}}$. Most fortunately, the exact structure of the linear maps $\varphi_{\mathfrak{t}}$, save for the fact that they form a system with finite Cauchy-Schwarz complexity, will be unimportant. \\

Following the philosophy of \cite{GT08} and \cite{GT10}, our next manoeuvre will be to replace $P_{\mathbf{a},N}(z,\mathbf{h})$ with a simpler function. To that end, let $w^*: \mathbb{N} \longrightarrow \mathbb{R}_{\geqslant 0}$ be a function for which $w^*(N) \rightarrow \infty$ as $N\rightarrow \infty$ and $w^*(n) \leqslant w(n)$ for all $n$. Recall from Section \ref{section conventions} that $W^* = W^*(N) = \prod_{p\leqslant w^*(N)} p$. 

\begin{Lemma}[Comparing $P_{\mathbf{a},N}(z,\mathbf{h})$ and $Q_{\mathbf{a},N}(z,\mathbf{h})$]
\label{Lemma comparing P and Q}
Define $Q_{\mathbf{a},N}(z,\mathbf{h})$ to be equal to

\begin{equation} \label{equation defining Q z,h}
\frac{1}{N^{d-1}}\int\limits_{\mathbf{(x_{[s]}^{(0)}},\mathbf{y})\in \mathbb{R}^{d-1}} \Big(\prod\limits_{\mathfrak{t} \in I_{[s+1]}} (\Lambda_{\mathbb{Z}/W^*\mathbb{Z}} \ast \chi) (\varphi_{\mathfrak{t}}(z,\mathbf{h},\mathbf{x_{[s]}^{(0)}},\mathbf{y}) + a_i)\Big) b((z,\mathbf{h},\mathbf{x_{[s]}^{(0)}},\mathbf{y})/N) \, d\mathbf{x_{[s]}^{(0)}} \, d\mathbf{y},
\end{equation} where $b((z,\mathbf{h},\mathbf{x_{[s]}^{(0)}},\mathbf{y})/N)$ here denotes the same function as is present in (\ref{equation defining P}). Then expression (\ref{expression with P}) is equal to 
\begin{equation}
\frac{1}{N^{s+2}}\int\limits_{(z,\mathbf{h}) \in \mathbb{R}^{s+2}} Q_{\mathbf{a},N}(z,\mathbf{h}) \Big(\prod\limits_{\boldsymbol{\omega}\in\{0,1\}^{s+1}}g_1(z + \sum\limits_{j=1}^{s+1}\omega_j h_j)\Big) \, dz\, d\mathbf{h} + o(1),
\end{equation}
\noindent where the $o(1)$ may depend on the function $w^*$. 
\end{Lemma}
\begin{proof}
Considering the upper bound $\vert g_1(x)\vert \leqslant (\nu_{N,w}^\gamma \ast \chi) (x + a_1)$, it suffices to show that \[ \frac{1}{N^{s+2}} \int\limits_{(z,\mathbf{h}) \in \mathbb{R}^{s+2}} \vert P_{\mathbf{a},N}(z,\mathbf{h}) - Q_{\mathbf{a},N}(z,\mathbf{h})\vert \Big(\prod\limits_{\bo \in \{0,1\}^{s+1}} (\nu_{N,w}^\gamma \ast \chi)(z + \sum\limits_{j=1}^{s+1} \omega_jh_j + a_1)\Big) \, dz \, d\mathbf{h}\] is $o(1)$. By Cauchy-Schwarz, it then suffices to show that both
\begin{equation}
\label{expression without P and Q}
\frac{1}{N^{s+2}} \int\limits_{\substack{(z,\mathbf{h}) \in \mathbb{R}^{s+2} \\ \Vert (z,\mathbf{h})\Vert_\infty \ll N}}\Big(\prod\limits_{\bo \in \{0,1\}^{s+1}} (\nu_{N,w}^\gamma \ast \chi)(z + \sum\limits_{j=1}^{s+1} \omega_jh_j + a_1) \Big)\, dz \, d\mathbf{h} \ll 1
\end{equation}
\noindent and
\begin{equation}
\label{expression with P and Q}
\frac{1}{N^{s+2}}\int\limits_{(z,\mathbf{h}) \in \mathbb{R}^{s+2}} (P_{\mathbf{a},N}(z,\mathbf{h}) - Q_{\mathbf{a},N}(z,\mathbf{h}))^2 \Big(\prod\limits_{\boldsymbol{\omega}\in\{0,1\}^{s+1}}(\nu_{N,w}^\gamma \ast \chi)(z + \sum\limits_{j=1}^{s+1}\omega_j h_j + a_1) \Big)\, dz\, d\mathbf{h} = o(1).
\end{equation}

The bound (\ref{expression without P and Q}) is immediate from Corollary \ref{Corollary more upper bounds}. To prove (\ref{expression with P and Q}), expanding out the square we must consider three expressions. One of them is \begin{equation}
\label{one of them is}
\frac{1}{N^{s+2}}\int\limits_{(z,\mathbf{h}) \in \mathbb{R}^{s+2}} P_{\mathbf{a},N}(z,\mathbf{h})^2 \Big(\prod\limits_{\boldsymbol{\omega}\in\{0,1\}^{s+1}}(\nu_{N,w}^\gamma\ast \chi)(z + \sum\limits_{j=1}^{s+1}\omega_j h_j + a_1)\Big) \, dz\, d\mathbf{h}.
\end{equation} When multiplied out, (\ref{one of them is}) is equal to the large expression
\begin{align}
\frac{1}{N^{2d +s}}\int\limits_{\substack{(z,\mathbf{h}) \in \mathbb{R}^{s+2}\\\mathbf{(x_{[s]}^{(0)}},\mathbf{y})\in \mathbb{R}^{d-1}\\ \mathbf{(\widetilde{x}_{[s]}^{(0)}},\mathbf{\widetilde{y}})\in \mathbb{R}^{d-1}}}\Big(\prod\limits_{\mathfrak{t} \in I_{[s+1]}} (\nu_{N,w}^\gamma \ast \chi) (\varphi_{\mathfrak{t}}(z,\mathbf{h},\mathbf{x_{[s]}^{(0)}},\mathbf{y}) + a_i)(\nu_{N,w}^\gamma \ast \chi) (\varphi_{\mathfrak{t}}(z,\mathbf{h},\mathbf{\widetilde{x}_{[s]}^{(0)}},\mathbf{\widetilde{y}}) + a_i) \Big)\nonumber \\\Big(\prod\limits_{\boldsymbol{\omega}\in\{0,1\}^{s+1}}(\nu_{N,w}^\gamma \ast \chi)(z + \sum\limits_{j=1}^{s+1}\omega_j h_j + a_1)\Big)b((z,\mathbf{h},\mathbf{x_{[s]}^{(0)}},\mathbf{y})/N)
b((z,\mathbf{h},\mathbf{\widetilde{x}_{[s]}^{(0)}},\mathbf{\widetilde{y}})/N)\nonumber \\ d\mathbf{x_{[s]}^{(0)}}\, d\mathbf{\widetilde{x}_{[s]}^{(0)}}  \, d\mathbf{y} \, d\mathbf{\widetilde{y}} \, dz \, d\mathbf{h}.
\end{align}

\noindent By applying Corollary \ref{Corollary switching functions} to the above expression, we may replace the functions $\nu_{N,w}^\gamma \ast \chi$ with $\Lambda_{\mathbb{Z}/W^*\mathbb{Z}} \ast \chi$, up to an $o(1)$ error.

It is worth noting why the application of Corollary \ref{Corollary switching functions} is valid. Indeed, the underlying set of linear forms is given by (for each $\mathfrak{t} \in I_{[s+1]}$) \[ (z,\mathbf{h},\mathbf{x_{[s]}^{(0)}},\mathbf{y}, \mathbf{\widetilde{x}_{[s]}^{(0)}},\mathbf{\widetilde{y}}) \mapsto (\varphi_{\mathfrak{t}}(z,\mathbf{h},\mathbf{x_{[s]}^{(0)}},\mathbf{y}),\varphi_{\mathfrak{t}}(z,\mathbf{h},\mathbf{\widetilde{x}_{[s]}^{(0)}},\mathbf{\widetilde{y}})).\] We need this linear map to have algebraic coefficients and to have finite Cauchy-Schwarz complexity. Algebraicity follows by the assumptions in the statement of Theorem \ref{Theorem Cauchy}. Establishing finite Cauchy-Schwarz complexity is rather involved, but fortunately this has already been done by Green and Tao, on pages 1826 and 1827 of \cite{GT10}, in the analysis of expression (C.14).

Replacing (\ref{one of them is}) with one of the other two terms that arises from expanding out the square in (\ref{expression with P and Q}), and performing the same estimation, the lemma follows. 
\end{proof}

Let us take stock. As a reminder, we are trying to establish that (\ref{very tricky thing involving inequalities}) holds. Lemma \ref{Lemma comparing P and Q} above reduces matters to choosing some function $w^*$ that tends to infinity for which the bound
\begin{equation}
\label{nearly nearly there}
\frac{1}{N^{s+2}}\Big\vert\int\limits_{(z,\mathbf{h})\in \mathbb{R}^{s+2}} Q_{\mathbf{a},N}(z,\mathbf{h})\prod\limits_{\boldsymbol{\omega}\in\{0,1\}^{s+1}}g_1(z + \sum\limits_{k=1}^{s+1}\omega_k h_k + a_1) dz \, d\mathbf{h}\Big\vert =o(1)
\end{equation} holds. If $Q_{\mathbf{a},N}(z,\mathbf{h})$ were identically equal to $1$, then expression (\ref{nearly nearly there}) would be of the order of $\Vert g_1\Vert_{U^{s+1}(\mathbb{R},2N)}$, and hence be $o(1)$ by the hypotheses of Theorem \ref{Theorem Cauchy}. Of course $Q_{\mathbf{a},N}(z,\mathbf{h})$ is not identically equal to $1$, but we do observe that $Q_{\mathbf{a},N}(z,\mathbf{h})$ is a function of the form considered in Lemma \ref{Lemma approximation of Q}. Indeed, consulting the definition of $Q_{\mathbf{a},N}(z,\mathbf{h})$ in (\ref{equation defining Q z,h}), the following table shows which objects in Lemma \ref{Lemma approximation of Q} correspond to which objects concerning the definition of $Q_{\mathbf{a},N}(z,\mathbf{h})$. 
\begin{center}
\begin{tabular}{c|c}
Lemma \ref{Lemma approximation of Q} & (\ref{equation defining Q z,h})\\
\hline
$\mathbf{a}$ & $\mathbf{a}$ \\
$ \mathbf{x}$ & $(\mathbf{x_{[s]}^{(0)}}, \mathbf{y})$ \\
$ \mathbf{y}$ & $ (z, \mathbf{h})$ \\
$\varphi_j(\mathbf{x})$ & $\varphi_{\mathfrak{t}}(0, \mathbf{0}, \mathbf{x_{[s]}^{(0)}}, \mathbf{y})$ \\
$\Psi(\mathbf{y})$ & $ \mathbf{c}(z,\mathbf{h})$
\end{tabular}
\end{center} 

From Lemma \ref{Lemma approximation of Q}, we therefore know that there exists some function $f_1:\mathbb{Z}^{\vert I_{[s+1]}\vert } \longrightarrow \mathbb{C}$ satisfying $\Vert f_1\Vert_\infty \ll (\log \log W^*)^{O(1)}$ for which \[ Q_{\mathbf{a},N}(z,\mathbf{h}) =b_{\mathbf{a},N}((z,\mathbf{h})/N) \sum\limits_{\Vert \mathbf{k}\Vert_\infty \leqslant (\log \log W^*)^{O(1)}} f_1(\mathbf{k}) e\Big(\frac{\mathbf{k} \cdot(\mathbf{c}(z,\mathbf{h}) + \mathbf{a})}{W^*}\Big) + o(1).\] 

Therefore, one gets an upper bound for the left-hand side of (\ref{nearly nearly there}), namely
\begin{align}
\label{equation short exponential sum imput}
&\frac{(\log \log W^*)^{O(1)} }{N^{s+2}} \times \nonumber \\ &\sup\limits_{\mathbf{k} \in \mathbb{Z}^{\vert I_{[s+1]}\vert}}\Big\vert  \int\limits_{(z,\mathbf{h}) \in \mathbb{R}^{s+2}} e\Big(\frac{\mathbf{k} \cdot \mathbf{c}(z,\mathbf{h})}{W^*}\Big)b_{\mathbf{a},N}((z,\mathbf{h})/N)\prod\limits_{\boldsymbol{\omega}\in \{0,1\}^{s+1}} g_1(z+ \sum\limits_{j=1}^{s+1} \omega_j h_j + a_1) \, dz\, d\mathbf{h}\Big\vert 
\end{align} plus an error of size
\begin{equation}
\label{plus an error of size 1}
o(1) \times \frac{1}{N^{s+2}} \int\limits_{\substack{(z,\mathbf{h}) \in \mathbb{R}^{s+2} \\ \Vert (z,\mathbf{h})\Vert_\infty \ll N}} \prod\limits_{\bo \in \{0,1\}^{s+1}} (\nu_{N,w}^\gamma \ast \chi) (z + \sum\limits_{k=1}^{s+1} \omega_kh_k + a_1) \, dz \, d\mathbf{h}.
 \end{equation} By Corollary \ref{Corollary more upper bounds}, the size of term (\ref{plus an error of size 1}) is $o(1)$. To analyse (\ref{equation short exponential sum imput}) we apply Lemma B.4 of \cite{Wa17}. Since the function $b_{\mathbf{a},N}$ is Lipschitz this means that for all $Y>2$ there exists a complex valued function $f_{\mathbf{a},N,2}$ such that $\Vert f_{\mathbf{a},N,2}\Vert_\infty \ll 1$ and for all $(z,\mathbf{h})$ one has \[ b_{\mathbf{a},N}((z,\mathbf{h})/N) =  \int\limits_{\Vert\mathbf{y}\Vert_\infty \leqslant Y} f_{\mathbf{a},N,2}(\mathbf{y}) e\Big(\frac{\mathbf{y} \cdot (z,\mathbf{h})}{N}\Big) \, d\mathbf{x} + O((\log Y) / Y).\] Choosing $Y$ to be a suitably large power of $\log\log W^*$, (\ref{equation short exponential sum imput}) may be bounded above by \begin{align}
 \label{equation short exponential sum imput second}
 &\frac{(\log\log W^*)^{O(1)} }{N^{s+2}} \times \nonumber \\ &\sup\limits_{\mathbf{k},\mathbf{y}}\Big\vert  \int\limits_{(z,\mathbf{h}) \in \mathbb{R}^{s+2}} e\Big(\frac{\mathbf{k} \cdot \mathbf{c}(z,\mathbf{h})}{W^*}\Big)e\Big(\frac{\mathbf{y} \cdot (z,\mathbf{h})}{N}\Big)\prod\limits_{\boldsymbol{\omega}\in \{0,1\}^{s+1}} g_1(z+ \sum\limits_{j=1}^{s+1} \omega_j h_j + a_1) \, dz\, d\mathbf{h}\Big\vert 
 \end{align}
 \noindent plus an error of size 
 \begin{equation}
 \label{plus a second error of size}
 o(1) \times \frac{1}{N^{s+2}} \int\limits_{\substack{(z,\mathbf{h}) \in \mathbb{R}^{s+2} \\ \Vert (z,\mathbf{h})\Vert_\infty \ll N}} \prod\limits_{\bo \in \{0,1\}^{s+1}} (\nu_{N,w}^\gamma \ast \chi) (z + \sum\limits_{k=1}^{s+1} \omega_kh_k + a_1) \, dz \, d\mathbf{h}.
  \end{equation}
  \noindent Using Corollary \ref{Corollary more upper bounds} as above, expression (\ref{plus a second error of size}) is $o(1)$. \\

   The term (\ref{equation short exponential sum imput second}) may be analysed using the standard methods. Indeed, by shifting the variable $z$ (and noting that $\mathbf{c}(z,\mathbf{h})$ is a linear function of $z$ and $\mathbf{h}$) we may assume that $a_1 = 0$. Then, by spreading the exponential functions across the different instances of $g_1$, we see it suffices to show that
 \begin{equation}
 \label{the above expression}
\frac{(\log \log W^*)^{O(1)}}{N^{s+2}}  \Big\vert \int\limits_{(z,\mathbf{h}) \in \mathbb{R}^{s+2}} \prod\limits_{\boldsymbol{\omega}\in \{0,1\}^{s+1}} g_{\bo}(z+ \sum\limits_{k=1}^{s+1} \omega_k h_k ) \, dz\, d\mathbf{h}\Big\vert = o(1), 
\end{equation} where each function $g_{\bo}$ is of the form \[ g_{\bo}(x): = g_1(x)e(\lambda_{\bo} x),\] for some $\lambda_{\bo} \in \mathbb{R}$. 
 
The argument is nearly complete. Considering expression (\ref{real gowers norm}), for each $\bo$ we observe that \[\Vert g_{\bo}\Vert _{U^{s+1}(\mathbb{R},2N)} = \Vert g_1 \Vert_{U^{s+1}(\mathbb{R},2N)}.\] So, by the Gowers-Cauchy-Schwarz inequality (recorded in this setting as Proposition A.4 of \cite{Wa17}), the left-hand side of expression (\ref{the above expression}) is \[O((\log \log W^*)^{O(1)}\Vert g_1 \Vert_{U^{s+1}(\mathbb{R},2N)}^{O(1)}).\] If $w^*$ grows slowly enough, this expression is $o(1)$.\\

We have therefore established the upper bound (\ref{nearly nearly there}), and so, by our long sequence of deductions, Theorem \ref{Theorem Cauchy} is finally proved. 
\end{proof}

\section{Combining the lemmas}
\label{section combining the lemmas}

With all the previous lemmas in hand, we may finally prove Theorem \ref{Theorem generalised von neumann} (and hence prove Theorem \ref{Main theorem}).

\begin{proof}[Proof of Theorem \ref{Theorem generalised von neumann}]
Assume the hypotheses of the theorem, fixing a suitably small value of $\gamma$. 

By applying Proposition \ref{Proposition separating out the kernel} and Proposition \ref{Proposition parametrising by normal form}, we conclude that there is some $s = O(1)$ and some $d^\prime = O(1)$ for which $\vert \widetilde{T}_{F,G,N}^{L,\mathbf{v}}(f_1\ast \chi,\dots,f_d\ast \chi)\vert$ is 
\begin{equation}
\label{final equation of all time}
 \ll_{L,\varepsilon} \Big\vert\frac{1}{N^{d^\prime}}\int\limits_{\mathbf{x} \in \mathbb{R}^{d^\prime}}F_2(\mathbf{x}/N)\prod\limits_{j=1}^d (f_j \ast \chi)(\psi_j^\prime(\mathbf{x}) + a_j) \, d\mathbf{x},\Big\vert,
 \end{equation} where $(\psi_1^\prime,\dots,\psi_d^\prime) = \Psi^\prime:\mathbb{R}^{d^\prime}\longrightarrow \mathbb{R}^d$ is in $s$-normal form, $F_2: \mathbb{R}^{d^\prime} \longrightarrow [0,1]$ has Lipschitz constant $O_L(\sigma^{-1})$ and $\Rad(F_2) = O_{C,L,\varepsilon}(1)$, and each $a_j$ satisfies $\vert a_j\vert = O_{C,L,\varepsilon}(N)$. Taking this value of $s$ in the hypotheses of Theorem \ref{Theorem generalised von neumann}, without loss of generality we may assume that 
 \begin{equation}
\label{gowers norm decay}
\Vert f_1\Vert_{U^{s+1}[N]} = o(1)
\end{equation} as $N\rightarrow \infty$. 
 
Then we may apply Theorem \ref{Theorem Cauchy} to expression (\ref{final equation of all time}). Indeed, by rescaling the variable $\mathbf{x}$ we may assume that $F_2$ is supported on $[-1,1]^{d^\prime}$. For each $j\in [d]$ we set \[ g_j : = f_j \ast \chi.\] Provided $\eta$ is small enough, by combining (\ref{gowers norm decay}) and Lemma \ref{Lemma linking different Gowers norms} we deduce that \[\Vert g_1 \Vert_{U^{s+1}(\mathbb{R},2N)} = o_{\eta}(1)\] as $N \rightarrow \infty$.  So Theorem \ref{Theorem Cauchy} may indeed be applied, which yields \begin{equation}
\label{final}
\vert\widetilde{T}_{F,G,N}^{L,\mathbf{v}}(f_1\ast \chi,\dots,f_d\ast \chi)\vert = o_{C,L,\gamma,\varepsilon,\eta,\sigma} (1)
\end{equation} as $N\rightarrow \infty$.
 
 But then, combining the estimate (\ref{final}) with Lemma \ref{Lemma transfer equation}, one derives the bound
\begin{equation}
\label{complicated right hand side}
\vert T_{F,G,N}^{L,\mathbf{v}}(f_1,\dots,f_d)\vert = O_{C,L,\gamma,\varepsilon,\sigma}(\eta) + o_{C,L,\gamma,\varepsilon,\eta,\sigma}(1).
\end{equation} Choosing $\eta = \eta(N)$ to be a function tending to zero suitably slowly with $N$, we conclude that \[ \vert T_{F,G,N}^{L,\mathbf{v}}(f_1,\dots,f_d)\vert = o_{C,L,\gamma,\varepsilon,\sigma}(1).\] This is the conclusion of Theorem \ref{Theorem generalised von neumann}, and we are done. 
\end{proof}

\noindent From the work in Section \ref{section Controlling by Gowers norms}, this means that Theorem \ref{Main theorem}, the main result of this paper, is finally settled. \qed

\part{Final deductions}
\label{part final deductions}
\section{Removing Lipschitz cut-offs}
\label{section removing sharp cut-offs}
In this section we assume Theorem \ref{Main theorem}, and deduce Theorem \ref{Main theorem simpler version}. This deduction will be a routine matter of removing Lipschitz cut-offs. 

\begin{Lemma}
\label{Lemma upper bound for short intervals}
Assume the hypotheses of Theorem \ref{Main theorem simpler version}. Let $\delta$ be a real number in the range $0<\delta<1/2$ and let $I \subset [0,1]$ be an interval of length $\delta $. Then 
\[ \frac{1}{N^{d-m}} \sum\limits_{i=1}^d\sum\limits_{\substack{\mathbf{n}\in [N]^d \\n_i \in N\cdot I }} \Big(\prod\limits_{j=1}^d \Lambda^\prime(n_j)\Big) 1_{[-\varepsilon,\varepsilon]^m}(L\mathbf{n} + \mathbf{v}) \ll_{L} \delta\varepsilon^m + o_{C,L,\delta,\varepsilon}(1).\]
\end{Lemma}
\noindent The reader will note that this lemma is a slight refinement of Corollary \ref{Corollary upper bound}.
\begin{proof}
Fix some $i\leqslant d$. Let $F:\mathbb{R}^d \longrightarrow [0,1]$ be a smooth function in $\mathcal{C}(\delta)$, supported on $\{\mathbf{x} \in [-1,2]^d: x_i \in I + [-\delta,\delta]\}$, that majorises the indicator function of the set $\{ \mathbf{x} \in [0,1]^d: x_i \in I\}$. Let $G:\mathbb{R}^m \longrightarrow [0,1]$ be some smooth function in $\mathcal{C}(\varepsilon)$, supported on $[-2\varepsilon,2\varepsilon]^m$, that majorises $1_{[-\varepsilon,\varepsilon]^m}$. Let $\gamma$ be small enough in terms of $L$. Then, by Theorem \ref{Theorem pseudorandomness} and Lemma \ref{Lemma problem for local von Mangoldt},
\begin{align}
\label{start of final home straight}
\frac{1}{N^{d-m}} \sum\limits_{\substack{\mathbf{n}\in [N]^d \\n_i \in N \cdot I }} \Big(\prod\limits_{j=1}^d \Lambda^\prime(n_j)\Big) 1_{[-\varepsilon,\varepsilon]^m}(L\mathbf{n} + \mathbf{v})& \ll_\gamma \frac{1}{N^{d-m}} \sum\limits_{\substack{\mathbf{n}\in [N]^d \\n_i \in N \cdot I}} \Big(\prod\limits_{j=1}^d \nu_{N,w}^{\gamma}(n_j)\Big) 1_{[-\varepsilon,\varepsilon]^m}(L\mathbf{n} + \mathbf{v})\nonumber \\
& \leqslant T_{F,G,N}^{L,\mathbf{v}}(\nu_{N,w}^\gamma,\dots,\nu_{N,w}^\gamma) \nonumber \\
& = T_{F,G,N}^{L,\mathbf{v}}(\Lambda_{\mathbb{Z}/W\mathbb{Z}},\dots,\Lambda_{\mathbb{Z}/W\mathbb{Z}}) + o_{C,L,\gamma,\delta,\varepsilon}(1)\nonumber \\
& = \int\limits_{\mathbf{x} \in \mathbb{R}^d}F(\mathbf{x}/N) G(L\mathbf{x} + \mathbf{v}) \, d\mathbf{x} + o_{C,L,\gamma,\delta,\varepsilon}(1) .
\end{align}

Since $L \notin V_{\degen}^\ast(m,d)$, for all $d$ of the coordinate subspaces $U \leqslant \mathbb{R}^d$ of dimension $d-1$ the map $L|_U: U \longrightarrow \mathbb{R}^m$ is surjective. We may therefore apply Lemma \ref{Lemma crude upper bound lemma}, and conclude that expression (\ref{start of final home straight}) is $O_L(\delta \varepsilon^m) + o_{C,L,\delta,\varepsilon,\gamma}(1)$. The lemma is proved, after having fixed a suitable $\gamma$. 
\end{proof}

\begin{Lemma}
\label{Lemma logged form of simple theorem}
Under the hypotheses of Theorem \ref{Main theorem simpler version},
\[\frac{1}{N^{d-m}}\sum\limits_{\mathbf{n} \in [N]^d} \Big(\prod\limits_{j=1}^d \Lambda^\prime(n_j)\Big) 1_{[-\varepsilon,\varepsilon]^m}(L\mathbf{n} + \mathbf{v}) = \frac{1}{N^{d-m}}\int\limits_{\mathbf{x} \in [0,N]^d} 1_{[-\varepsilon,\varepsilon]^m}(L\mathbf{x} + \mathbf{v}) \, d\mathbf{x} + o_{C,L,\varepsilon}(1).\]
\end{Lemma}
\begin{proof}
Let $\delta$ be a positive parameter in the range $(0,1/2)$, to be chosen later. Let us first consider \[ \frac{1}{N^{d-m}}\sum\limits_{\substack{\mathbf{n} \in  \mathbb{Z}^d \\ \mathbf{n} \in [\delta N,(1-\delta)N]^d}} \Big(\prod\limits_{j=1}^d \Lambda^\prime(n_j)\Big) 1_{[-\varepsilon,\varepsilon]^m}(L\mathbf{n} + \mathbf{v}) .\] Let $F^{\pm \delta}:\mathbb{R}^d \longrightarrow [0,1]$ be two Lipschitz functions satisfying \[1_{[3\delta/2 ,1-3\delta/2 ]^d}\leqslant F^{-\delta}\leqslant 1_{[\delta,1-\delta]^d}\leqslant F^{+\delta}\leqslant 1_{[\delta/2 ,1-\delta/2 ]^d},\] with Lipschitz constants depending only on $\delta$. Let $G^{\pm \delta}: \mathbb{R}^{m} \longrightarrow [0,1]$ be two Lipschitz functions satisfying \[1_{[-\varepsilon(1-\delta) ,\varepsilon(1 -\delta) ]^m}\leqslant G^{-\delta}\leqslant 1_{[-\varepsilon,\varepsilon]^m}\leqslant G^{+\delta}\leqslant 1_{[-\varepsilon(1 +\delta) ,\varepsilon(1 +\delta)]^m},\] with Lipschitz constants\footnote{The existence of such functions is immediate by interpolating linearly, or by appealing to the results of Section \ref{section smooth functions}.} depending only on $\delta$. Then we have
\begin{align}
\label{sandwiching expression}
\sum\limits_{\mathbf{n} \in \mathbb{Z}^d}\Big(\prod\limits_{j=1}^d \Lambda(n_j)\Big) F^{-\delta}(\mathbf{n}/N)  G^{-\delta}(L\mathbf{n} + \mathbf{v}) \leqslant \sum\limits_{\substack{\mathbf{n} \in  \mathbb{Z}^d \\ \mathbf{n} \in [\delta N,(1-\delta)N]^d}} \Big(\prod\limits_{j=1}^d \Lambda^\prime(n_j)\Big) 1_{[-\varepsilon,\varepsilon]^m}(L\mathbf{n} + \mathbf{v}) \nonumber\\
\leqslant \sum\limits_{\mathbf{n} \in \mathbb{Z}^d} \Big(\prod\limits_{j=1}^d \Lambda(n_j)\Big)F^{+\delta}(\mathbf{n}/N)  G^{+\delta}(L\mathbf{n} + \mathbf{v}).
\end{align}

By Theorem \ref{Main theorem}, the lower bound in (\ref{sandwiching expression}) is equal to \[ \sum\limits_{\mathbf{n} \in \mathbb{Z}^d} \Big(\prod\limits_{j=1}^d \Lambda_{\mathbb{Z}/W\mathbb{Z}}(n_j) \Big) F^{-\delta}(\mathbf{n}/N) G^{-\delta}(L\mathbf{n} + \mathbf{v}) + o_{C,L,\delta,\varepsilon}(N^{d-m}),\] since we may replace $\Lambda_{\mathbb{Z}/W\mathbb{Z}}^+$ with $\Lambda_{\mathbb{Z}/W\mathbb{Z}}$ as $F$ is supported on $[0,1]^d$. By Lemma \ref{Lemma problem for local von Mangoldt}, and the properties of the support of $F^{-\delta}$ and $G^{-\delta}$, this is at least 
\begin{equation}
\label{this is at most}
\int\limits_{\mathbf{x} \in [3\delta N/2,N( 1-3\delta/2)]^d} 1_{[-\varepsilon (1 -\delta),\varepsilon(1 +\delta) ]^m}(L\mathbf{x} + \mathbf{v}) \, d\mathbf{x} + o_{C,L,\delta,\varepsilon}(N^{d-m}).
\end{equation} Note that the singular series $\mathfrak{S}$ is equal to $1$ in this instance, since $L$ is purely irrational. By Lemma \ref{Lemma moving error terms}, expression (\ref{this is at most}) is at least \[ \int\limits_{\mathbf{x} \in [0,N]^d} 1_{[-\varepsilon,\varepsilon]^m} (L\mathbf{x} + \mathbf{v}) \, d\mathbf{x} - O(\delta\varepsilon^m N^{d-m})  + o_{C,L,\delta,\varepsilon}(N^{d-m}).\] 

By performing an analogous manipulation with the upper bound, we may conclude that \[\sum\limits_{\substack{\mathbf{n} \in  \mathbb{Z}^d \\ \mathbf{n} \in [\delta N,(1-\delta)N]^d}} \Big(\prod\limits_{j=1}^d \Lambda^\prime(n_j)\Big) 1_{[-\varepsilon,\varepsilon]^m}(L\mathbf{n} + \mathbf{v})\] is equal to
\begin{equation}
\label{expression msot well suited to removing log weighting} \int\limits_{\mathbf{x} \in [0,N]^d} 1_{[-\varepsilon,\varepsilon]^m}(L\mathbf{x} + \mathbf{v}) \, d\mathbf{x} + O(\delta\varepsilon^m N^{d-m}) + o_{C,L,\delta,\varepsilon}(N^{d-m}).
\end{equation}

Therefore, by Lemma \ref{Lemma upper bound for short intervals}, we have that  \[\frac{1}{N^{d-m}}\sum\limits_{\mathbf{n} \in [N]^d} \Big(\prod\limits_{j=1}^d \Lambda^\prime(n_j)\Big) 1_{[-\varepsilon,\varepsilon]^m}(L\mathbf{n} + \mathbf{v})  \] is equal to \[\frac{1}{N^{d-m}} \int\limits_{\mathbf{x} \in [0,N]^d} 1_{[-\varepsilon,\varepsilon]^m}(L\mathbf{x} + \mathbf{v}) \, d\mathbf{x} +O(\delta\varepsilon^m)  + o_{C,L,\delta,\varepsilon}(1).\] Letting $\delta$ be a function of $N$, tending to zero suitably slowly as $N$ tends to infinity, the lemma follows. 
\end{proof}

To establish Theorem \ref{Main theorem simpler version} as given, i.e. to establish Lemma \ref{Lemma logged form of simple theorem} without the log weighting, is standard. To spell it out, Lemma \ref{Lemma logged form of simple theorem} implies that, for any $\delta$ in the range $0<\delta < 1/2$, 
\begin{align}
\label{upper bound removing logs}
\sum\limits_{\mathbf{p} \in [\delta N,N]^d }1_{[-\varepsilon,\varepsilon]^m}(L\mathbf{p} + \mathbf{v}) \leqslant \frac{1}{(\log(\delta) + \log N)^d}\sum\limits_{\substack{\mathbf{n} \in  \mathbb{Z}^d \\ \mathbf{n} \in [\delta N,N]^d}} \Big(\prod\limits_{j=1}^d \Lambda^\prime(n_j)\Big) 1_{[-\varepsilon,\varepsilon]^m}(L\mathbf{n} + \mathbf{v}) \nonumber \\
 \leqslant \frac{(1 + o_\delta(1))}{( \log N)^d}\Big(\int\limits_{\mathbf{x} \in [0,N]^d} 1_{[-\varepsilon,\varepsilon]^m}(L\mathbf{x} + \mathbf{v})  \, d\mathbf{x}+ o_{C,L,\varepsilon}(N^{d-m})\Big).
\end{align} But also, from expression (\ref{expression msot well suited to removing log weighting})
\begin{align}
\label{lower bound removing logs}
\sum\limits_{\mathbf{p} \in [\delta N,N]^d} 1_{[-\varepsilon,\varepsilon]^m}(L\mathbf{p} + \mathbf{v}) \geqslant \frac{1}{(\log N)^d} \sum\limits_{\mathbf{n} \in [\delta N , (1-\delta)N]^d} \Big(\prod\limits_{j=1}^d \Lambda^\prime(n_j)\Big) 1_{[-\varepsilon,\varepsilon]^m}(L\mathbf{n} + \mathbf{v}) \nonumber \\
\geqslant \frac{1}{(\log N)^d} \Big(\int\limits_{\mathbf{x} \in [0,N]^d} 1_{[-\varepsilon,\varepsilon]^m}(L\mathbf{x} + \mathbf{v}) \, d\mathbf{x} + O(\delta \varepsilon^m N^{d-m}) + o_{C,L,\delta,\varepsilon}(N^{d-m})\Big).
\end{align}

By Lemma \ref{Lemma general upper bound}, \[\int\limits_{\mathbf{x} \in [0,N]^d} 1_{[-\varepsilon,\varepsilon]^m}(L\mathbf{x} + \mathbf{v}) \, d\mathbf{x} = O_{L,\varepsilon}(N^{d-m}).\] Hence, choosing $\delta$ to be a function of $N$ tending to zero suitably slowly, combining bounds (\ref{lower bound removing logs}) and (\ref{upper bound removing logs}) establishes Theorem \ref{Main theorem simpler version}. \qed

\part{Appendices}
\appendix
\section{Estimating integrals}
\label{section Easy calculations}
In this appendix we include the lemmas that help us estimate the `global factor' from Theorem \ref{Main theorem simpler version}, namely \[ \int\limits_{\mathbf{x} \in [0,N]^d} 1_{[-\varepsilon,\varepsilon]^m}(L\mathbf{x} + \mathbf{v}) \, d\mathbf{x}.\]

\begin{Lemma}[Upper bound]
\label{Lemma general upper bound}
Let $h$ be a natural number, let $m$ be a non-negative integer, and let $C,K$ be positive constants. Let $L:\mathbb{R}^h \longrightarrow \mathbb{R}^m$ be a surjective linear map. Let $F:\mathbb{R}^h \longrightarrow [0,1]$ and $G:\mathbb{R}^m \longrightarrow [0,1]$ be any compactly supported measurable functions, and assume that $F$ is supported on a box of the form $\mathbf{x^{(0)}} + [-C,C]^h$ and $G$ is supported on a box of the form $\mathbf{y^{(0)}} + [-K,K]^m$. Then 
\begin{equation}
\label{general integeral upper bound}
\int\limits_{\mathbf{x} \in \mathbb{R}^h} F(\mathbf{x}) G(L\mathbf{x}) \, d\mathbf{x} \ll_L C^{h-m} K^m \Vert F\Vert_\infty \Vert G\Vert_\infty.
\end{equation}
\end{Lemma}
\begin{proof}
Split $\mathbb{R}^h$ as a direct sum $(\ker L) \oplus (\ker L)^\perp$. Observe that $L|_{(\ker L)^\perp}$ is an injective linear map, so has bounded inverse. Hence the integrand in (\ref{general integeral upper bound}) is zero unless $\mathbf{x} |_{(\ker L)^\perp}$ is contained within a region which has volume $O_{L}(K^m)$. The integrand is also zero unless $\mathbf{x}|_{(\ker L)}$ is contained within a region which has volume $O_{L}(C^{h-m})$. Together, these observations combine to give the required bound. 
\end{proof}
\begin{Lemma}
\label{Lemma singular integral}
Let $N$, $m$, $d$ be natural numbers, with $d\geqslant m+1$, and let $\varepsilon$ be a positive parameter. Let $L: \mathbb{R}^d \longrightarrow \mathbb{R}^m$ be a surjective purely irrational linear map. Let $\mathbf{v} \in \mathbb{R}^m$ be any vector. Then there exists a parameter $C_{L,\mathbf{v}/N}$ satisfying $\vert C_{L,\mathbf{v}/N}\vert = O_{L}(1)$ such that 

\begin{equation}
\label{asymptotic for singular integral}
\int\limits_{\mathbf{x} \in [0,N]^d} 1_{[-\varepsilon,\varepsilon]^m}(L\mathbf{x} + \mathbf{v}) \, d\mathbf{x} = C_{L,\mathbf{v}/N}\varepsilon^mN^{d-m} + O_{L}(\varepsilon^{m+1}N^{d-m-1}).
\end{equation}
\noindent Furthermore, if $\Vert \mathbf{v}\Vert_\infty = o(N)$ then there exists a constant $C_L$, independent of $\mathbf{v}$ and $N$, for which \[C_{L,\mathbf{v}/N} = C_L + o_L(1).\]
\end{Lemma}
\begin{proof}
Since $L$ has rank $m$, without loss of generality we may assume that the first $m$ columns of $L$ form an invertible submatrix $M$. So \begin{align}
\int\limits_{\mathbf{x} \in [0,N]^d} 1_{[-\varepsilon,\varepsilon]^m}(L\mathbf{x} + \mathbf{v}) \, d\mathbf{x} &= N^d \int\limits_{\mathbf{x} \in [0,1]^d} 1_{ M^{-1}([-\frac{\varepsilon}{N},\frac{\varepsilon}{N}]^m)}(M^{-1}L\mathbf{x} + M^{-1} \mathbf{v}/N) \, d\mathbf{x}. \nonumber 
\end{align} The first $m$ columns of $M^{-1} L$ form the identity matrix, and so this expression is equal to
\begin{equation}
\label{second expression of appendix lemma}
N^d \int\limits_{\mathbf{x_{m+1}^d} \in [0,1]^{d-m}} \vol((M^{-1}([-\frac{\varepsilon}{N},\frac{\varepsilon}{N}]^m) - \sum\limits_{j=m+1}^d x_j \mathbf{a^{(j)}} - M^{-1} \mathbf{v}/N )\cap [0,1]^m) \, d\mathbf{x_{m+1}^d},
\end{equation}
\noindent where $\mathbf{a^{(j)}} \in \mathbb{R}^m$ is the $j^{th}$ column of the matrix $M^{-1} L$.

If the vector \[\sum\limits_{j=m+1}^d x_j \mathbf{a^{(j)}} - M^{-1} \mathbf{v}/N\] lies in $[0,1]^m$, then unless it lies close to the boundary of $[0,1]^m$ we have \[\vol((M^{-1}([-\frac{\varepsilon}{N},\frac{\varepsilon}{N}]^m) - \sum\limits_{j=m+1}^d x_j \mathbf{a^{(j)}} - M^{-1} \mathbf{v}/N )\cap [0,1]^m) = N^{-m} \varepsilon^m \vol(M^{-1}([-1,1]^m)).\] More precisely, letting $C_1$ be a constant that is sufficiently large in terms of $L$, we have that expression (\ref{second expression of appendix lemma}) is equal to  
\begin{equation}
\label{main term of integral} N^{d-m}\varepsilon^m \vol (M^{-1}([-1,1]^m))\int\limits_{\substack{\mathbf{x_{m+1}^d} \in [0,1]^{d-m} \\ -\sum\limits_{j=m+1}^d x_j \mathbf{a^{(j)}} - M^{-1}\mathbf{v}/N \in [0,1]^m}} 1 \, d\mathbf{x_{m+1}^d},
\end{equation}
plus an error term of size at most 
\begin{equation}
\label{error term of integral}
\ll_L N^{d-m}\varepsilon^m  \int\limits^* 1 \, d\mathbf{x_{m+1}^d},\
\end{equation} where $\int^*$ indicates integration over those $\mathbf{x_{m+1}^d} \in [0,1]^{d-m}$ for which \[\dist(-\sum\limits_{j=m+1}^d x_j \mathbf{a^{(j)}} - M^{-1}\mathbf{v}/N, \partial ([0,1]^m)) \leqslant C_1 \varepsilon/N.\] We remind the reader that $\partial$ refers to the topological boundary. 

Define \[ C_{L,\mathbf{v}/N} : =  \vol (M^{-1}([-1,1]^m))\int\limits_{\substack{\mathbf{x_{m+1}^d} \in [0,1]^{d-m} \\ -\sum\limits_{j=m+1}^d x_j \mathbf{a^{(j)}} - M^{-1}\mathbf{v}/N \in [0,1]^m}} 1 \, d\mathbf{x_{m+1}^d}\] and 
\begin{equation}
\label{CL}
C_L =  \vol (M^{-1}([-1,1]^m))\int\limits_{\substack{\mathbf{x_{m+1}^d} \in [0,1]^{d-m} \\ -\sum\limits_{j=m+1}^d x_j \mathbf{a^{(j)}} \in [0,1]^m}} 1 \, d\mathbf{x_{m+1}^d}.
\end{equation}

Then certainly $\vert C_{L,\mathbf{v}/N}\vert = O_{L}(1)$. To prove the first part of the lemma it then suffices to control the error term (\ref{error term of integral}). Let $\Phi:\mathbb{R}^{d-m} \longrightarrow \mathbb{R}^m$ denote the map \[\mathbf{x_{m+1}^{d}} \mapsto \sum\limits_{j=m+1}^d x_j \mathbf{a^{(j)}}.\] For all $i\leqslant m$, let $\pi_i:\mathbb{R}^{m} \longrightarrow \mathbb{R}$ denote projection onto the $i^{th}$ coordinate. Then the size of the error term (\ref{error term of integral}) is \[ \ll_L N^{d-m} \varepsilon^m \sup\limits_{\substack{i\leqslant m \\ I \subset \mathbb{R}: \vert I\vert \leqslant 2C_1 \varepsilon/N}}\Big\vert\int\limits_{\mathbf{x_{m+1}^{d}} \in [0,1]^{d-m}} 1_I(\pi_i(\Phi(\mathbf{x_{m+1}^{d}})) \, d \mathbf{x_{m+1}^{d}}\Big\vert,\] where the supremum is over intervals $I$.  

Since $L$ is purely irrational (and so $M^{-1} L$ is also purely irrational), for all $i\leqslant m$ the linear map $\pi_i \circ \Phi:\mathbb{R}^{d-m} \longrightarrow \mathbb{R}$ is non-zero.  From this we conclude that (\ref{error term of integral}) has size at most \[ \ll_{L} N^{d-m-1} \varepsilon^{m+1},\] from which the first part of the lemma follows. \\

For the second part, assume that $\Vert \mathbf{v}\Vert_\infty = o(N)$. Then note that
\begin{equation}
\label{error term from second integral}
\vert C_{L,\mathbf{v}/N} - C_L\vert \ll_L \int^{\ast\ast} 1 \, d\mathbf{x_{m+1}^d},
\end{equation} where $\int^{**}$ indicates integration over those $\mathbf{x_{m+1}^d} \in [0,1]^{d-m}$ for which \[ \dist(- \sum\limits_{j=m+1}^d x_j \mathbf{a^{(j)}}, \partial([0,1]^m)) = o_{L}(1).\] One can estimate (\ref{error term from second integral}) by exactly the same procedure as was used to estimate (\ref{error term of integral}), and thereby conclude that (\ref{error term from second integral}) is $o_{L}(1)$. This settles the second part of the lemma.
\end{proof}
The next lemma concerns the global factor when one of the variables is restricted to a short interval. 
\begin{Lemma}
\label{Lemma crude upper bound lemma}
Let $N,m,d$ be natural numbers, with $d\geqslant m+1$, and let $\varepsilon,\delta$ be positive parameters. Let $L:\mathbb{R}^{d} \longrightarrow \mathbb{R}^m$ be a surjective linear map, and assume that for all $d$ of the  coordinate subspaces\footnote{A coordinate subspace of $\mathbb{R}^d$ is a subspace generated by a subset of the standard basis vectors.} $U \leqslant \mathbb{R}^d$ of dimension $d-1$, the map $L|_U: U \longrightarrow \mathbb{R}^m$ is also surjective. Let $\mathbf{v} \in \mathbb{R}^m$ be any vector, and let $I \subset \mathbb{R}$ be an interval of length $\delta N$. Fix a coordinate $j \in [d]$. Then \[ \Big\vert\int\limits_{\substack{\mathbf{x} \in [-N,N]^d \\ x_j \in I}} 1_{[-\varepsilon,\varepsilon]^m} (L\mathbf{x} + \mathbf{v}) \, d\mathbf{x}\Big\vert \ll_{L} \delta  \varepsilon^m N^{d-m}.\]
\end{Lemma}
\begin{proof}
By the assumptions of the surjectivity of the restrictions of $L$, without loss of generality we may assume that $j = d$ and that the first $m$ columns of $L$ form an invertible matrix $M$. Then, by integrating over $x_{m+1},\dots,x_{d}$, one has 
 \[ \int\limits_{\substack{\mathbf{x} \in [-N,N]^d \\ x_j \in I}} 1_{[-\varepsilon,\varepsilon]^m} (L\mathbf{x} + \mathbf{v}) \, d\mathbf{x} \ll \delta N^{d-m} \sup\limits_{\mathbf{a} \in \mathbb{R}^m} \int\limits_{\mathbf{x} \in \mathbb{R}^m} 1_{[-\varepsilon,\varepsilon]^m}(M\mathbf{x} + \mathbf{a}) \, d\mathbf{x} \ll_L \delta  \varepsilon^m N^{d-m},\]since $M$ is invertible. 
\end{proof}
The final lemma of this section details what occurs when one permutes the parameters in the global factor. 
\begin{Lemma}
\label{Lemma moving error terms}
Let $N$, $m$, $d$ be natural numbers, with $d\geqslant m+1$, and let $\delta,\varepsilon$ be positive parameters. Assume further that $\delta < 1/2$. Let $L: \mathbb{R}^d \longrightarrow \mathbb{R}^m$ be a surjective purely irrational linear map, and assume that for all $d$ of the  coordinate subspaces $U \leqslant \mathbb{R}^d$ of dimension $d-1$, the map $L|_U: U \longrightarrow \mathbb{R}^m$ is also surjective. Let $\mathbf{v} \in \mathbb{R}^m$ be any vector. Then 
\begin{align}
\label{asymptotic for singular integral the second}
&\int\limits_{\mathbf{x} \in [\pm\delta N,(1\mp \delta) N]^d} 1_{[-\varepsilon(1\pm\delta),\varepsilon(1\mp\delta)]^m}(L\mathbf{x} + \mathbf{v}) \, d\mathbf{x} \nonumber \\
& = \int\limits_{\mathbf{x} \in [0,N]^d} 1_{[-\varepsilon,\varepsilon]^m}(L\mathbf{x} + \mathbf{v}) \, d\mathbf{x} +O_{L}(\delta \varepsilon^mN^{d-m}) 
+ O_{L}(\varepsilon^{m+1}N^{d-m-1}).
\end{align}

\end{Lemma}
The proof of this lemma is very similar to the proof of Lemma \ref{Lemma singular integral}. We merely sketch the relevant changes. 
\begin{proof}
By Lemma \ref{Lemma crude upper bound lemma} one may replace the left-hand side of (\ref{asymptotic for singular integral the second}) by \[ \int\limits_{ \mathbf{x} \in [0,N]^d} 1_{[-\varepsilon( 1\pm \delta),\varepsilon(1\mp\delta)]^m}(L\mathbf{x} + \mathbf{v}) \, d\mathbf{x}.\] Let $C_1$ be a suitably large constant (depending on $L$). Following the procedure in the proof of Lemma \ref{Lemma singular integral}, and with the same notation for $M$ and $\mathbf{a^{(j)}}$, one establishes that the left-hand side of (\ref{asymptotic for singular integral the second}) is equal to 
\begin{equation}
\label{second main}
N^{d-m} \varepsilon^m(1+ O(\delta)) \vol (M^{-1}([-1,1]^m)) \int\limits_{ \substack{\mathbf{x_{m+1}^d} \in [0,1]^{d-m} \\ -\sum\limits_{j={m+1}}^d x_j \mathbf{a^{(j)}} - M^{-1} \mathbf{v}/N \in [0,1]^m}} 1 \, d\mathbf{x_{m+1}^d},
\end{equation} plus an error of size at most 
\begin{equation}
\label{second error}
\ll_L N^{d-m} \varepsilon^m \int^\ast 1 \, d\mathbf{x_{m+1}^d},
\end{equation}
\noindent where $\int^*$ indicates integration over those $\mathbf{x_{m+1}^d} \in [-1,2]^{d-m}$ for which \[\dist (-\sum\limits_{j=m+1}^d x_j \mathbf{a^{(j)}} - M^{-1} \mathbf{v}/N,\partial([0,1]^m)) \leqslant C_1 (\delta + \varepsilon/N).\]

The error (\ref{second error}) may be bounded above by $O_{L}((\delta + \varepsilon/N) \varepsilon^m N^{d-m}))$, by the same method as we used to bound (\ref{error term of integral}). 

The main term (\ref{second main}), from the work in Lemma \ref{Lemma singular integral}, is equal to \[ (1+ O(\delta)) \Big(\int\limits_{\mathbf{x} \in [0,N]^d} 1_{[-\varepsilon,\varepsilon]^m}(L\mathbf{x} + \mathbf{v}) \, d\mathbf{x} + O_{L}(\varepsilon^{m+1} N^{d-m-1})\Big).\] Bounding this integral using Lemma \ref{Lemma general upper bound}, the present lemma follows. 
\end{proof}

\begin{Remark}
\emph{In the proofs above, we used, in a critical way, the fact that the convex domains $[-\varepsilon,\varepsilon]^m$ and $[0,N]^d$ are axis-parallel boxes.}
\end{Remark}

\section{An analytic argument}
\label{section an analytic argument}

We take the opportunity to record a rather more direct argument which yields an asymptotic formula for expressions of the form\[ \sum\limits_{\mathbf{n} \in [N]^d } \Big( \prod\limits_{j=1}^d \Lambda^\prime(n_j)\Big) 1_{[-\varepsilon,\varepsilon]^m}(L\mathbf{x}+\mathbf{v})\] in the case when $L:\mathbb{R}^{d}\longrightarrow \mathbb{R}^m$ is a linear map with $d$ at least $2m+1$ (and certain irrationality conditions hold). This method is a simple elaboration on Parsell's ideas \cite{Pa02}, and can handle more general coefficients than Theorem \ref{Main theorem simpler version} (although it requires more variables, of course). We suspect that this result has been obvious to the experts for fifteen years or more, but we feel that it should appear explicitly in the literature. 
\begin{Theorem}
\label{Theorem Parsell generalised}
Let $N,m,d$ be natural numbers, with $d\geqslant 2m+1$, and let $\varepsilon$ be a positive parameter. Let $L:\mathbb{R}^d \longrightarrow \mathbb{R}^m$ be a surjective linear map. Assume further that, when written as matrix with respect to the standard bases, all the $m$-by-$m$ sub-matrices of $L$ have non-zero determinant. Assume also that there does not exist a vector $\ba \in \mathbb{R}^m\setminus \{\mathbf{0}\}$ such that $L^T \ba \in \mathbb{Z}^d$ (i.e. in the language of Definition \ref{Definition rational space} assume that $L$ is purely irrational). Then, for all vectors $\mathbf{v} \in \mathbb{R}^m$, 
\begin{equation}
\label{Parsell general asymptotic}
\sum\limits_{\mathbf{n} \in [N]^d } \Big( \prod\limits_{j=1}^d \Lambda^\prime(n_j)\Big) 1_{[-\varepsilon,\varepsilon]^m}(L\mathbf{x} + \mathbf{v}) = \int\limits_{\mathbf{x} \in [0,N]^d} 1_{[-\varepsilon,\varepsilon]^m}(L\mathbf{x} + \mathbf{v}) \, d\mathbf{x} + o_{L,\varepsilon}(N^{d-m}).
\end{equation}
\end{Theorem}
\begin{Remark}
\emph{The asymptotic size of the main term may easily be established using Lemma \ref{Lemma singular integral}.}
\end{Remark}
\begin{proof}[Sketch proof]
We sketch the argument, referring heavily to estimates from \cite{Fr02} and \cite{Pa02}. Define $$f(\theta) = \sum\limits_{n\leqslant N} \Lambda^\prime(n) e(\theta n).$$ Let $T(N)$ be a function that tends to infinity as $N$ tends to infinity, to be defined later, and let $\delta = \delta(N)$ be a function (depending on the function $T$) that tends to zero suitably slowly as $N$ tends to infinity.  

By Lemma \ref{Lemma smooth approximations}, there exists smooth functions $G_{\pm}:\mathbb{R}^m \longrightarrow [0,1]$ for which $G_{\pm} \in \mathcal{C}(\delta)$ and \[ 1_{[-\varepsilon(1 - \delta),\varepsilon(1-\delta)]} \leqslant G_{-} \leqslant 1_{[-\varepsilon,\varepsilon]}\leqslant G_{+} \leqslant 1_{[-\varepsilon(1 + \delta),\varepsilon(1+\delta)]}.\] By Fourier inversion we see that 
\begin{align}
\int\limits_{\ba \in\mathbb{R}^m} \left(\prod\limits_{j=1}^d f(\sum\limits_{i=1}^m \alpha_i \lambda_{ij})\right) \Big(\prod\limits_{i=1}^m \widehat{G_-}(\alpha_i)\Big) e(\ba \cdot \mathbf{v})\, d\ba \leqslant  \sum\limits_{\mathbf{n} \in [N]^d } \Big( \prod\limits_{j=1}^d \Lambda^\prime(n_j)\Big) 1_{[-\varepsilon,\varepsilon]^m}(L\mathbf{x})\nonumber\\  \leqslant  \int\limits_{\ba \in\mathbb{R}^m} \left(\prod\limits_{i=1}^d f(\sum\limits_{i=1}^m \alpha_i \lambda_{ij})\right) \Big(\prod\limits_{i=1}^m \widehat{G_+}(\alpha_i)\Big) e(\ba \cdot \mathbf{v}) \, d\ba,
\end{align}
\noindent where $L = (\lambda_{ij})_{i\leqslant m, \,j\leqslant d}$. We estimate the integrals by splitting the range of integration in three regions.
\begin{equation*}
\begin{cases}
\mathfrak{t}:= \{\ba:\Vert \ba\Vert_\infty \geqslant T(N)\} &\text{   trivial arc   }\\
\mathfrak{m}:= \{\ba:\Vert \ba\Vert_\infty \in [N^{-1}\log ^B N, T(N))\} &\text{   minor arc   }\\
\mathfrak{M}:= \{\ba:\Vert \ba\Vert_\infty\leqslant N^{-1} \log ^B N\} &\text{   major arc   }\end{cases}
\end{equation*}
\noindent for some large constant $B$. See Section \ref{section Inequalities in arithmetic progressions} for another instance of this technique. 

Much of the estimation relies on the following tight mean value bound.

\begin{Lemma}
\label{Lemma mean value}
Let $U\subset \mathbb{R}^m$ be a domain. Let $d>2m$ and let $l$ be a positive real number satisfying \[ \left(\begin{matrix} d \\ m \end{matrix} \right)l >  2\left(\begin{matrix} d-1 \\ m-1 \end{matrix} \right).\] Then
\begin{align}
\label{mean value}
\int\limits_{\ba \in U} \Big\vert\prod\limits_{j=1}^d f(\sum\limits_{i=1}^m \alpha_i \lambda_{ij}) \Big\vert^{l} \, d\ba \ll_{L} (\diam U + 1)^{O(1)}N^{dl-m}
\end{align}
\end{Lemma}
\begin{proof}
We write the left-hand side of (\ref{mean value}) as
\begin{align*}
\int\limits_{\ba \in U}\prod\limits_{\substack{S \subset [d] \\ \vert S \vert = m}} \Big\vert\prod\limits_{j\in S} f(\sum\limits_{i=1}^m \alpha_i \lambda_{ij})\Big\vert ^{l/\left(\begin{smallmatrix} d-1  \\ m-1\end{smallmatrix} \right)} \, d\ba 
\end{align*} 
which is
\begin{equation}
\label{expression after holder}
\ll \prod\limits_{\substack{S \subset [d] \\ \vert S \vert = m}} \Big(\int\limits_{\ba \in U} \Big\vert\prod\limits_{j\in S} f(\sum\limits_{i=1}^m \alpha_i \lambda_{ij})\Big\vert ^{\left(\begin{smallmatrix} d  \\ m\end{smallmatrix} \right)l/\left(\begin{smallmatrix} d-1  \\ m-1\end{smallmatrix} \right)} \, d\ba \Big)^{1/\left(\begin{smallmatrix} d  \\ m\end{smallmatrix} \right)}
\end{equation}
\noindent by H\"{o}lder's inequality. Note that $dl/m = \left(\begin{smallmatrix} d  \\ m\end{smallmatrix} \right)l/\left(\begin{smallmatrix} d-1  \\ m-1\end{smallmatrix} \right) > 2$. 
 
Now recall the bound from \cite[Lemma 3]{Pa02}, namely \begin{equation*}
\label{tight mean value}
\int\limits_{0}^{1} \vert f(\alpha)\vert^u  \, d\alpha \ll N^{u-1} \, \qquad \text{if } u > 2.
\end{equation*} Note that for each fixed $S$ the $m$-by-$m$ submatrix of $L$ given by $(\lambda_{i,j})_{i\leqslant m, \,j\in S}$ is invertible. Therefore, by applying an invertible change of variables and splitting $U$ into boxes,  (\ref{expression after holder}) is \[ \ll_{L} (\diam U + 1)^{O(1)} N^{ (dl/m - 1)m} = (\diam U + 1)^{O(1)} N^{dl-m}.\] This implies the lemma. 
\end{proof}
\textbf{Trivial arc:} The estimation on the trivial arc proceeds very similarly to page 8 of \cite{Pa02}. Indeed, by the bound in Lemma \ref{Lemma by parts} we have
\begin{align*}
\Big\vert\int\limits_{\ba \in\mathfrak{t}} \Big(\prod\limits_{j=1}^d f(\sum\limits_{i=1}^m \alpha_i \lambda_{ij})\Big)\Big(\prod\limits_{i=1}^m \widehat{G_{\pm}}(\alpha_i)\Big) e(\ba \cdot \mathbf{v})  \, d\ba \Big\vert  & \ll_\delta
\int\limits_{\ba \in\mathfrak{t}} \Big\vert\prod\limits_{j=1}^d f(\sum\limits_{i=1}^m \alpha_i \lambda_{ij})\Big\vert\Vert \ba\Vert_\infty ^{-m-1} \, d\ba \\
\end{align*}
\noindent which is
\begin{align*}
&\ll_\delta \sum\limits_{n\geqslant T(N)} n^{-m-1}  \int\limits_{\ba: n\leqslant \Vert \ba\Vert_\infty \leqslant n+1} \Big\vert\prod\limits_{j=1}^d f(\sum\limits_{i=1}^m \alpha_i \lambda_{ij})\Big\vert \, d\ba \\
& \ll_\delta \sum\limits_{n\geqslant T(N)} n^{-2} N^{d-m} \\
&\ll_\delta T(N)^{-1} N^{d-m} \\
&= o_L(N^{d-m})
\end{align*} by Lemma \ref{Lemma mean value}, provided $\delta(N)$ decays slowly enough. \\

\textbf{Minor arc:} The following is the natural higher-dimensional version of the argument in \cite{Pa02}.

\begin{Lemma}
\label{parsell minor arc estimate} For any positive $A$ and $B$, 
\begin{equation}
\label{parsell minor arc estimate equation}
\sup\limits_{\frac{\log ^B N}{N}\leqslant \Vert \ba\Vert_\infty\leqslant A} \Big\vert\prod\limits_{j=1}^d f(\sum\limits_{i=1}^m \alpha_i \lambda_{ij})\Big\vert = o_{A,B,L}(N^{d}).
\end{equation}
\end{Lemma}
\begin{proof}
Note that by the prime number theorem one has the trivial bound $\vert f(\alpha)\vert \ll N$. Assuming for contradiction that the lemma is false, there exists some $A$, some $B$, and some positive $\varepsilon$ such that, for infinitely many $N$, there exists a vector $\ba^{(N)} \in \mathbb{R}^m$ satisfying $\frac{\log ^B N}{N}\leqslant \Vert \ba^{(N)}\Vert_\infty\leqslant A$ and \[ \Big\vert f(\sum\limits_{i=1}^m \alpha_i^{(N)} \lambda_{ij})\Big\vert \geqslant \varepsilon N \qquad \forall j\in [d].\] Then by \cite[Lemma 1]{Pa02} it follows that for each such $N$, and for all $j\in [d]$, there exist integers $q_j^{(N)}$ and $a_j^{(N)}$ such that $1\leqslant q_j^{(N)}\ll_\varepsilon 1$ and \[ q_j^{(N)}(\sum\limits_{i=1}^m \alpha_i^{(N)} \lambda_{ij}) = a_j^{(N)} + \theta_j^{(N)}, \qquad \forall j\in [d],\] where $\vert \theta_j^{(N)}\vert \ll_\varepsilon N^{-1}$. We observe that, if $N$ is large enough, we have the bound $a_j^{(N)} \ll_{A,L,\varepsilon} 1$. Since $A$ and $\varepsilon$ are fixed we may (by taking a subsequence of $N$) assume that both $q_j^{(N)}$ and $a_j^{(N)}$ are independent of $N$. We call these integers $q_j$ and $a_j$ respectively. 

Now, suppose that $a_j = 0$ for all $j$. Then \[ \Big\vert\sum\limits_{i=1}^m \alpha_i^{(N)} \lambda_{ij}\Big\vert \ll_\varepsilon N^{-1}, \qquad \forall j\in[ d].\] Since the map $L^T:\mathbb{R}^m \longrightarrow \mathbb{R}^d$ is injective, there exists an inverse linear map \[M: (\im L^T, \Vert \cdot\Vert_\infty)\longrightarrow (\mathbb{R}^m,\Vert \cdot\Vert_\infty),\] which must necessarily be bounded. Hence $\Vert \ba^{(N)}\Vert_\infty \ll_\varepsilon N^{-1}$, which is a contradiction for large enough $N$, since $\Vert \ba^{(N)}\Vert_\infty\gg (\log^BN) / N$. Therefore there exists some $j$ for which $a_j\neq 0$. 

Finally, the sequence $\alpha_i^{(N)}$ is contained in a compact domain, and so it must have a convergent subsequence with limit $\alpha_i$, say. Taking this limit, we observe that \[ q_j \sum\limits_{i=1}^m \alpha_i \lambda_{ij} = a_j\] and so \[ \sum\limits_{i=1}^m (\alpha_i \prod\limits_{k=1}^d q_k) \lambda_{ij} = a_j \prod\limits_{\substack{k\leqslant d \\ k\neq j}} q_k.\]  Hence there exists a vector $\bb \in \mathbb{R}^m \setminus \{\mathbf{0} \}$ such that $L^T \bb \in \mathbb{Z}^d$, contradicting the assumptions of Theorem \ref{Theorem Parsell generalised}. This proves the lemma. 
\end{proof}

In the usual fashion, one may use Lemma \ref{parsell minor arc estimate} to deduce that there is some slowly growing function $T(N)$ such that $T(N)\rightarrow \infty$ as $N\rightarrow \infty$ such that \[\sup\limits_{\frac{\log ^B N}{N}\leqslant \Vert \ba\Vert_\infty\leqslant T(N)} \Big\vert\prod\limits_{j=1}^d f(\sum\limits_{i=1}^m \alpha_i \lambda_{ij})\Big\vert = o_{B,L}(N^{d}),\] the details being given in Section 3 of \cite{Pa02}. Defining the minor arc $\mathfrak{m}$ using this function $T(N)$, we have exactly \[\sup\limits_{\ba \in \mathfrak{m}} \Big\vert\prod\limits_{j=1}^d f(\sum\limits_{i=1}^m \alpha_i \lambda_{ij})\Big\vert = o_{B,L}(N^{d}).\] Therefore, picking some positive  parameter $\eta$ that is small enough such that taking $l=1-\eta$ satisfies the hypotheses of Lemma \ref{Lemma mean value},
\begin{align}
\Big\vert\int\limits_{\ba \in\mathfrak{m}} \Big(\prod\limits_{j=1}^d f(\sum\limits_{i=1}^m \alpha_i \lambda_{ij})\Big)\Big(\prod\limits_{i=1}^m \widehat{G_{\pm}}(\alpha_i)\Big) e(\ba \cdot \mathbf{v})  \, d\ba \Big\vert \nonumber
\end{align}
\noindent is
\begin{align}
 & \ll \sup\limits_{\ba \in \mathfrak{m}} \Big\vert\prod\limits_{j=1}^d f(\sum\limits_{i=1}^m \alpha_i \lambda_{ij})\Big\vert^\eta \int\limits_{\ba \in \mathfrak{m}} \Big\vert \prod\limits_{j=1}^{d} f(\sum\limits_{i=1}^m \alpha_i \lambda_{ij}) \Big\vert^{1-\eta} \, d\ba \nonumber \\
  &\ll o_{B,L}(N^{\eta d}) T(N)^{O(1)} N^{d(1-\eta) - m} \nonumber \\
  &\ll o_{B,L}(N^{d-m})
\end{align}
\noindent if $T(N)$ grows slowly enough. \\

\textbf{Major arc:} The analysis of the contribution from the major arc is routine, given the lemmas we established in Appendix \ref{section Easy calculations}. Let $c$ be a small positive constant whose exact value may change between each line. By the estimate (7) from \cite{Pa02} one has, for $\ba \in \mathfrak{M}$, \[ f(\sum\limits_{i=1}^m \alpha_i \lambda_{ij}) = v(\sum\limits_{i=1}^m \alpha_i \lambda_{ij}) + O(N \exp( - c\sqrt{\log N})),\] where \[v(\beta) = \int\limits_{0}^N e(\beta x) \, dx.\] Since the measure of $\mathfrak{M}$ is $O((\log ^{mB} N) N^{-m})$, we have 
\begin{align*}
&\int\limits_{\ba \in\mathfrak{M}} \Big(\prod\limits_{j=1}^d f(\sum\limits_{i=1}^m \alpha_i \lambda_{ij})\Big)\widehat{G_{\pm}(\ba)} e(\ba \cdot \mathbf{v})  \, d\ba \\=  &\int\limits_{\ba \in\mathfrak{M}} \Big(\prod\limits_{j=1}^d v(\sum\limits_{i=1}^m \alpha_i \lambda_{ij})\Big)\widehat{G_{\pm}(\ba)} e(\ba \cdot \mathbf{v}) \, d\ba + O_B(N^{d-m} \exp(-c\sqrt{\log N})).
\end{align*}

Since \[ \vert v(\beta) \vert \ll N(1+ \vert \beta\vert N)^{-1},\] we may extend the above integral to all of $\mathbb{R}^m$ at the cost of an error of $O(N^{d-m} \log^{-B} N)$. In other the words, the contribution from the major arcs is\[\int\limits_{\mathbf{x} \in [0,N]^d} \int\limits_{\ba \in \mathbb{R}^m} \Big(\prod\limits_{j=1}^d e(\sum\limits_{i=1}^m \alpha_i \lambda_{ij} x_j)\Big) \widehat{G_{\pm}(\ba)} e(\ba \cdot \mathbf{v}) \, d\ba\, d\mathbf{x} + O_B(N^{d-m} \log^{-B}N),\] which is \[\int\limits_{\mathbf{x} \in [0,N]^d} G_{\pm}(L\mathbf{x} + \mathbf{v}) \, d\mathbf{x} + O_B(N^{d-m} \log^{-B}N).\] Fixing a large value of $B$, since $\delta = o(1)$ this expression is equal to \[ \int\limits_{ \mathbf{x} \in [0,N]^d} 1_{[-\varepsilon,\varepsilon]^m}(L\mathbf{x} + \mathbf{v}) \, d\mathbf{x} + o_{L,\varepsilon}(N^{d-m}),\] by Lemma \ref{Lemma moving error terms}. This completes the theorem.  
\end{proof}

\bibliographystyle{plain}
\bibliography{Linearinequalitiesinprimes}

\begin{thebibliography}{10}

\bibitem{Ba67}
A.~Baker.
\newblock On some diophantine inequalities involving primes.
\newblock {\em J. Reine Angew. Math.}, 228:166--181, 1967.

\bibitem{Ba92}
Antal Balog.
\newblock Linear equations in primes.
\newblock {\em Mathematika}, 39(2):367--378, 1992.

\bibitem{B17}
P.-Y. Bienvenu.
\newblock A higher-dimensional {S}iegel-{W}alfisz theorem.
\newblock {\em Acta Arith.}, 179(1):79--100, 2017.

\bibitem{BFI87}
E.~Bombieri, J.~B. Friedlander, and H.~Iwaniec.
\newblock Primes in arithmetic progressions to large moduli. {II}.
\newblock {\em Math. Ann.}, 277(3):361--393, 1987.

\bibitem{DaHe46}
H.~Davenport and H.~Heilbronn.
\newblock On indefinite quadratic forms in five variables.
\newblock {\em J. London Math. Soc.}, 21:185--193, 1946.

\bibitem{Fr00}
D.~E. Freeman.
\newblock Asymptotic lower bounds for {D}iophantine inequalities.
\newblock {\em Mathematika}, 47(1-2):127--159, 2000.

\bibitem{Fr02}
D.~E. Freeman.
\newblock Asymptotic lower bounds and formulas for {D}iophantine inequalities.
\newblock In {\em Number theory for the millennium, {II} ({U}rbana, {IL},
  2000)}, pages 57--74. A K Peters, Natick, MA, 2002.

\bibitem{Gr08}
L.~Grafakos.
\newblock {\em Classical {F}ourier analysis}, volume 249 of {\em Graduate Texts
  in Mathematics}.
\newblock Springer, New York, second edition, 2008.

\bibitem{Gr05}
B.~Green.
\newblock Roth's theorem in the primes.
\newblock {\em Ann. of Math. (2)}, 161(3):1609--1636, 2005.

\bibitem{Gr07}
B.~Green.
\newblock Montr\'eal notes on quadratic {F}ourier analysis.
\newblock In {\em Additive combinatorics}, volume~43 of {\em CRM Proc. Lecture
  Notes}, pages 69--102. Amer. Math. Soc., Providence, RI, 2007.

\bibitem{GT08}
B.~Green and T.~Tao.
\newblock The primes contain arbitrarily long arithmetic progressions.
\newblock {\em Annals of Mathematics}, pages 481--547, 2008.

\bibitem{GT08a}
B.~Green and T.~Tao.
\newblock Quadratic uniformity of the {M}\"obius function.
\newblock {\em Ann. Inst. Fourier (Grenoble)}, 58(6):1863--1935, 2008.

\bibitem{GT10}
B.~Green and T.~Tao.
\newblock Linear equations in primes.
\newblock {\em Ann. of Math. (2)}, 171(3):1753--1850, 2010.

\bibitem{GT12}
B.~Green and T.~Tao.
\newblock The {M}\"obius function is strongly orthogonal to nilsequences.
\newblock {\em Ann. of Math. (2)}, 175(2):541--566, 2012.

\bibitem{GTa12}
B.~Green and T.~Tao.
\newblock The quantitative behaviour of polynomial orbits on nilmanifolds.
\newblock {\em Ann. of Math. (2)}, 175(2):465--540, 2012.

\bibitem{GTZ12}
B.~Green, T.~Tao, and T.~Ziegler.
\newblock An inverse theorem for the {G}owers {$U^{s+1}[N]$}-norm.
\newblock {\em Ann. of Math. (2)}, 176(2):1231--1372, 2012.

\bibitem{HB96}
D.~R. Heath-Brown.
\newblock A new form of the circle method, and its application to quadratic
  forms.
\newblock {\em J. Reine Angew. Math.}, 481:149--206, 1996.

\bibitem{Pa02}
S.~T. Parsell.
\newblock Irrational linear forms in prime variables.
\newblock {\em J. Number Theory}, 97(1):144--156, 2002.

\bibitem{TaVu10}
T.~Tao and V.~H. Vu.
\newblock {\em Additive combinatorics}, volume 105 of {\em Cambridge Studies in
  Advanced Mathematics}.
\newblock Cambridge University Press, Cambridge, 2010.

\bibitem{Va74}
R.~C. Vaughan.
\newblock Diophantine approximation by prime numbers. {I}.
\newblock {\em Proc. London Math. Soc. (3)}, 28:373--384, 1974.

\bibitem{Wa17}
A.~Walker.
\newblock Gowers norms control diophantine inequalities.
\newblock arXiv:1703.00885.

\end{thebibliography}
\end{document}